\RequirePackage{fix-cm}
\documentclass{amsart} 
\usepackage[utf8]{inputenc}
\usepackage{mathptmx}
\usepackage{latexsym}
\usepackage{amsmath}
\usepackage{amssymb}
\usepackage{amsfonts} 
\usepackage{eucal} 
\usepackage{amsbsy}
\usepackage{amsthm}
\usepackage[all]{xy}
\usepackage{hyperref}
\usepackage{cleveref}
\usepackage[boxsize=.8em]{ytableau}
\usepackage{graphicx}
\usepackage{framed}
\usepackage{tikz}

\newtheorem{thm}{Theorem}[section]
\newtheorem*{thm*}{Theorem}
\newtheorem{lemma}[thm]{Lemma}
\newtheorem{prop}[thm]{Proposition}
\newtheorem{cor}[thm]{Corollary}
\newtheorem*{cor*}{Corollary}

\theoremstyle{definition}
\newtheorem{defn}[thm]{Definition}
\newtheorem{example}[thm]{Example}

\theoremstyle{remark}
\newtheorem{remark}[thm]{Remark}

\newcommand {\intg}  {\ensuremath{\mathbb{Z}}}

\newcommand {\cplx}  {\ensuremath{\mathbb{C}}}
\newcommand {\rat}   {\ensuremath{\mathbb{Q}}}

\newcommand {\codim} {\ensuremath{\operatorname{codim}}}

\newcommand {\BM}   {{\ensuremath{\operatorname{BM}}}}

\newcommand {\cl}   {{\ensuremath{\operatorname{cl}}}}
\newcommand {\Sing}    {\ensuremath{\operatorname{Sing}}}

\makeatletter
\@namedef{subjclassname@2020}{%
  \textup{2020} Mathematics Subject Classification}
\makeatother

\begin{document}


\title{The Uniqueness Theorem for Gysin Coherent Characteristic Classes of Singular Spaces}

\author{Markus Banagl}

\address{Mathematisches Institut, Universit\"{a}t Heidelberg,
  Im Neuenheimer Feld 205, 69120 Heidelberg, Germany}

\email{banagl@mathi.uni-heidelberg.de}

\author{Dominik J. Wrazidlo}

\address{Mathematisches Institut, Universit\"{a}t Heidelberg,
  Im Neuenheimer Feld 205, 69120 Heidelberg, Germany}

\email{dwrazidlo@mathi.uni-heidelberg.de}

\thanks{This work is funded in part by the Deutsche Forschungsgemeinschaft (DFG, German Research Foundation) through a research grant to the first author (Projektnummer 495696766).
The second author was supported in addition by a Feodor Lynen Return Fellowship of the Alexander von Humboldt Foundation.}

\date{\today}




\begin{abstract}
We establish a general computational scheme designed for a systematic computation of characteristic classes of singular complex algebraic varieties that satisfy a Gysin axiom in a transverse setup.
This scheme is explicitly geometric and of a recursive nature terminating on genera of explicit characteristic subvarieties that we construct.
It enables us e.g. to apply intersection theory of Schubert varieties to obtain a uniqueness result for such characteristic classes in the homology of an ambient Grassmannian.
Our framework applies in particular to the Goresky-MacPherson $L$-class by virtue of the Gysin restriction formula obtained by the first author in previous work.
We illustrate our approach for a systematic computation of the $L$-class in terms of normally nonsingular expansions in examples of singular Schubert varieties that do not satisfy Poincar\'{e} duality over the rationals.
\end{abstract}

\maketitle


\tableofcontents


\section{Introduction}
We establish a general computational scheme that allows for a systematic computation of characteristic classes of singular complex algebraic varieties that satisfy a Gysin axiom.
This scheme is explicitly geometric and of a recursive nature terminating on genera of explicit characteristic subvarieties that we construct.

We introduce and investigate the notion of \emph{Gysin coherent characteristic classes} $c\ell$ defined for inclusions of certain subvarieties in ambient smooth complex algebraic varieties (here, varieties are understood to be pure-dimensional complex and quasiprojective).
In \Cref{definition gysin coherent characteristic classes}, we define such a class to be a pair $c\ell = (c\ell^{\ast}, c\ell_{\ast})$ consisting of a function $c\ell^{\ast}$ that assigns to every inclusion $f \colon M \rightarrow W$ of a smooth closed subvariety $M \subset W$ in a smooth variety $W$ a normalized element $c\ell^{\ast}(f) \in H^{\ast}(M; \mathbb{Q})$, and a function $c\ell_{\ast}$ that assigns to every inclusion $i \colon X \rightarrow W$ of a compact possibly singular subvariety $X \subset W$ in a smooth variety $W$ an element $c\ell_{\ast}(i) \in H_{\ast}(W; \mathbb{Q})$ whose highest non-trivial homogeneous component is the ambient fundamental class of $X$ in $W$ such that the following axioms hold.
Apart from the Gysin restriction formula $f^{!} c\ell_{\ast}(i) = c\ell^{\ast}(f) \cap c\ell_{\ast}(M \pitchfork X \subset M)$ in a transverse setup, we also require that $c\ell_{\ast}$ is multiplicative under products, that $c\ell^{\ast}$ and $c\ell_{\ast}$ transform naturally under isomorphisms of ambient smooth varieties, and that $c\ell_{\ast}$ is natural with respect to inclusions in larger ambient smooth varieties.
The genus $|c\ell_{\ast}|$ associated to a Gysin coherent characteristic class $c\ell$ is defined as the composition of $c \ell_{\ast}$ with the homological augmentation, $|c\ell_{\ast}| = \varepsilon_{\ast} c \ell_{\ast} \in \mathbb{Q}$.

Often, such classes $c \ell$ arise from generalizations $c_{\ast}$ to singular varieties of bundle theoretic cohomological classes $c^{\ast}$.
In such a situation, the pair $c \ell$ is of the form $c \ell^{\ast}(f) = c^{\ast}(\nu_{f})$, where $\nu_{f}$ is the normal bundle of the smooth embedding $f$, and $c \ell_{\ast}(i) = i_{\ast} c_{\ast}(X)$ for inclusions $i$ of compact possibly singular subvarieties $X$ in ambient smooth varieties.
By virtue of the Verdier-Riemann-Roch type formulae derived by the first author in \cite{banagllgysin}, our framework applies in particular to the topological characteristic class $c_{\ast} = L_{\ast}$ of Goresky and MacPherson \cite{gmih1} that generalizes the cohomological Hirzebruch class $c^{\ast} = L^{\ast}$ \cite{hirzebruch} to singular spaces (see \Cref{proposition l class is l type characteristic class}).
In this case, the associated genus is the signature $\sigma(X)$ of the Goresky-MacPherson-Siegel intersection form on middle-perversity intersection homology of the Witt space $X$.

By its very definition, the Goresky-MacPherson $L$-class is uniquely determined by signature normalization in degree zero, and by compatibility with Gysin restriction associated to normally nonsingular topological embeddings of topological singular spaces with \emph{trivial} normal bundle (see Cappell-Shaneson \cite{cs}).
Uniqueness can be shown by a Thom-Pontrjagin type approach via transverse regular maps to spheres.
On the other hand, by dropping the triviality assumption for normal bundles, the Gysin axiom introduced in the present paper seems especially well-suited for concrete computations in transverse situations within the realm of complex algebraic geometry.
However, the Thom-Pontrjagin method is usually not directly applicable when ranging only over algebraic varieties rather than all topological Witt spaces.
(Regular level sets of PL representatives of homotopy classes of maps from a variety to a sphere are not subvarieties in general.)
Our main result is the following uniqueness theorem for Gysin coherent characteristic classes of singular varieties embedded in Grassmannians.

\begin{thm}[Uniqueness Theorem]\label{main result on Gysin coherent characteristic classes}
Let $c\ell$ and $\widetilde{c\ell}$ be Gysin coherent characteristic classes.
If $c\ell^{\ast} = \widetilde{c\ell}^{\ast}$ and $|c\ell_{\ast}| = |\widetilde{c\ell}_{\ast}|$ for the associated genera, then we have $c\ell_{\ast}(i) = \widetilde{c\ell}_{\ast}(i)$ for all inclusions $i \colon X \rightarrow G$ of compact irreducible subvarieties in ambient Grassmannians.
\end{thm}

Since the inclusion of Schubert subvarieties in Grassmannians induces an injective map on homology with rational coefficients, we obtain

\begin{cor}\label{corollary Gysin coherent characteristic classes of Schubert varieties}
Let $c_{\ast}$ and $\widetilde{c}_{\ast}$ be generalizations to singular varieties of a bundle theoretic cohomological class $c^{\ast}$ such that $|c_{\ast}| = |\widetilde{c}_{\ast}|$ for the associated genera.
If $c_{\ast}$ and $\widetilde{c}_{\ast}$ induce Gysin coherent characteristic classes as explained above, then we have $c_{\ast}(X) = \widetilde{c}_{\ast}(X)$ for all Schubert varieties $X$.
\end{cor}

The reader can directly verify our \Cref{main result on Gysin coherent characteristic classes} for the toy example of inclusions into ambient projective spaces by applying the Gysin axiom inductively to intersections of subvarieties with generic hyperplanes.
To provide enough flexibility for applications, we state \Cref{main result on Gysin coherent characteristic classes} in a slightly more general form (see \Cref{main result on Gysin coherent characteristic classes with respect to x}) that accounts for a fixed family $\mathcal{X}$ of admissible inclusions $i \colon X \rightarrow W$ which satisfy an analog of the Kleiman-Bertini transversality theorem for an appropriate notion of transversality.
In \Cref{proof of main theorem}, we derive \Cref{main result on Gysin coherent characteristic classes with respect to x} by induction on the dimension of the ambient Grassmannian from a more technical result (see \Cref{main theorem normally nonsingular expansion}) that exploits the intersection theory of Schubert cycles, and specifically the Segre product of subvarieties of Grassmannians in an ambient Grassmannian.

The additional value of our \Cref{main theorem normally nonsingular expansion} is that it yields a systematic method for the recursive computation of Gysin coherent characteristic classes in ambient Grassmannians in terms of the normal geometry of Schubert varieties.
First examples of these normally nonsingular expansions appered in \cite[Section 4]{banagllgysin}.
There, the first author computed the Goresky-MacPherson $L$-class in (real) codimension $4$ for the Schubert varieties $X_{2, 1}$ and $X_{3, 2}$, which are sufficiently singular so as not to be rational Poincar\'{e} complexes.
Beyond these examples, it seems to be an open problem to compute $L$-classes of singular Schubert varieties.
The Chern-Schwartz-MacPherson classes of Schubert varieties were computed by Aluffi and Mihalcea \cite{aluffi}.
Our recursive formula (\ref{equation main result recuvrsive coefficient formula}) reduces computations to concrete Kronecker products (integrals) that capture the normal geometry of Schubert varieties with L-shaped Young diagrams over triple intersections of Schubert varieties (see \Cref{remark integral expression}), and to genera of explicitly constructed characteristic subvarieties (see \Cref{remark characteristic varieties}).
These characteristic subvarieties are obtained by taking the product of the given embedded variety with a certain Schubert variety, and then intersecting the Segre embedded product in the larger Grassmannian with another appropriate Schubert variety.
If the given variety is a Schubert variety, then our characteristic varieties turn out to be triple intersections of Schubert varieties.
In the literature, intersections of more than two general translates of Schubert varieties were studied by Billey and Coskun \cite{bc} as a generalization of Richardson varieties \cite{rich}.
They employed Kleiman's transversality theorem \cite{kleiman} to determine the singular locus of such intersection varieties.
The explicit computation of the integrals and genera that appear in our recursive formula (\ref{equation main result recuvrsive coefficient formula}) requires separate techniques and is hence not pursued in the present paper.

Despite the general applicability of our recursive technique asserted by \Cref{main theorem normally nonsingular expansion}, it can be more convenient in practice to implement slightly modified algorithms for doing concrete computations.
In \Cref{l class of x 3 2 1}, we illustrate such a related recursive method to compute the Goresky-MacPherson $L$-class for the example of the singular Schubert variety $X_{3, 2, 1}$ of real dimension $12$, which does not satisfy global Poincar\'{e} duality over the rationals.
Note that the computation of $L_{4}(X_{3, 2, 1})$ goes beyond the scope of the computations in \cite[Section 4]{banagllgysin}, which are limited to the $L$-class in real codimension $4$.

In \cite{banagllgysin}, the first author derived Verdier-Riemann-Roch type formulae for the Gysin restriction of both the topological characteristic classes $L_{\ast}$ of Goresky and MacPherson \cite{gmih1} and the Hodge-theoretic intersection Hirzebruch characteristic classes $IT_{1 \ast}$ of Brasselet, Sch\"{u}rmann and Yokura \cite{bsy}.
For an introduction to characteristic classes of singular spaces via mixed Hodge theory in the complex algebraic context see Sch\"{u}rmann's expository paper \cite{schuermannmsri}.
The formulae in \cite{banagllgysin} have the potential to yield new evidence for the equality of the characteristic classes $L_{\ast}$ and $IT_{1 \ast}$ for pure-dimensional compact complex algebraic varieties, as conjectured by Brasselet, Sch\"{u}rmann and Yokura in \cite[Remark 5.4]{bsy}.
Cappell, Maxim, Sch\"{u}rmann and Shaneson proved the conjecture in \cite[Cor. 1.2]{cmss1} for orbit spaces $X = Y/G$, with $Y$ a projective $G$-manifold and $G$ a finite group of algebraic automorphisms.
They also showed the conjecture for certain complex hypersurfaces with isolated singularities \cite[Theorem 4.3]{cmss2}.
The conjecture holds for simplicial projective toric varieties as shown by Maxim and Sch\"{u}rmann \cite[Corollary 1.2(iii)]{ms}.
Furthermore, the conjecture was established by the first author in \cite{banaglcovertransfer} for normal connected complex projective $3$-folds $X$ that have at worst canonical singularities, trivial canonical divisor, and $\operatorname{dim} H^{1}(X; \mathcal{O}_{X}) > 0$.
Generalizing the above cases, Fern\'{a}ndez de Bobadilla and Pallar\'{e}s \cite{fdbp} proved the conjecture for all compact complex algebraic varieties that are rational homology manifolds.
In ongoing work with J\"{o}rg Sch\"{u}rmann, we apply the methods developed in the present paper to prove the ambient version of the conjecture for a certain class of subvarieties in Grassmannians.
This class includes all Schubert varieties.
Since the homology of Schubert varieties injects into the homology of ambient Grassmannians, this would imply the conjecture for all Schubert varieties.
Furthermore, we shall clarify how other algebraic characteristic classes such as Chern classes fit into the framework.

\textbf{Acknowledgement.}
We thank J\"{o}rg Sch\"{u}rmann and Laurentiu Maxim for insightful comments on an earlier version of the paper.

\textbf{Notation.}
(Co)homology groups will be with rational coefficients unless otherwise stated.
Complex algebraic varieties are not assumed to be irreducible.
If $X$ and $Y$ are subvarieties of  a smooth variety $W$, then the symbol $X \cap Y$ denotes the set theoretic intersection of $X$ and $Y$ in $W$.
Given a set $S$, the Kronecker delta of two elements $a, b \in S$ is $\delta_{ab} = 1$ for $a = b$, and $\delta_{ab} = 0$ else.
The Kronecker product $\langle -, -\rangle$ is defined by $\langle \xi, x\rangle = \varepsilon_{\ast}(\xi \cap x) \in \mathbb{Q}$ for classes $\xi \in H^{\ast}(A)$ and $x \in H_{\ast}(A)$ on a topological space $A$, where $\varepsilon_{\ast} \colon H_{\ast}(A) \rightarrow H_{\ast}(\operatorname{pt}) = \mathbb{Q}$ denotes the augmentation map induced by the unique map $A \rightarrow \operatorname{pt}$.

\section{Whitney Transversality}
Let $W$ be a smooth manifold.
Recall that a Whitney stratification of a closed subset $Z \subset W$ is a certain decomposition of $Z$ into locally closed smooth submanifolds of $W$ such that the pieces of this decomposition fit together via Whitney's conditions A and B (for details, see e.g. \cite[Section 1.2, p. 37]{gmsmt}).
Two Whitney stratified subsets $Z_{1}, Z_{2} \subset W$ are called transverse if every stratum of $Z_{1}$ is transverse to every stratum of $Z_{2}$ as smooth submanifolds of $W$.
The transverse intersection $Z_{1} \cap Z_{2}$ then has a canonical Whitney stratification in $W$ whose strata are given by the intersections of the strata of $Z_{1}$ and $Z_{2}$.

Pure-dimensional closed subvarieties of a nonsingular complex algebraic variety can always be Whitney stratified by strata of even codimension.
In the context of this paper, we call two such subvarieties \emph{Whitney transverse} if they can be equipped with transverse Whitney stratifications.
The present section collects various results about Whitney stratified spaces that will be applied to Whitney transverse subvarieties later on in the paper.

\begin{lemma}\label{lemma whitney restratification manifold}
Let $Z_{1}, Z_{2} \subset W$ be Whitney stratified subspaces of a smooth manifold $W$ that are transverse to each other.
If $U \subset W$ is a smooth submanifold that is open as a subset of $Z_{2}$, then every stratum of $Z_{1}$ is transverse to $U$ in $W$.
\end{lemma}

\begin{proof}
Let $z \in Z_{1} \cap U$.
Let $S \subset Z_{1}$ and $S' \subset Z_{2}$ denote the strata containing $z$.
Then, $S$ and $S'$ are transverse at $z$ in $W$.
Note that $U' := W \setminus (Z_{2} \setminus U)$ is an open subset of $W$ because $Z_{2}$ is a closed subset of $W$.
Since $z \in U'$, it follows that $U' \cap S$ and $U' \cap S'$ are transverse at $z$ in $U'$.
Since $U \subset U'$ is a smooth submanifold that contains $U' \cap S'$, it follows that $U' \cap S$ and $U$ are transverse at $z$ in $U'$.
Consequently, $S$ and $U$ are transverse at $z$ in $W$, and the claim follows.
\end{proof}

The following topological version of Kleiman's transversality theorem \cite{kleiman} will be applied for the canonical action of the Lie group $G = GL(V)$ on the Grassmannian $W = G_{k}(V)$ for a complex vector space $V$ of finite dimension.

\begin{thm}\label{thm kleiman for whitney transversality}
(See e.g. \cite[p. 39, Theorem 1.3.6 and Examples 1.3.7]{gmsmt}.)
Let $W$ be a smooth manifold that is homogeneous under the action of a Lie Group $G$.
Given Whitney stratified subspaces $Z_{1}, Z_{2} \subset W$, the set $U$ of all $g \in G$ such that $g \cdot Z_{1}$ and $Z_{2}$ are transverse in $W$ is dense in $G$.
Moreover, if $Z_{1}$ is compact, then $U$ is also open in $G$.
\end{thm}

Recall that an inclusion $g \colon Y \hookrightarrow X$ of topological spaces is called (oriented) normally nonsingular (of codimension $r$) if there is a (oriented) real vector bundle $\nu \colon E \rightarrow Y$ (of rank $r$), a neighborhood $U \subset E$ of the zero section of $\nu$ (which we also denote by $Y$), and a homeomorphism $j \colon U \rightarrow X$ onto an open subset $j(U) \subset X$ such that $g$ factorizes as the composition $Y \xrightarrow{\operatorname{incl}} U \xrightarrow{j} X$ (see e.g. \cite[Section 1.11, p. 46f]{gmsmt}).
We also call $\nu$ a normal bundle of the normally nonsingular inclusion $g$.
For example, transverse intersections give rise to normally nonsingular inclusions as follows.

\begin{thm}\label{theorem normally nonsingular inclusions induced by transverse intersections}
(See Theorem 1.11 in \cite[p. 47]{gmsmt}.)
Let $X \subset W$ be a Whitney stratified subset of a smooth manifold $W$.
Suppose that $M \subset W$ is a smooth submanifold of codimension $r$ that is transverse to every stratum of $X$, and set $Y = M \cap X$.
Then, the inclusion $g \colon Y \hookrightarrow X$ is normally nonsingular of codimension $r$ with respect to the normal bundle $\nu = \nu_{M \subset W}|_{Y}$ given by restriction of the normal bundle $\nu_{M \subset W}$ of $M$ in $W$.
\end{thm}

An oriented normally nonsingular inclusion $g \colon Y \hookrightarrow X$ of a closed subset $Y \subset X$ with normal bundle $\nu \colon E \rightarrow Y$ of rank $r$ induces a Gysin homomorphism
$$
g^{!} \colon H_{\ast}(X; \mathbb{Q}) \rightarrow H_{\ast-r}(Y; \mathbb{Q})
$$
given by the composition
$$
H_{\ast}(X) \xrightarrow{\operatorname{incl_{\ast}}} H_{\ast}(X, X \setminus Y) \xleftarrow[\cong]{e_{\ast}} H_{\ast}(E, E_{0}) \xrightarrow[\cong]{u \cap -} H_{\ast-r}(E) \xrightarrow[\cong]{\nu_{\ast}} H_{\ast-r}(Y),
$$
where $u \in H^{r}(E, E_{0})$ is the Thom class with $E_{0} = E \setminus Y$ the complement of the zero section of $\nu$ in $E$, and $e_{\ast}$ denotes the excision isomorphism induced by the open embedding $j \colon U \rightarrow X$.

Gysin homomorphisms are compatible with pushforward under embeddings as follows.

\begin{prop}[Base Change]\label{proposition gysin compatible with inclusions}
Consider a cartesian square
\[ \xymatrix{
L \ar[r]^{\beta} \ar[d]_{g} & Y \ar[d]^{f} \\
K \ar[r]_{\alpha} & X
} \]
of topological spaces and continuous maps, where $f$ and $g$ are oriented normally nonsingular inclusions of closed subsets with normal bundles $\nu_{f}$ and $\nu_{g}$, respectively, such that $\nu_{g} = \beta^{\ast}\nu_{f}$.
Then,
$$
\beta_{\ast} g^{!} = f^{!} \alpha_{\ast}.
$$
\end{prop}

\begin{proof}
The statement follows from the naturality of Thom classes, using $\nu_{g} = \beta^{\ast}\nu_{f}$.
\end{proof}

Recall that every compact oriented $r$-dimensional pseudomanifold $X$ possesses a fundamental class
$$
[X]_{X} \in H_{r}(X; \mathbb{Z}).
$$
If $X$ is contained in an ambient space $W$, then we write $[X]_{W}$ for the image of $[X]_{X}$ under the map $H_{\ast}(X; \mathbb{Z}) \rightarrow H_{\ast}(W; \mathbb{Z})$ induced by the inclusion $X \hookrightarrow W$.

\begin{prop} \label{lem.gysinfund}
Let $W$ be an oriented smooth manifold,
$X,K\subset W$ Whitney stratified subspaces which are oriented
pseudomanifolds with $K\subset X$ and $K$ compact.
Let $M\subset W$ be an oriented smooth submanifold which is closed as a subset.
Suppose that $M$ is transverse to the 
Whitney strata of $X$ and to the Whitney strata of $K$.
Then the Gysin map
\[ g^!: H_* (X;\rat) \longrightarrow H_{*-r} (Y;\rat) \]
associated to the normally nonsingular embedding $g: Y=M\cap X \hookrightarrow X$,
where $r$ is the (real) codimension of $Y$ in $X$,
sends the fundamental class
$[K]_X \in H_* (X;\rat)$ of $K$ to the fundamental class
$[K\cap Y]_Y$ of the intersection $K\cap Y = M \cap K$ (which is again an oriented pseudomanifold),
\[ g^! [K]_X = [K\cap Y]_Y. \]
\end{prop}
\begin{proof}
Let $\nu_M$ denote the normal bundle of $M$ in $W$.
This is an oriented bundle, since $M$ and $W$ are oriented.
Consider the cartesian square
\[ \xymatrix{
Y \ar[r] \ar[d]_g & M \ar[d] \\
X \ar[r] & W.
} \]
Since $M$ is transverse to the Whitney stratification of $X$, the inclusion
$g: Y\subset X$ is normally nonsingular with oriented normal bundle
\[ \nu_Y = \nu_M|_Y \]
by \Cref{theorem normally nonsingular inclusions induced by transverse intersections}.
The associated Gysin map is
\[ g^!: H_\ast (X) \longrightarrow H_{\ast-r} (Y). \]
Furthermore, $Y$ is an (oriented) pseudomanifold,
since $X$ is.
Set 
\[ L := Y \cap K = M\cap X \cap K = M\cap K,  \]
which is compact because it is by assumption a closed subset of the compact space $K$.
Consider the cartesian square
\[ \xymatrix{
L \ar[r] \ar[d]_{g_K} & M \ar[d] \\
K \ar[r] & W.
} \]
Since $M$ is transverse to the Whitney stratification of $K$, the inclusion
$g_K: L \rightarrow K$ is normally nonsingular with oriented normal bundle
\[ \nu_L = \nu_M|_L. \]
The associated Gysin map is
\[ g^!_K: H_\ast (K) \longrightarrow H_{\ast-r} (L). \]
Furthermore, $L$ is a compact (oriented) pseudomanifold,
since $K$ is. In particular, $K$ and $L$ have fundamental classes
\[ [K]_K \in H_d (K),~ [L]_L \in H_{d-r} (L), \]
where $d$ is the (real) dimension of $K$ and, hence, $d-r$ is the (real)
dimension of $L$.
Consider the cartesian square
\[ \xymatrix{
L \ar[r]^{f_{Y}} \ar[d]_{g_{K}} & Y \ar[d]^{g} \\
K \ar[r]_{f} & X.
} \]
The above isomorphisms of normal bundles yield an isomorphism
\begin{equation} \label{equ.nuwnuyw}
\nu_L = \nu_M|_L = (\nu_M|Y)|_L = \nu_Y|_L = f^*_Y \nu_Y. 
\end{equation}
Hence, base change (see \Cref{proposition gysin compatible with inclusions}) implies
\begin{equation} \label{equ.gysgcommf}
 f_{Y*} g^!_K = g^! f_*. 
\end{equation}
As Gysin restriction maps the fundamental class to the fundamental class, we conclude that
\[
g^! [K]_X
= g^! f_* [K]_K 
= f_{Y*} g^!_K [K]_K 
= f_{Y*} [L]_L 
= [L]_Y.
\]
\end{proof}

\section{Generic Transversality}
In this section, we review the concept of generic transversality.
As an application (see \Cref{lemma intersection is cap product for whitney transverse subvarieties} below), we discuss the relation of the intersection of algebraic subvarieties of an ambient smooth projective variety to the intersection product of their fundamental classes.

In the following, let $X,Y \subset W$ be irreducible closed subvarieties of a nonsingular irreducible complex algebraic variety $W$.

\begin{defn}\label{def.gentransverse}
(See Eisenbud-Harris \cite[p. 18]{eh}.)
We say that $X$ and $Y$ are \emph{generically transverse in $W$} if every irreducible component of $X\cap Y$ contains a point $p$ such that $X$ and $Y$ are both smooth at $p$ and their tangent spaces satisfy $T_{p} X + T_{p} Y = T_{p} W$.
\end{defn}

\begin{remark}\label{rem.gentransverse}
The set of points $p$ of $X \cap Y$ at which both $X$ and $Y$ are smooth and $T_{p} X + T_{p} Y = T_{p} W$ is a Zariski open subset of $X\cap Y$.
Thus, if this set is not empty in some irreducible component of $X\cap Y$, then it is dense in that component.
\end{remark}

\begin{prop}\label{prop.gentransverse}
(See Richardson \cite[Proposition 1.2]{rich}.)
Let $Z$ be an irreducible component of $X \cap Y$, and let $p \in Z$ be a point at which both $X$ and $Y$ are smooth and $T_{p} X + T_{p} Y = T_{p} W$.
Then, $p$ is a smooth point of $Z$, and $T_{p} Z = T_{p} X \cap T_{p} Y$.
\end{prop}

Recall that the codimension of $X$ in $W$ is defined as $\codim_W X = \dim W - \dim X$.

\begin{cor}\label{prop codimension of transverse intersection}
(See also Eisenbud-Harris \cite[Proposition 1.28, p. 33]{eh}.)
If $X,Y$ are generically transverse in $W$, then every irreducible component $Z$ of $X\cap Y$ satisfies
\[ \codim_W Z = \codim_W X + \codim_W Y. \]
In particular, the closed subvariety $X \cap Y \subset W$ is pure-dimensional.
\end{cor}

If $W$ is a homogeneous space under the action of some algebraic group $G$, then any given $X,Y \subset W$ can be moved into generically transverse position by virtue of Kleiman's transversality theorem:

\begin{thm} \label{thm.kleiman}
(Kleiman; see Eisenbud-Harris \cite[p. 20, Theorem 1.7(a)]{eh}.)
Suppose that an algebraic group $G$ acts transitively on a nonsingular irreducible complex algebraic variety $W$.
Let $X,Y \subset W$ be closed irreducible subvarieties.
Then there exists an open dense set of $g\in G$ such that $gX$ is generically
transverse to $Y$ in $W$.
\end{thm}

From now on, let us assume that $W = P$ is projective of dimension $n$.
Since $P$ is a closed, oriented $2n$-dimensional real smooth manifold, there is
a Poincar\'e duality isomorphism
\[ PD: H_i (P) \longrightarrow H^{2n-i} (P), \]
which is inverse to capping with the fundamental class.
Using the cup product
\[ \cup: H^{2n-i} (P) \otimes H^{2n-j} (P) \longrightarrow H^{4n-i-j} (P) \]
on cohomology, Poincar\'e duality induces a homological intersection pairing
\[ \cdot: H_i (P) \otimes H_j (P) \to H_{i+j-2n} (P) \]
given by
\[ a\cdot b = PD^{-1} (PD(a) \cup PD (b)). \]
Moreover, any closed subvariety $Q \subset P$ of pure dimension $d$ is a closed oriented real $2d$-dimensional pseudomanifold and hence has a fundamental class $[Q]_{Q} \in H_{2d}(Q)$.
Recall that we write $[Q]_{P}$ for the image of $[Q]_{Q}$ under the map $H_{\ast}(Q) \rightarrow H_{\ast}(P)$ induced by the inclusion $Q \hookrightarrow P$.

\begin{prop}\label{lemma intersection is cap product for whitney transverse subvarieties}
Let $X,Y \subset P$ be irreducible closed subvarieties of a nonsingular irreducible complex projective algebraic variety $P$.
If $X$ and $Y$ are generically transverse in $P$, then
$$
[X]_{P} \cdot [Y]_{P} = \sum_{Z}[Z]_{P} = [X \cap Y]_{P},
$$
where the sum ranges over the finite set of irreducible components $Z$ of $X \cap Y$.
\end{prop}

\begin{proof}
As for the first equation, we follow Fulton (see \cite[Appendix B, Equation (9), p. 213]{fultonyoung} and \cite[Appendix B.3, pp. 219--222]{fultonyoung}, where we work with rational instead of integral coefficients).
Note that, in Fulton's terminology, the intersection $X \cap Y$ is proper by \Cref{prop codimension of transverse intersection}, and $X$ and $Y$ meet transversely by \Cref{def.gentransverse}, \Cref{rem.gentransverse}, and \Cref{prop.gentransverse}.
The second equation is due to the fact that the fundamental class behaves additively under decomposition into irreducible components.
Recall that the fundamental class of an oriented closed pseudomanifold is the sum the oriented top-dimensional simplices of a triangulation.
\end{proof}

\section{Preliminaries on Schubert Varieties}
In this section, we recall preliminaries on Schubert varieties and fix notation.

\subsection{Partitions}
For integers $m, k \geq 0$, let
$$
\mathcal{P}(m, k) = \{a = (a_{1}, \dots, a_{k}) \in \mathbb{Z}^{k} | \; m \geq a_{1} \geq \dots \geq a_{k} \geq 0\}.
$$
Thus, an element $a \in \mathcal{P}(m, k)$ can be considered as a partition of the nonnegative integer $|a| := a_{1} + \dots + a_{k}$.
Note that we have $\mathcal{P}(m, 0) = \{(\varnothing)\}$, where $(\varnothing)$ denotes the empty partition.
For $a, b \in \mathcal{P}(m, k)$, we will write $b \leq a$ if $b_{i} \leq a_{i}$ for all $i$.
Let $[m \times k] \in \mathcal{P}(m, k)$ denote the partition $[m \times k] = a = (a_{1}, \dots, a_{k})$ given by $m = a_{1} = \dots = a_{k}$.
For $m \leq m'$ and $k \leq k'$, extension by zero yields natural maps $\mathcal{P}(m, k) \rightarrow \mathcal{P}(m', k')$ assigning to $a = (a_{1}, \dots, a_{k}) \in \mathcal{P}(m, k)$ the partition $a' = (a_{1}', \dots, a_{k'}') \in \mathcal{P}(m', k')$ given by $a_{i}' = a_{i}$ for $1 \leq i \leq k$ and $a_{i}' = 0$ for $i > k$.
Let $[m \times k]_{m', k'} \in \mathcal{P}(m', k')$ denote the image of $[m \times k] \in \mathcal{P}(m, k)$ under the natural map $\mathcal{P}(m, k) \rightarrow \mathcal{P}(m', k')$.

We may represent a partition $a = (a_{1}, \dots, a_{k}) \in \mathcal{P}(m, k)$ by an
(upside-down) Young diagram $D_{a}$. It consists of left-justified vertically stacked rows
of boxes. There is a row for each positive $a_i$. The bottom row consists of $a_1$ boxes,
the row above it of $a_2$ boxes, and so on.

\begin{example}\label{example partitions and young diagrams}
The partition $a=(6,6,4,4,4,2,1,1, 0, 0) \in \mathcal{P}(8, 10)$ has diagram $D_{a}$ given by
\begin{center}
\ydiagram{1,1,2,4,4,4,6,6}
\end{center}
\end{example}

\subsection{Flags}
Let $V$ be a complex $n$-dimensional vector space.
A \emph{(complete) flag $F_{\ast}$ in $V$} is a sequence of linear subspaces
$$
\{0\} = F_{0} \subset F_{1} \subset \dots \subset F_{n-1} \subset F_{n} = V
$$
such that $\operatorname{dim}_{\mathbb{C}}F_{i} = i$ for all $i$.
Let $GL(V)$ denote the general linear group of $V$.
For $V = \mathbb{C}^{n}$, we also write $GL_{n}(\mathbb{C}) = GL(\mathbb{C}^{n})$.
Given a flag $F_{\ast}$ in $V$ and $g \in GL(V)$, we obtain a flag $g \cdot F_{\ast}$ in $V$ given by the sequence of subspaces
$$
\{0\} = g \cdot F_{0} \subset g \cdot F_{1} \subset \dots \subset g \cdot F_{n-1} \subset g \cdot F_{n} = V.
$$
Conversely, any two flags $F_{\ast}$ and $F_{\ast}'$ in $V$ are related by $F_{\ast}' = g \cdot  F_{\ast}$ for a suitable $g \in GL(V)$.
Since $GL(V)$ is path connected, there is thus a continuous family of linear automorphisms of $V$ taking $F_{\ast}$ to $F_{\ast}'$.

\subsection{Schubert Varieties}
For integers $0 \leq k \leq n$, let $G = G_{k}(V)$ denote the Grassmann variety of $k$-dimensional linear subspaces of an $n$-dimensional complex vector space $V$.
This is a nonsingular complex projective algebraic variety of dimension $\dim_\cplx G = k(n-k)$.
For a partition $a \in \mathcal{P}(n-k, k)$ and a flag $F_{\ast}$ in $V$, the \emph{Schubert variety}
\begin{align*}
X_{a}(F_{\ast}) = \{P \in G| \operatorname{dim}_{\mathbb{C}}(P \cap F_{a_{k+1-i}+i}) \geq i, \; 1 \leq i \leq k\} \subset G
\end{align*}
is a closed subvariety of $G$ of complex dimension $|a| = a_{1} + \dots + a_{k}$ (see e.g. \cite{aluffi}).
If the flag $F_{\ast}$ is understood, we will also write $X_{a} = X_{a}(F_{\ast})$.
Grassmannians are themselves Schubert varieties as we have $G = X_{[(n-k) \times k]}$.
For any partition $b \in \mathcal{P}(n-k, k)$ with $b\leq a$, we have a closed embedding $X_b \subset X_a$.

The Grassmannian $G$ comes with a transitive action of $GL(V)$.
Note that we have $g \cdot X_{a}(F_{\ast}) = X_{a}(g \cdot F_{\ast})$ for all $g \in GL(V)$.
Since any two flags in $V$ are related by a continuous family of linear automorphisms of $V$, any Schubert variety $X_{a}(F_{\ast}')$ that is associated to another flag $F_{\ast}'$ in $V$ is a translate of the original variety $X_{a}(F_{\ast})$ by an isotopy of $G$.
In particular, the ambient fundamental class $[X_{a}]_{G} \in H_{\ast}(G; \mathbb{Z})$ is well-defined without specifying a flag.

If only the Young diagram of the partition $a$ (but not $a$ itself) are known, then the diagram determines
Schubert varieties in every Grassmannian $G_k (V)$ with
$k$ at least as large as the number of rows and $n-k \geq a_1$.
It turns out that all these Schubert varieties are isomorphic to each other.
Thus we may speak of \emph{the} Schubert variety associated to a Young diagram.
The dimension of $X_{a}$ is the number of boxes in the diagram of $a$.

\subsection{The Homology of Schubert Varieties}
The Chow homology $A_* (X_a)$ of a Schubert variety $X_a$
is freely generated by the Schubert classes $[X_b]$ for all
$b\leq a$ (see \cite[p. 7]{aluffi}).
For any complex variety $X$, let 
\[ \cl: A_i (X) \longrightarrow H^\BM_{2i} (X;\intg) \]
denote the cycle map from Chow homology to Borel-Moore homology.
By \cite[p. 378, Example 19.1.11]{fultonintth}, the cycle map is an isomorphism 
for varieties $X$ that possess a cellular decomposition
in the sense of Fulton \cite[p. 23, Example 1.9.1]{fultonintth}.
Now using Schubert cells, Schubert varieties do indeed have a cellular decomposition
in this sense.
Therefore,
\[ \cl: A_i (X_a) \cong H_{2i} (X_a;\intg) \]
for Schubert varieties $X_a$.
(Note that Schubert varieties are compact, and thus their ordinary singular homology
agrees with Borel-Moore homology.)
Thus $H_{2i} (X_a; \intg)$ is a free abelian group generated by all fundamental classes
$[X_b]$ of Schubert varieties $X_b$ with $b\leq a$ and $\dim X_b = \sum b_j = i$.
The homology $H_{2i+1} (X_a;\intg)$ in odd degrees vanishes.
All of this applies of course to $H_* (G;\intg)$ itself, since the
Grassmannian $G = G_{k}(V)$ is a particular Schubert variety.
If $X_a$ is a Schubert variety in $G$, then the closed embedding
$j: X_a \hookrightarrow G$
induces a map
\[ j_*: H_* (X_a;\intg) \longrightarrow H_* (G;\intg). \]
In consistency with our earlier notation on fundamental classes, we shall write $[X_b]_{X_a} \in H_* (X_a;\intg)$ for the homology class which $X_b$ defines in $X_a$ for a partition $b \leq a$, and
$[X_b]_G \in H_* (G;\intg)$ for the homology class of $X_b$
in the Grassmannian, that is,
\[ [X_b]_G = j_* [X_b]_{X_a}.  \]
Note that the map $j_*$ is always a monomorphism, since
$b\leq a$ and $a\leq (n-k, \ldots, n-k)$ implies
$b\leq (n-k, \ldots, n-k)$, so $X_b$ defines a homology generator
for $H_* (G; \mathbb{Z})$.

\subsection{The Singular Set of a Schubert Variety}
\label{sec.singset}
We compute the singular set of a Schubert variety $X_a$ using
a result of Lakshmibai-Weyman
\cite[Theorem 5.3, p. 203]{lakwey}.
Given a partition $a \in \mathcal{P}(m, k)$ with $a_{k} > 0$, write $a$ as
\[ a = (\underbrace{p_1, \ldots, p_1}_{q_1}, 
        \underbrace{p_2,\ldots, p_2}_{q_2},\ldots, 
        \underbrace{p_r, \ldots, p_r}_{q_r}) \]
with $p_1 > p_2 > \cdots > p_r$.
Thus the Young diagram of $X_a$ consists of a bottom rectangle with $q_1$ rows
of length $p_1$, then a rectangle with $q_2$ rows of length $p_2$, etc.

\begin{example}  \label{exple.a66444211}
For $a=(6,6,4,4,4,2,1,1)$ with the Young diagram of \Cref{example partitions and young diagrams}, we have
\[ p_1 = 6, q_1 =2;~ p_2 = 4, q_2 = 3;~ p_3 =2, q_3 =1;~ p_4 =1, q_4 =2. \]
The diagram contains $4$ rectangles, so $r=4$.
\end{example}
With this notation, the theorem asserts:
\begin{thm} (Lakshmibai, Weyman.) \label{thm.lakwey}
The singular set of $X_a$ is a union of $r-1$ Schubert-subvarieties given by
\[ \Sing X_a =  X_{a^{(1)}} \cup \cdots \cup X_{a^{(r-1)}}, \]
where the partition $a^{(i)}$ is given by
\[ \begin{array}{rll}
a^{(i)} =(& p_1, \ldots, p_1 & (q_1 \text{ entries}) \\
 & \vdots & \\ 
 & p_{i-1}, \ldots, p_{i-1} & (q_{i-1} \text{ entries}) \\
 & p_{i}, \ldots, p_{i} & (q_i -1 \text{ entries}) \\
 & p_{i+1} -1, \ldots, p_{i+1} -1 & (q_{i+1} +1 \text{ entries}) \\
 & p_{i+2}, \ldots, p_{i+2} & (q_{i+2} \text{ entries}) \\
 & \vdots & \\ 
 & p_{r}, \ldots, p_{r}) & (q_r \text{ entries}) \\
\end{array} \]
In particular, $X_a$ is nonsingular if and only if $a$ is a rectangular partition.
(In the latter case, $r=1$.)
\end{thm}
Pictorially, this means that 
in order to obtain the Young diagram $D_{a^{(i)}}$, one must delete one row from the $i$-th rectangle in the diagram $D_{a}$,
delete from the rectangle above it the rightmost column, and, 
after having done this, add one more row to that rectangle.

\begin{example}
We return to the partition $a=(6,6,4,4,4,2,1,1)$ considered in \Cref{exple.a66444211}.
The component $X_{a^{(1)}}$ ($i=1$) of the singular set of the Schubert variety $X_a$ has Young diagram $D_{a^{(1)}}$ obtained by deleting the following L-shaped part marked in yellow in the diagram $D_{a}$:
\begin{center}
\ydiagram[*(yellow)]{0,0,0,3+1,3+1,3+1,3+3,0}
*[*(white)]{1,1,2,4,4,4,6,6} 
\end{center}
\end{example}

\section{Intersection Theory of Schubert Cycles}
\label{Intersection of Schubert varieties}
In this section, we employ transversality of flags (see \Cref{Transversality of flags}) and the Segre product of Grassmannians (see \Cref{Segre products in ambient Grassmannians}) to study physical intersections of Schubert cycles in an ambient Grassmannian $G$, and deduce formulae for intersection products of Schubert classes in the homology of $G$.

\subsection{Transversality of Flags}\label{Transversality of flags}
We adopt the following useful notion of transversality for flags in an $n$-dimensional complex vector space $V$.

\begin{defn}\label{definition transverse flags}
(See Eisenbud-Harris \cite[Definition 4.4, p. 139]{eh}.)
We say that two flags $F_{\ast}$ and $F_{\ast}'$ in $V$ are \emph{transverse} if any of the following equivalent statements holds:
\begin{enumerate}
\item[(a)] $F_{i} \cap F_{n-i}' = \{0\}$ for all $i$.
\item[(b)] $\operatorname{dim}_{\mathbb{C}}(F_{i} \cap F_{j}') = \operatorname{max} \{i+j-n, 0\}$ for all $i, j$.
\item[(c)] There exists a basis $v_{1}, \dots, v_{n}$ of $V$ such that $F_i = \langle v_1,\ldots, v_i \rangle$ and $F_i' = \langle v_{n+1-i},\ldots, v_n \rangle$ for all $i$.
\end{enumerate}
\end{defn}

\begin{prop}\label{proposition transversality of flags is open and dense condition}
If $F_{\ast}$ and $F_{\ast}'$ are transverse flags in $\mathbb{C}^{n}$, then the subset $U \subset GL_{n}(\cplx)$ of all $g \in GL_{n}(\cplx)$ such that $F_{\ast}$ and $g \cdot F_{\ast}'$ are transverse is a Zariski open neighborhood of $\operatorname{id}_{\mathbb{C}^{n}} \in GL_{n}(\cplx)$.
\end{prop}

\begin{proof}
Choose bases $v_{1}, \dots, v_{n}$ and $v_{1}', \dots, v_{n}'$ of $\mathbb{C}^{n} = \mathbb{C}^{n \times 1}$ such that
\[ F_i = \langle v_1,\ldots, v_i \rangle,~ i=1,\ldots, n, \]
and
\[ F_i' = \langle v_1',\ldots, v_i' \rangle,~ i=1,\ldots, n. \]
Then, for $g \in GL_{n}(\cplx)$, $F_{\ast}$ is transverse to $g \cdot F_{\ast}'$ if and only if
\begin{equation}\label{equation transverse flags algebraic locus}
\operatorname{det}(v_{1}, \dots, v_{i}, g \cdot v_{1}', \dots, g \cdot v_{n-i}') \neq 0, \quad 1 \leq i \leq n-1,
\end{equation}
which is a Zariski open condition on $g \in GL_{n}(\cplx)$.
(More precisely, $GL_{n}(\cplx)$ is the zero locus in $\cplx^{n^{2}} \times \cplx$ of the polynomial $(g, t) \mapsto t \cdot \operatorname{det}(g) - 1$, and condition (\ref{equation transverse flags algebraic locus}) is understood in terms of the polynomials $(g, t) \mapsto \operatorname{det}(v_{1}, \dots, v_{i}, g \cdot v_{1}', \dots, g \cdot v_{n-i}')$, $1 \leq i \leq n-1$.)
Consequently, $U$ is Zariski open in $GL_{n}(\cplx)$.
\end{proof}

\begin{prop}\label{proposition transformation of transverse flags}
If $(E_{\ast}, E_{\ast}')$ and $(F_{\ast}, F_{\ast}')$ are pairs of transverse flags in $V$, then there exists $g \in GL(V)$ such that $g \cdot E_{\ast} = F_{\ast}$ and $g \cdot E_{\ast}' = F_{\ast}'$.
\end{prop}

\begin{proof}
Using statment (c) of \Cref{definition transverse flags}, we find bases $v_{1}, \dots, v_{n}$ and $w_{1}, \dots, w_{n}$ of $V$ such that $E_i = \langle v_1,\ldots, v_i \rangle$ and $E_i' = \langle v_{n+1-i},\ldots, v_n \rangle$ for all $i$, and $F_i = \langle w_1,\ldots, w_i \rangle$ and $F_i' = \langle w_{n+1-i},\ldots, w_n \rangle$ for all $i$.
Let $g \in GL(V)$ be defined by $g \cdot v_{i} = w_{i}$ for all $i$.
Then, as desired, we have $g \cdot E_{i} = F_{i}$ and $g \cdot E_{i}' = F_{i}'$ for all $i$.
\end{proof}

\begin{prop}\label{proposition existence of flag with transversality properties}
Let $0 \leq k \leq n$ be integers, and let $a \in \mathcal{P}(n-k, k)$.
Let $V$ be a complex $n$-dimensional vector space, and let $E_{\ast}$ be a flag in $V$.
Let $X \subset G$ be a closed irreducible subvariety of the Grassmannian $G := G_{k}(V)$.
Then, there exists a flag $F_{\ast}$ in $V$ that is transverse to $E_{\ast}$, and such that the Schubert subvariety $X_{a}(F_{\ast}) \subset G$ is simultaneously Whitney transverse and generically transverse to $X \subset G$.
\end{prop}

\begin{proof}
Let $E_{\ast}'$ be a flag in $V$ that is transverse to $E_{\ast}$.
By \Cref{proposition transversality of flags is open and dense condition}, there is an open dense subset $U \subset GL(V)$ (in the Zariski topology) such that the flags $E_{\ast}$ and $g \cdot E_{\ast}'$ are transverse in $V$ for $g \in U$.
Fix Whitney stratifications $\mathcal{W}$ and $\mathcal{W}'$ on $X \subset G$ and $X_{a}(E_{\ast}') \subset G$, respectively.
By a version of the Kleiman transversality theorem (see \Cref{thm kleiman for whitney transversality}), there is an open dense subset $U' \subset GL(V)$ (in the complex topology) such that $X$ (Whitney stratified by $\mathcal{W}$) and $g \cdot X_{a}(E_{\ast}')$ (Whitney stratified by $g \cdot \mathcal{W}'$) are transverse in $G$ for $g \in U'$.
Moreover, by Kleiman's transversality theorem (see \Cref{thm.kleiman}), there is an open dense subset $U'' \subset GL(V)$ (in the Zariski topology) such that $X$ and $g \cdot X_{a}(E_{\ast}')$ are generically transverse in $G$ for $g \in U''$.
By \cite[Theorem 1, p. 58]{mumfordred}, the open subsets $U, U'' \subset GL(V)$ are also dense in the complex topology.
Therefore, in the complex topology on $GL(V)$, the intersection $U \cap U' \cap U''$ is a dense subset of $GL(V)$, and hence nonempty.
Fix $g \in U \cap U' \cap U''$.
Then, $F_{\ast} := g \cdot E_{\ast}'$ is a flag in $V$ with the desired properties because $g \cdot X_{a}(E_{\ast}') = X_{a}(g \cdot E_{\ast}') = X_{a}(F_{\ast})$.
\end{proof}

\begin{cor}\label{proposition transverse flags imply whitney transverse}
Let $0 \leq k \leq n$ be integers.
Let $V$ be a complex $n$-dimensional vector space, and let $a, b \in \mathcal{P}(n-k, k)$.
If $F_{\ast}$ and $F_{\ast}'$ are transverse flags in $V$, then $X_{a}(F_{\ast})$ and $X_{b}(F_{\ast}')$ are simultaneously Whitney transverse and generically transverse in $G_{k}(V)$.
\end{cor}

\begin{proof}
Let $G = G_{k}(V)$.
By \Cref{proposition existence of flag with transversality properties}, there exists a flag $F_{\ast}''$ in $V$ that is transverse to $F_{\ast}$, and such that the Schubert subvariety $X_{b}(F_{\ast}'') \subset G$ is simultaneously Whitney transverse and generically transverse to $X_{a}(F_{\ast}) \subset G$.
Then, as $(F_{\ast}, F_{\ast}')$ and $(F_{\ast}, F_{\ast}'')$ are pairs of transverse flags in $V$, there exists by \Cref{proposition transformation of transverse flags} an element $h \in GL(V)$ such that $h \cdot F_{\ast} = F_{\ast}$ and $h \cdot F_{\ast}' = F_{\ast}''$.
Thus, under the automorphism of $P$ induced by $h$, we have
\begin{align*}
h \cdot X_{a}(F_{\ast}) &= X_{a}(h \cdot F_{\ast}) = X_{a}(F_{\ast}), \\
h \cdot X_{b}(F_{\ast}') &= X_{b}(h \cdot F_{\ast}') = X_{b}(F_{\ast}'').
\end{align*}
We conclude that $X_{a}(F_{\ast}) = h^{-1} \cdot X_{a}(F_{\ast})$ and $X_{b}(F_{\ast}') = h^{-1} \cdot X_{b}(F_{\ast}'')$ are simultaneously Whitney transverse and generically transverse in $G$.
\end{proof}

\begin{remark}\label{remark richardson irreducible}
Intersections of the general translates of two Schubert varieties (like in \Cref{proposition transverse flags imply whitney transverse}) are called Richardson varieties \cite{rich}.
Richardson varieties are well-known to be irreducible (see e.g. \cite[Remark 2.2]{bc}).
\end{remark}

\begin{cor}\label{corollary transverse flags imply intersection product}
Let $0 \leq k \leq n$ be integers.
Let $V$ be a complex $n$-dimensional vector space, and let $G = G_{k}(V)$.
Let $a, b \in \mathcal{P}(n-k, k)$.
If $F_{\ast}$ and $F_{\ast}'$ are transverse flags in $V$, then $X_{a}(F_{\ast}) \cap X_{b}(F_{\ast}')$ is a pure-dimensional closed subvariety of $G$, and
$$
[X_{a}(F_{\ast}) \cap X_{b}(F_{\ast}')]_{G} = [X_{a}]_{G} \cdot [X_{b}]_{G}.
$$
\end{cor}

\begin{proof}
The claim holds by \Cref{prop codimension of transverse intersection} and \Cref{lemma intersection is cap product for whitney transverse subvarieties}, where we note that $X_{a}(F_{\ast})$ and $X_{b}(F_{\ast}')$ are generically transverse by \Cref{proposition transverse flags imply whitney transverse}.
\end{proof}

For some specific partitions we can compute the intersection of the general translates of two Schubert varieties as follows.
The following result follows from the proof of \cite[Proposition 4.6, p. 141]{eh} (where we note that a partition $a \in \mathcal{P}(n-k, k)$ in our notation corresponds to a partition $\overline{a}$ with $\overline{a}_{i} = (n-k) - a_{k+1-i}$ for all $i$ in the notation of \cite{eh}).

\begin{prop}\label{proposition empty intersection of schubert varieties}
Let $0 \leq k \leq n$ be integers.
Let $V$ be a complex $n$-dimensional vector space, and let $F_{\ast}$ and $F_{\ast}'$ be transverse flags in $V$.
Then, for $a, b \in \mathcal{P}(n-k, k)$ the following holds in $G_{k}(V)$:
\begin{enumerate}
\item[(a)] If there exists $1 \leq i_{0} \leq k$ such that $a_{i_{0}} + b_{k+1-i_{0}} < n-k$, then
$$
X_{a}(F_{\ast}) \cap X_{b}(F_{\ast}') = \varnothing.
$$
\item[(b)] If $a_{i} + b_{k+1-i} = n-k$ for all $1 \leq i \leq k$, then
$$
X_{a}(F_{\ast}) \cap X_{b}(F_{\ast}') = \{\bigoplus_{i=1}^{k}F_{a_{i}+k+1-i} \cap F_{b_{k+1-i}+i}'\}.
$$
\end{enumerate}
\end{prop}

In view of \Cref{corollary transverse flags imply intersection product}, we have the following

\begin{cor}\label{corollary empty intersection of schubert varieties}
Let $0 \leq k \leq n$ be integers.
Let $V$ be a complex $n$-dimensional vector space, and let $G = G_{k}(V)$.
Then, for $a, b \in \mathcal{P}(n-k, k)$ the following holds:
\begin{enumerate}
\item[(a)] If there exists $1 \leq i_{0} \leq k$ such that $a_{i_{0}} + b_{k+1-i_{0}} < n-k$, then
$$
[X_{a}]_{G} \cdot [X_{b}]_{G} = 0 \qquad \in H_{\ast}(G; \mathbb{Q}).
$$
\item[(b)] If $a_{i} + b_{k+1-i} = n-k$ for all $1 \leq i \leq k$, then
$$
[X_{a}]_{G} \cdot [X_{b}]_{G} = [\operatorname{pt}]_{G} \qquad \in H_{\ast}(G; \mathbb{Q}).
$$
\end{enumerate}
\end{cor}

\subsection{Segre Products in Ambient Grassmannians}\label{Segre products in ambient Grassmannians}
In this section, we use the Segre product of Grassmannians in an ambient Grassmannian (see \Cref{proposition segre product of schubert varieties} below) to compute intersections of Schubert varieties associated to amalgamated partitions (see \Cref{corollary twisting rule}).

Suppose that the $n$-dimensional complex vector space $V$ is written as the direct sum $V = V' \oplus V''$ of two linear subspaces $V', V'' \subset V$ of dimensions $\operatorname{dim}_{\mathbb{C}}V' = n'$ and $\operatorname{dim}_{\mathbb{C}}V'' = n''$. Then, for any flags $F_{\ast}'$ in $V'$ and $F_{\ast}''$ in $V''$, we obtain a flag $F_{\ast}' \oplus F_{\ast}''$ in $V$ by defining
$$
(F_{\ast}' \oplus F_{\ast}'')_{w} = \begin{cases}
F_{w}', &\text{if } 0 \leq w \leq n', \\
V' \oplus F_{w-n'}'', &\text{if } n'+1 \leq w \leq n.
\end{cases}
$$
Conversely, any flag $F_{\ast}$ in $V$ with $F_{n'} = V'$ can be written as $F_{\ast} = F_{\ast}' \oplus F_{\ast}''$ for unique flags $F_{\ast}'$ in $V'$ and $F_{\ast}''$ in $V''$, namely $F_{u}' = F_{u}$ for all $u$ and $F_{v}'' = F_{v+n'} \cap V''$ for all $v$.

Writing $m'' = m-m'$ and $k'' = k-k'$ (so that $m = m'+m''$ and $k = k'+k''$), the \emph{amalgamation of partitions} (see \Cref{figure segre partitions}) is defined as the map
\begin{align*}
\sqcup \colon \mathcal{P}(m', k') \times \mathcal{P}(m'', k'') &\rightarrow \mathcal{P}(m, k), \\
(c', c'') &\mapsto \sqcup(c', c'') =: c' \sqcup c'', \\
(c' \sqcup c'')_{i} &= \begin{cases}
m' + c_{i}'', &\text{if } 1 \leq i \leq k'', \\
c_{i-k''}', &\text{if } k''+1 \leq i \leq k.
\end{cases}
\end{align*}

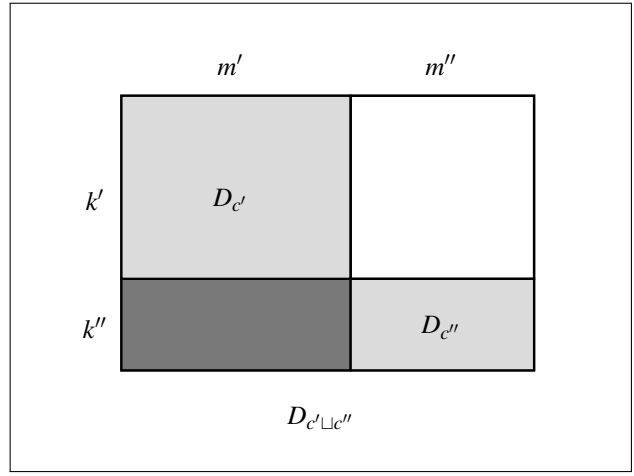
\begin{figure}[htbp]
\centering
\fbox{\begin{tikzpicture}
\draw (0, 0) node {\includegraphics[width=0.6\textwidth]{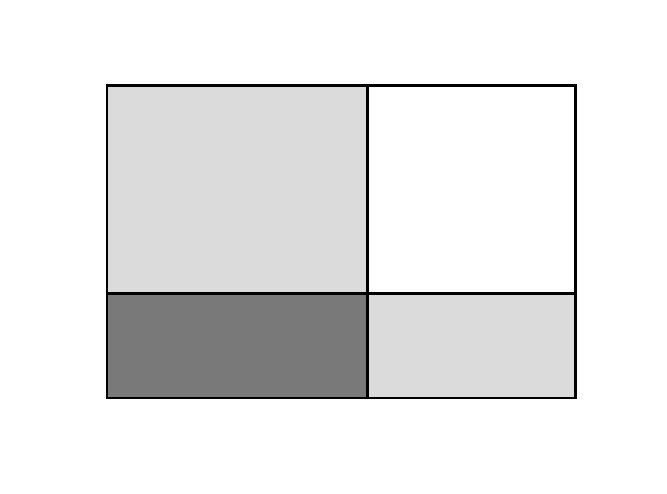}};
\draw (-1.2, 0.5) node {$D_{c'}$};
\draw (1.6, -1.2) node {$D_{c''}$};
\draw (0, -2.4) node {$D_{c' \sqcup c''}$};
\draw (-1.2, 2.3) node {$m'$};
\draw (1.6, 2.3) node {$m''$};
\draw (-3.0, 0.5) node {$k'$};
\draw (-3.0, -1.2) node {$k''$};
\end{tikzpicture}}
\caption{Relation between the Young diagrams $D_{c'}$, $D_{c''}$, and $D_{c' \sqcup c''}$.}
\label{figure segre partitions}
\end{figure}

An important tool is the Segre product of Grassmannians in an ambient Grassmannian (see e.g. Definition 2.13 in \cite{bc}).

\begin{thm}[Segre Product]\label{proposition segre product of schubert varieties}
Let $0 \leq k' \leq n'$ and $0 \leq k'' \leq n''$ be integers, and set $n = n' + n''$ and $k = k' + k''$.
Then, for a complex vector space $V$ of dimension $\operatorname{dim}_{\mathbb{C}}V = n$ that is the direct sum $V = V' \oplus V''$ of linear subspaces of dimensions $\operatorname{dim}_{\mathbb{C}}V' = n'$ and $\operatorname{dim}_{\mathbb{C}}V'' = n''$, the \emph{Segre product of the Grassmannians $G' = G_{k'}(V')$ and $G'' = G_{k''}(V'')$ in $G = G_{k}(V)$},
\begin{align}\label{equation segre product}
S := S^{k', k''}_{V', V''} \colon G' \times G'' \rightarrow G, \qquad S(P', P'') = P' \oplus P'',
\end{align}
is a closed algebraic embedding.
Let $m' = n'-k'$, $m'' = n''-k''$, and $m= n-k \; (=m'+m'')$.
Then, given pairs of flags $(E_{\ast}', F_{\ast}')$ in $V'$ and $(E_{\ast}'', F_{\ast}'')$ in $V''$, and $a', b' \in \mathcal{P}(m', k')$, $a'', b'' \in \mathcal{P}(m'', k'')$, we have
$$
X_{a' \sqcup a''}(E_{\ast}' \oplus E_{\ast}'') \cap X_{b'' \sqcup b'}(F_{\ast}'' \oplus F_{\ast}') = S((X_{a'}(E_{\ast}') \cap X_{b'}(F_{\ast}')) \times (X_{a''}(E_{\ast}'') \cap X_{b''}(F_{\ast}''))).
$$
\end{thm}

\begin{proof}
As for the first claim, it suffices to show that $S$ is an embedding of complex manifolds.
In fact, by Chow's theorem, any analytic subvariety of projective space is algebraic (see e.g. Griffiths-Harris \cite[p. 167]{gh}).
Furthermore, any holomorphic map between two smooth algebraic varieties can be described by rational functions (see e.g. Griffiths-Harris \cite[Fact 2., p. 170]{gh}.)
To show that $S$ is an embedding of complex manifolds, it suffices to work in local charts of the Grassmannians, where the linear subspaces of $V'$ and $V''$ are spanned by the rows of matrices having a permutation matrix as a certain minor, and this presentation is compatible with passing to the direct sum $V' \oplus V''$, so $S$ is locally a homolorphic embedding.
Note that $S$ is injective in view of the projections of $V = V' \oplus V''$ to the first and second summand.

To show the second claim, we set $E_{\ast} = E_{\ast}'' \oplus E_{\ast}'$, $F_{\ast} = F_{\ast}' \oplus F_{\ast}''$, and $a = a' \sqcup a''$, $b = b'' \sqcup b'$.
We recall that
\begin{align*}
X_{a}(E_{\ast}) &= \{P \in G = G_{k}(V)| \operatorname{dim}_{\mathbb{C}}(P \cap E_{a_{k+1-i}+i}) \geq i, \; 1 \leq i \leq k\}, \\
X_{a'}(E_{\ast}') &= \{P' \in G_{k'}(E_{n'})| \operatorname{dim}_{\mathbb{C}}(P' \cap E_{a_{k'+1-i'}'+i'}') \geq i', \; 1 \leq i' \leq k'\} \\
 &= \{W' \in G_{k'}(E_{n'})| \operatorname{dim}_{\mathbb{C}}(P' \cap E_{a_{k+1-i'}+i'}) \geq i', \; 1 \leq i' \leq k'\}, \\
X_{a''}(E_{\ast}'') &= \{P'' \in G_{k''}(F_{n''})| \operatorname{dim}_{\mathbb{C}}(W'' \cap E_{a_{k''+1-i''}''+i''}'') \geq i'', \; 1 \leq i'' \leq k''\} \\
 &= \{W'' \in G_{k''}(F_{n''})| \operatorname{dim}_{\mathbb{C}}(P'' \cap E_{a_{k''+1-i''}+i''+k'} \cap F_{n''}) \geq i'', \; 1 \leq i'' \leq k''\}, \\
X_{b}(F_{\ast}) &= \{P \in G = G_{k}(V)| \operatorname{dim}_{\mathbb{C}}(P \cap F_{b_{k+1-i}+i}) \geq i, \; 1 \leq i \leq k\}, \\
X_{b'}(F_{\ast}') &= \{P' \in G_{k'}(E_{n'})| \operatorname{dim}_{\mathbb{C}}(P' \cap F_{b_{k'+1-i'}'+i'}') \geq i', \; 1 \leq i' \leq k'\} \\
 &= \{W' \in G_{k'}(E_{n'})| \operatorname{dim}_{\mathbb{C}}(P' \cap E_{n'} \cap F_{b_{k'+1-i'}+i'+k''}') \geq i', \; 1 \leq i' \leq k'\}, \\
X_{b''}(F_{\ast}'') &= \{P'' \in G_{k''}(F_{n''})| \operatorname{dim}_{\mathbb{C}}(P'' \cap F_{b_{k''+1-i''}''+i''}'') \geq i'', \; 1 \leq i'' \leq k''\} \\
&= \{P'' \in G_{k''}(F_{n''})| \operatorname{dim}_{\mathbb{C}}(P'' \cap F_{b_{k+1-i''}+i''}) \geq i'', \; 1 \leq i'' \leq k''\}.
\end{align*}
Without loss of generality, it suffices to prove that
\begin{equation}\label{equation special case a}
X_{a}(E_{\ast}) \cap X_{[m'' \times k''] \sqcup [m' \times k']}(F_{\ast}) = S(X_{a'}(E_{\ast}') \times X_{a''}(E_{\ast}''))
\end{equation}
and
\begin{equation}\label{equation special case b}
X_{[m' \times k'] \sqcup [m'' \times k'']}(E_{\ast}) \cap X_{b}(F_{\ast}) = S(X_{b'}(F_{\ast}') \times X_{b''}(F_{\ast}'')).
\end{equation}
(In fact, using that $a \leq [m' \times k'] \sqcup [m'' \times k'']$ and $b \leq [m'' \times k''] \sqcup [m' \times k']$, we then obtain
\begin{align*}
X_{a}(E_{\ast}) \cap X_{b}(F_{\ast}) &= S(X_{a'}(E_{\ast}') \times X_{a''}(E_{\ast}'')) \cap S(X_{b'}(F_{\ast}') \times X_{b''}(F_{\ast}'')) \\
&= S((X_{a'}(E_{\ast}') \times X_{a''}(E_{\ast}'')) \cap (X_{b'}(F_{\ast}') \times X_{b''}(F_{\ast}''))) \\
&= S((X_{a'}(E_{\ast}') \cap X_{b'}(F_{\ast}')) \times (X_{a''}(E_{\ast}'') \cap X_{b''}(F_{\ast}'')))
\end{align*}
where we have also used that $S$ is injective.)
In the following, we show (\ref{equation special case a}), where we write $b' = [m' \times k']$, $b'' = [m'' \times k'']$, and $b = [m'' \times k''] \sqcup [m' \times k']$ (the proof of (\ref{equation special case b}) is similar).

Let us first show the inclusion $S(X_{a'}(E_{\ast}') \times X_{a''}(E_{\ast}'')) \subset X_{a}(E_{\ast}) \cap X_{b}(F_{\ast})$ in (\ref{equation special case a}).
For this purpose, suppose that $P' \in X_{a'}(E_{\ast}')$ and $P'' \in X_{a''}(E_{\ast}'')$, and consider $P := P' \oplus P'' = S(P', P'') \in G = G_{k}(V)$.
To show that $P \in X_{a}(E_{\ast})$, we have to show that $\operatorname{dim}_{\mathbb{C}}(P \cap E_{a_{k+1-i}+i}) \geq i$ for $1 \leq i \leq k$.
If $1 \leq i \leq k'$, then it follows from $P' \subset P$ and $P' \in X_{a'}(E_{\ast}')$ that
$$
\operatorname{dim}_{\mathbb{C}}(P \cap E_{a_{k+1-i}+i}) \geq \operatorname{dim}_{\mathbb{C}}(P' \cap E_{a_{k+1-i}+i}) \geq i.
$$
If $k'+1 \leq i \leq k$, then we use $P = P' \oplus P''$, $P' \subset E_{n'} \subset E_{a_{k+1-i}+i}$ (where the second inclusion holds since $k+1-i \leq k''$ implies that $a_{k+1-i} \geq m'$, and hence $a_{k+1-i} + i\geq n'$), and $P'' \in X_{a''}(E_{\ast}'')$ to conclude that
\begin{align*}
\operatorname{dim}_{\mathbb{C}}(P \cap E_{a_{k+1-i}+i}) &\geq \operatorname{dim}_{\mathbb{C}}(P' \cap E_{a_{k+1-i}+i}) + \operatorname{dim}_{\mathbb{C}}(P'' \cap E_{a_{k+1-i}+i}) \\
&\geq \operatorname{dim}_{\mathbb{C}}(P') + \operatorname{dim}_{\mathbb{C}}(P'' \cap E_{a_{k''+1-(i-k')}+(i-k')+k'} \cap F_{n''}) \\
&\geq k' + (i-k') = i.
\end{align*}
To show that $P \in X_{b}(F_{\ast})$, we have to show that $\operatorname{dim}_{\mathbb{C}}(P \cap F_{b_{k+1-i}+i}) \geq i$ for $1 \leq i \leq k$.
Since $b_{1} = \dots = b_{k'} = m$ and $b_{k'+1} = \dots = b_{k} = m''$, it suffices to show that $\operatorname{dim}_{\mathbb{C}}(P \cap F_{n}) \geq k$ (which implies $\operatorname{dim}_{\mathbb{C}}(P \cap F_{b_{k+1-i}+i}) \geq i$ for $1 \leq i \leq k'$) and $\operatorname{dim}_{\mathbb{C}}(P \cap F_{n''}) \geq k''$ (which implies $\operatorname{dim}_{\mathbb{C}}(P \cap F_{b_{k+1-i}+i}) \geq i$ for $k'+1 \leq i \leq k$).
Indeed, it follows from $P \subset F_{n}$ that $\operatorname{dim}_{\mathbb{C}}(P \cap F_{n}) = \operatorname{dim}_{\mathbb{C}}(P) = k$, and from $P'' \subset P \cap F_{n''}$ that $\operatorname{dim}_{\mathbb{C}}(P \cap F_{n''}) \geq \operatorname{dim}_{\mathbb{C}}(P) \geq k''$.

Conversely, let us show that $X_{a}(E_{\ast}) \cap X_{b}(F_{\ast}) \subset S(X_{a'}(E_{\ast}') \times X_{a''}(E_{\ast}''))$ in (\ref{equation special case a}).
Given $P \in X_{a}(E_{\ast}) \cap X_{b}(F_{\ast})$, we set $P' = P \cap E_{n'}$ and $P'' = P \cap F_{n''}$, and note that $\operatorname{dim}_{\mathbb{C}}(P') = \operatorname{dim}_{\mathbb{C}}(P \cap E_{n'}) \geq \operatorname{dim}_{\mathbb{C}}(P \cap E_{a_{k+1-k'} +k'}) \geq k'$ because $P \in X_{a}(E_{\ast})$ (where we note that $E_{a_{k+1-k'} +k'} \subset E_{n'}$ holds because $a_{k''+1} \leq m'$ implies that $a_{k+1-k'} +k' \leq n'$), and $\operatorname{dim}_{\mathbb{C}}(P'') = \operatorname{dim}_{\mathbb{C}}(P \cap F_{n''}) = \operatorname{dim}_{\mathbb{C}}(P \cap F_{b_{k+1-k''} +k''}) \geq k''$ because $P \in X_{b}(F_{\ast})$ (where we note that $F_{b_{k+1-k''} +k''} = F_{n''}$ holds because $b_{k'+1} = m''$ implies that $b_{k+1-k''} +k'' = n''$).
Hence, we have $\operatorname{dim}_{\mathbb{C}}(P') = k'$ and $\operatorname{dim}_{\mathbb{C}}(P'') = k''$ because $P' \cap P'' = 0$.
It remains to show that $P' \in X_{a'}(E_{\ast}')$ and $P'' \in X_{a''}(E_{\ast}'')$.
Then, it follows from $k = k' + k''$ that $S(P', P'') = P' + P'' = P$.
To show that $P' \in X_{a'}(E_{\ast}')$, we have to show that $\operatorname{dim}_{\mathbb{C}}(P' \cap E_{a_{k+1-i'}+i'}) \geq i'$ for $1 \leq i' \leq k'$.
In fact, it follows from $E_{a_{k+1-i'}+i'} \subset E_{n'}$ (which holds because $k+1-i' \geq k''+1$ implies that $a_{k+1-i'} \leq m'$, and hence $a_{k+1-i'}+i' \leq n'$) and $P \in X_{a}(E_{\ast})$ that
$$
\operatorname{dim}_{\mathbb{C}}(P' \cap E_{a_{k+1-i'}+i'}) = \operatorname{dim}_{\mathbb{C}}(P \cap E_{n'} \cap E_{a_{k+1-i'}+i'}) = \operatorname{dim}_{\mathbb{C}}(P \cap E_{a_{k+1-i'}+i'}) \geq i'.
$$
Finally, to show that $P'' \in X_{a''}(E_{\ast}'')$, we have to show that $\operatorname{dim}_{\mathbb{C}}(P'' \cap E_{a_{k''+1-i''}+i''+k'} \cap F_{n''}) \geq i''$ for $1 \leq i'' \leq k''$.
Using $P'' = P \cap F_{n''}$ and $P \in X_{a}(F_{\ast})$, we see for $1 \leq i'' \leq k''$ that
\begin{align*}
&\operatorname{dim}_{\mathbb{C}}(P'' \cap E_{a_{k''+1-i''}+i''+k'} \cap F_{n''}) \\
&= \operatorname{dim}_{\mathbb{C}}(P \cap E_{a_{k''+1-i''}+i''+k'} \cap F_{n''}) \\
&= \operatorname{dim}_{\mathbb{C}}(P \cap E_{a_{k''+1-i''}+i''+k'}) + \operatorname{dim}_{\mathbb{C}}(P \cap F_{n''}) - \operatorname{dim}_{\mathbb{C}}(P \cap E_{a_{k''+1-i''}+i''+k'} + P \cap F_{n''}) \\
&\geq \operatorname{dim}_{\mathbb{C}}(P \cap E_{a_{k+1-(i''+k')}+i''+k'}) + \operatorname{dim}_{\mathbb{C}}(P'') - \operatorname{dim}_{\mathbb{C}}(P) \\
&\geq (i'' + k') + k'' - k = i''.
\end{align*}

This completes the proof of \Cref{proposition segre product of schubert varieties}.
\end{proof}

\begin{cor}\label{corollary twisting of partitions}
For $a', b' \in \mathcal{P}(m', k')$ and $a'', b'' \in \mathcal{P}(m'', k'')$ the map $S_{\ast} \colon H_{\ast}(G' \times G''; \mathbb{Q}) \rightarrow H_{\ast}(G; \mathbb{Q})$ induced by the Segre product (\ref{equation segre product}) satisfies
$$
[X_{a' \sqcup a''}]_{G} \cdot [X_{b'' \sqcup b'}]_{G} = S_{\ast}(([X_{a'}]_{G'} \cdot [X_{b'}]_{G'}) \times ([X_{a''}]_{G''} \cdot [X_{b''}]_{G''})).
$$
\end{cor}

\begin{proof}
Fix pairs of transverse flags $(E_{\ast}', F_{\ast}')$ in $V'$ and $(E_{\ast}'', F_{\ast}'')$ in $V''$.
Then, it follows that $E_{\ast}' \oplus E_{\ast}''$ and $F_{\ast}'' \oplus F_{\ast}'$ are transverse flags in $V = V' \oplus V''$.
By \Cref{corollary transverse flags imply intersection product}, $X' = X_{a'}(E_{\ast}') \cap X_{b'}(F_{\ast}')$ is a pure-dimensional closed subvariety of $G'$, and $[X']_{G'} = [X_{a'}]_{G'} \cdot [X_{b'}]_{G'}$.
Similarly, $X'' = X_{a''}(E_{\ast}'') \cap X_{b''}(F_{\ast}'')$ is a pure-dimensional closed subvariety of $G''$, and $[X'']_{G''} = [X_{a''}]_{G''} \cdot [X_{b''}]_{G''}$.
Moreover, $X = X_{a' \sqcup a''}(E_{\ast}' \oplus E_{\ast}'') \cap X_{b'' \sqcup b'}(F_{\ast}'' \oplus F_{\ast}')$ is a pure-dimensional closed subvariety of $G$, and $[X]_{G} = [X_{a' \sqcup a''}]_{G} \cdot [X_{b'' \sqcup b'}]_{G}$.
Using \Cref{proposition segre product of schubert varieties}, we have
\begin{equation}\label{equation 1 proof twisting of partitions}
[X_{a' \sqcup a''}]_{G} \cdot [X_{b'' \sqcup b'}]_{G} = [X]_{G} = [S(X' \times X'')]_{G}.
\end{equation}
Next, we have
$$
[X' \times X'']_{G' \times G''} = [X']_{G'} \times [X'']_{G''} = ([X_{a'}]_{G'} \cdot [X_{b'}]_{G'}) \times ([X_{a''}]_{G''} \cdot [X_{b''}]_{G''}).
$$
Hence, using the Segre product $S \colon G' \times G'' \rightarrow G$ (\ref{equation segre product}) and the induced map $S_{\ast} \colon H_{\ast}(G' \times G''; \mathbb{Q}) \rightarrow H_{\ast}(G; \mathbb{Q})$, we obtain
\begin{equation}\label{equation 2 proof twisting of partitions}
[S(X' \times X'')]_{G} = S_{\ast}([X' \times X'']_{G' \times G''}) = S_{\ast}(([X_{a'}]_{G'} \cdot [X_{b'}]_{G'}) \times ([X_{a''}]_{G''} \cdot [X_{b''}]_{G''})).
\end{equation}
Finally, the claim follows by combining (\ref{equation 1 proof twisting of partitions}) and (\ref{equation 2 proof twisting of partitions}).
\end{proof}

An immediate consequence of \Cref{corollary twisting of partitions} is

\begin{thm}[``Intersection Box Extension Principle'']\label{corollary twisting rule}
Given $a', a'_{\ast}, b', b'_{\ast} \in \mathcal{P}(m', k')$ with $[X_{a'}]_{G'} \cdot [X_{b'}]_{G'} = [X_{a'_{\ast}}]_{G'} \cdot [X_{b'_{\ast}}]_{G'}$ and $a'', b'' \in \mathcal{P}(m'', k'')$, we have
$$
[X_{a' \sqcup a''}]_{G} \cdot [X_{b'' \sqcup b'}]_{G} = [X_{a'_{\ast} \sqcup a''}]_{G} \cdot [X_{b'' \sqcup b'_{\ast}}]_{G}.
$$
\end{thm}

\section{Gysin Coherent Characteristic Classes}
In this section, we introduce the notion of Gysin coherent characteristic classes (see \Cref{definition gysin coherent characteristic classes} below), which is central to the main result of this paper (see \Cref{main result on Gysin coherent characteristic classes with respect to x}).

In the following, by a variety we mean a pure-dimensional complex quasiprojective algebraic variety.

Let $\mathcal{X}$ be a family of inclusions $i \colon X \rightarrow W$, where $W$ is a smooth variety, and $X \subset W$ is a compact irreducible subvariety.
We require the following properties for $\mathcal{X}$:
\begin{itemize}
\item For every Schubert subvariety $X \subset G$ of a Grassmannian $G$, the inclusion $X \rightarrow G$ is in $\mathcal{X}$.
\item If $i \colon X \rightarrow W$ and $i' \colon X' \rightarrow W'$ are in $\mathcal{X}$, then the product $i \times i' \colon X \times X' \rightarrow W \times W'$ is in $\mathcal{X}$.
\item Given inclusions $i \colon X \rightarrow W$ and $i' \colon X' \rightarrow W'$ of compact subvarieties in smooth varieties, and an isomorphism $W \stackrel{\cong}{\longrightarrow} W'$ that restricts to an isomorphism $X \stackrel{\cong}{\longrightarrow} X'$, it follows from $i \in \mathcal{X}$ that $i' \in \mathcal{X}$.
\item For all closed subvarieties $X \subset M \subset W$ such that $X$ is compact and $M$ and $W$ are smooth, it holds that if the inclusion $X \rightarrow M$ is in $\mathcal{X}$, then the inclusion $X \rightarrow W$ is in $\mathcal{X}$.
\end{itemize}

Next, for a given family $\mathcal{X}$ of inclusions as above, by \emph{$\mathcal{X}$-transversality}, we mean a symmetric relation for closed irreducible subvarieties of a smooth variety that satisfies the following properties:
\begin{itemize}
\item The intersection $Z \cap Z'$ of two $\mathcal{X}$-transverse closed irreducible subvarieties $Z, Z' \subset W$ of a smooth variety $W$ is \emph{proper}, that is, $Z \cap Z'$ is pure-dimensional of codimension $c + c'$, where $c$ and $c'$ are the codimensions of $Z$ and $Z'$ in $W$, respectively.
\item The following analog of Kleiman's transversality theorem holds for the action of $GL_{n}(\mathbb{C})$ on the Grassmannians $G = G_{k}(\mathbb{C})$.
If $i \colon X \rightarrow G$ and $i' \colon X' \rightarrow G$ are inclusions in $\mathcal{X}$, then there is a nonempty open dense subset $U \subset GL_{n}(\mathbb{C})$ (in the complex topology) such that $X$ is $\mathcal{X}$-transverse to $g \cdot X'$ for all $g \in U$.
\item Locality: If $Z, Z' \subset W$ are $\mathcal{X}$-transverse closed irreducible subvarieties of a smooth variety $W$ and $U \subset W$ is an open subset that has nontrivial intersections with $Z$ and $Z'$, then $Z \cap U$ and $Z' \cap U$ are $\mathcal{X}$-transverse in $U$.
\end{itemize}

\begin{example}
The family $\mathcal{X}$ consisting of all inclusions of compact irreducible subvarieties in smooth varieties satisfies the above requirements.
In this case, we typically choose $\mathcal{X}$-transversality to mean simultaneously Whitney transversality and generic transversality.
In future applications of the framework, one may wish to restrict to Cohen-Macaulay $X$ and one may wish to incorporate Tor-independence into the notion of $\mathcal{X}$-transversality.
Note that the above requirements for $\mathcal{X}$ and $\mathcal{X}$-transversality are then still satisfied by Sierra's general homological Kleiman-Bertini theorem \cite{sierra}.
\end{example}

Recall that every inclusion $f \colon M \rightarrow W$ of a smooth closed subvariety $M$ of (complex) codimension $c$ in a smooth variety $W$ induces a topological Gysin map $f^{!} \colon H_{\ast}(W; \mathbb{Q}) \rightarrow H_{\ast - 2c}(M; \mathbb{Q})$.

\begin{defn}\label{definition gysin coherent characteristic classes}
A \emph{Gysin coherent characteristic class $c\ell$ with respect to $\mathcal{X}$} is a pair
$$
c\ell = (c\ell^{\ast}, c\ell_{\ast})
$$
consisting of a function $c\ell^{\ast}$ that assigns to every inclusion $f \colon M \rightarrow W$ of a smooth closed subvariety $M \subset W$ in a smooth variety $W$ an element
$$
c\ell^{\ast}(f) = c\ell^{0}(f) + c\ell^{1}(f) + c\ell^{2}(f) + \dots \in H^{\ast}(M; \mathbb{Q}), \quad c\ell^{p}(f) \in H^{p}(M; \mathbb{Q}),
$$
with $c\ell^{0}(f) = 1$, and a function $c\ell_{\ast}$ that assigns to every inclusion $i \colon X \rightarrow W$ of a compact possibly singular subvariety $X \subset W$ of complex dimension $d$ in a smooth variety $W$ an element
$$
c\ell_{\ast}(i) = c\ell_{0}(i) + c\ell_{1}(i) + c\ell_{2}(i) + \dots + c\ell_{2d}(i) \in H_{\ast}(W; \mathbb{Q}), \quad c\ell_{p}(i) \in H_{p}(W; \mathbb{Q}),
$$
with $c\ell_{2d}(i) = [X]_{W}$, such that the following properties hold:
\begin{enumerate}
\item\label{axiom multiplicativity} \emph{(Multiplicativity)}
For every $i \colon X \rightarrow W$ and $i' \colon X' \rightarrow W'$, we have
$$
c\ell_{\ast}(i \times i') = c\ell_{\ast}(i) \times c\ell_{\ast}(i').
$$
\item\label{axiom isomorphism} \emph{(Isomorphism invariance)}
For every $f \colon M \rightarrow W$ and $f' \colon M' \rightarrow W'$, and every isomorphism $W \stackrel{\cong}{\longrightarrow} W'$ that restricts to an isomorphism $\phi \colon M \stackrel{\cong}{\longrightarrow} M'$, we have
$$
\phi^{\ast}c\ell^{\ast}(f') = c\ell^{\ast}(f).
$$
Moreover, for every $i \colon X \rightarrow W$ and $i' \colon X' \rightarrow W'$, and every isomorphism $\Phi \colon W \stackrel{\cong}{\longrightarrow} W'$ that restricts to an isomorphism $X \stackrel{\cong}{\longrightarrow} X'$, we have
$$
\Phi_{\ast}c\ell_{\ast}(i) = c\ell_{\ast}(i').
$$
\item\label{axiom locality} \emph{(Naturality)}
For every $i \colon X \rightarrow W$ and $f \colon M \rightarrow W$ such that $X \subset M$, the inclusion $i^{M} := i| \colon X \rightarrow M$ satisfies
$$
f_{\ast}c\ell_{\ast}(i^{M}) = c\ell_{\ast}(i).
$$
\item\label{axiom gysin} \emph{(Gysin restriction in a transverse setup)}
There exists a notion of $\mathcal{X}$-transversality such that the following holds.
For every inclusion $i \colon X \rightarrow W$ in $\mathcal{X}$ and every inclusion $f \colon M \rightarrow W$ such that $M$ is irreducible, and $M$ and $X$ are $\mathcal{X}$-transverse in $W$, the inclusion $j \colon Y \rightarrow M$ of the pure-dimensional compact subvariety $Y := M \cap X \subset M$ satisfies
$$
f^{!} c\ell_{\ast}(i) = c\ell^{\ast}(f) \cap c\ell_{\ast}(j).
$$
\end{enumerate}
Such a class $c\ell$ is called \emph{Gysin coherent characteristic class} if $\mathcal{X}$ is the family of all inclusions of compact irreducible subvarieties in smooth varieties.
\end{defn}

\begin{example}
We prove in \Cref{The Goresky-MacPherson L-Class} that the pair $(c\ell^{\ast}, c\ell_{\ast})$ given by $c\ell^{\ast}(f \colon M \hookrightarrow W) = L^{\ast}(\nu_{M \subset W})$ and $c\ell_{\ast}(i \colon X \hookrightarrow W) = i_{\ast}L_{\ast}(X)$, where $L^{\ast}$ is Hirzebruch's cohomological $L$-class and $L_{\ast}$ is the Goresky-MacPherson $L$-class, forms a Gysin coherent characteristic class.
In future work, we plan to discuss other characteristic classes such as Chern classes, Todd classes, as well as motivic Hodge classes in this framework.
Note that the $L$-genus, i.e. the signature, agrees with the genus of $IT_{1 \ast}$ by Saito's intersection Hodge index theorem.
\end{example}

The main result of this paper is the following

\begin{thm}[Uniqueness Theorem]\label{main result on Gysin coherent characteristic classes with respect to x}
Let $c\ell$ and $\widetilde{c\ell}$ be Gysin coherent characteristic classes with respect to $\mathcal{X}$.
If $c\ell^{\ast} = \widetilde{c\ell}^{\ast}$ and $|c\ell_{\ast}| = |\widetilde{c\ell}_{\ast}|$ for the associated genera,
then we have $c\ell_{\ast}(i) = \widetilde{c\ell}_{\ast}(i)$ for all inclusions $i \colon X \rightarrow G$ in $\mathcal{X}$ of compact subvarieties in ambient Grassmannians.
\end{thm}

The proof of this result is provided in \Cref{proof of main theorem} and requires the technique of normally nonsingular expansions developed in the next section.

Note that \Cref{main result on Gysin coherent characteristic classes with respect to x} implies \Cref{main result on Gysin coherent characteristic classes} of the introduction by taking $\mathcal{X}$ to be the family of all inclusions of compact irreducible subvarieties in smooth varieties.
\section{Normally Nonsingular Expansion}\label{Preparatory computations}
In this section, we establish a recursive formula for the computation of Gysin coherent characteristic classes in ambient Grassmannians (see \Cref{main theorem normally nonsingular expansion} below).
Its proof will be provided in \Cref{proof of normally nonsingular expansion}.

Recall that the rational homology groups $H_{\ast}(G; \mathbb{Q})$ of the Grassmannian $G = G_{k}(\mathbb{C}^{m+k})$ determined by integers $m, k \geq 0$ are concentrated in even degrees, and a basis of $H_{2r}(G; \mathbb{Q})$ is given by the set of all fundamental classes $[X_{a}]_{G}$ of Schubert subvarieties $X_{a} \subset G$, $a \in \mathcal{P}(m, k)$, of complex dimension $|a| = r$.
Now let $c \ell$ be a Gysin coherent characteristic class with respect to $\mathcal{X}$.
Then, for any inclusion $i \colon X \hookrightarrow G$ of an irreducible closed algebraic subvariety $X$ of complex dimension $d$ in a Grassmannian $G = G_{k}(\mathbb{C}^{m+k})$ determined by integers $m, k \geq 0$, the only nonzero components of $c \ell_{\ast}(i) \in H_{\ast}(G; \mathbb{Q})$ are the components $c \ell_{2r}(i)$ for $r \in \{0, \dots, d\}$, and we can uniquely write
\begin{align}\label{equation l class linear combination of schubert cycles}
c \ell_{2r}(i) &= \sum_{\substack{a \in \mathcal{P}(m, k), \\ |a| = r}} \lambda^{a}_{i} \cdot [X_{a}]_{G}, \qquad \lambda^{a}_{i} \in \mathbb{Q}.
\end{align}

The main result of this section is the following recursive formula for the family of coefficients $\{\lambda^{a}_{i}\}_{i, a}$ as introduced in (\ref{equation l class linear combination of schubert cycles}) associated to a Gysin coherent characteristic class.

\begin{thm}\label{main theorem normally nonsingular expansion}
Let $i' \colon X' \hookrightarrow G'$ be the inclusion of an irreducible closed algebraic subvariety $X'$ of dimension $d'$ in the Grassmannian $G' = G_{k'}(\mathbb{C}^{m'+k'})$ determined by integers $m', k' \geq 0$.
Fix $a' \in \mathcal{P}(m', k')$ with $|a'| \leq d'$, and set $l := d' - |a'|$.
We suppose that $k' > 0$, and as shown \Cref{figure complementary partition}, we define the integers $m'', k'' \geq 0$ by
\begin{align*}
m'' &:= a_{1}', \\
k'' &:= k' - \operatorname{max} \{t \in \mathbb{Z} | 1 \leq t \leq k', \; a_{t}' = a_{1}'\},
\end{align*}
and the partition $a'' \in \mathcal{P}(m'', k'')$ by $a_{t}'' = m'' - a_{k'+1-t}'$ for $1 \leq t \leq k''$.
Let $i'' \colon X_{a''}(D_{\ast}'') \hookrightarrow G''$ be the inclusion into the Grassmannian $G'' = G_{k''}(\mathbb{C}^{m''+k''})$ of the Schubert subvariety $X_{a''}(D_{\ast}'') \subset G''$ determined by a fixed flag $D_{\ast}''$ on $\mathbb{C}^{m''+k''}$.
Then, given a Gysin coherent characteristic class $c \ell$ with respect to $\mathcal{X}$ such that we have $i' \in \mathcal{X}$, the associated family of coefficients $\{\lambda^{a}_{i}\}_{i, a}$ introduced in (\ref{equation l class linear combination of schubert cycles}) satisfies
\begin{equation}\label{equation main result recuvrsive coefficient formula}
\lambda^{a'}_{i'} = |c\ell_{\ast}|(i', i'') - \sum_{r = 0}^{l} \sum_{\substack{0 \leq r' < l, \\ 0 \leq r'' \leq l, \\ r' + r'' = r}} \sum_{\substack{b' \in \mathcal{P}(m', k'), \\ |b'| = d'-r'}} \sum_{\substack{b'' \in \mathcal{P}(m'', k''), \\ |b''| = |a''|-r''}} \lambda^{b'}_{i'} \lambda^{b''}_{i''} \cdot \langle c \ell^{\ast} \rangle(b', b''),
\end{equation}
where $|c\ell_{\ast}|(i', i'') \in \mathbb{Q}$ will be constructed in \Cref{main theorem signature independent of choices}, and $\langle c \ell^{\ast} \rangle(b', b'') \in \mathbb{Q}$ will be constructed in \Cref{proposition invariance of integral coefficient for recursive formula}.
\end{thm}

\begin{figure}[htbp]
\centering
\fbox{\begin{tikzpicture}
\draw (0, 0) node {\includegraphics[width=0.6\textwidth]{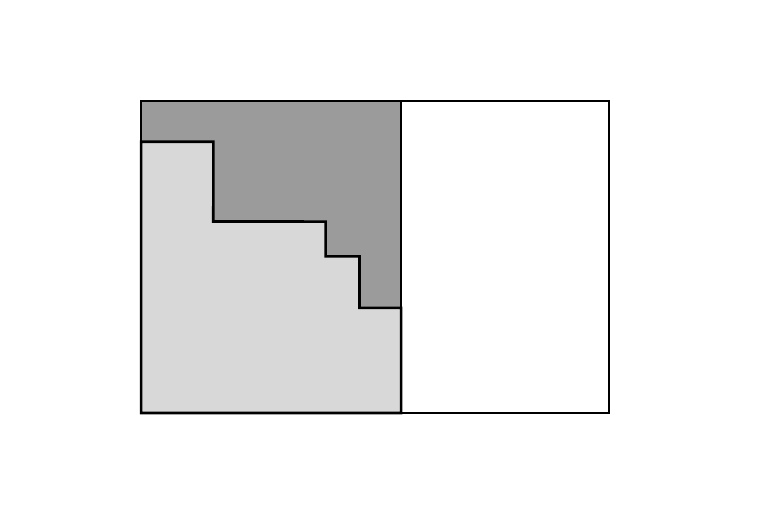}};
\draw (-1.4, -0.6) node {$D_{a'}$};
\draw (-0.9, 1.0) node {$\widetilde{D}_{a''}$};
\draw (0, -2.0) node {$m'$};
\draw (-1.05, 2.0) node {$m''$};
\draw (2.7, 0.0) node {$k'$};
\draw (0.6, 0.55) node {$k''$};
\end{tikzpicture}}
\caption{Definition of the partition $a'' \in \mathcal{P}(m'', k'')$ in terms of $a' \in \mathcal{P}(m', k')$.
The Young diagram $D_{a''}$ is obtained by rotating the shown diagram $\widetilde{D}_{a''}$ by $180$ degrees.}
\label{figure complementary partition}
\end{figure}
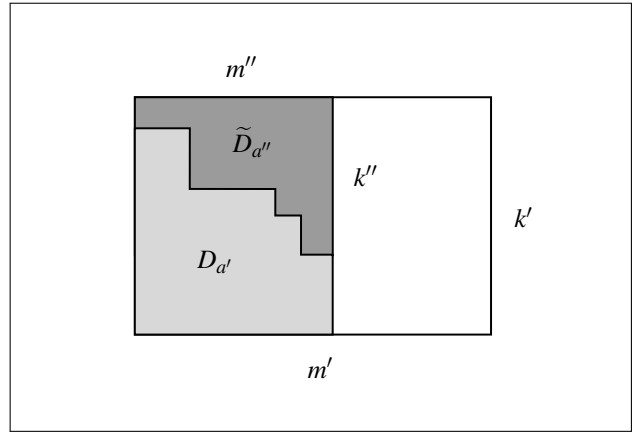

Note that the coefficients $\lambda^{b'}_{i'}$ and $\lambda^{b''}_{i''}$ that appear in (\ref{equation main result recuvrsive coefficient formula}) are recursively known, which will be exploited in the proof of the Uniqueness Theorem (\Cref{main result on Gysin coherent characteristic classes with respect to x}) in \Cref{proof of main theorem} (compare \Cref{figure induction scheme}).
Namely, $\lambda^{b''}_{i''}$ is a coefficient of $c \ell_{2 |b''|}(i'')$, where we note that $i''$ embeds in a Grassmannian of strictly smaller dimension than $i'$.
Furthermore, $\lambda^{b'}_{i'}$ is a coefficient of $c \ell_{2 |b'|}(i')$, where we note that $|b'| > d'-l = a'$ since $r'$ is strictly less than the codimension $l$.

The next section defines the normally nonsingular integration $\langle c \ell^{\ast} \rangle(b', b'') \in \mathbb{Q}$, while \Cref{Characteristic Subvarieties} defines the genera $|c\ell_{\ast}|(i', i'') \in \mathbb{Q}$ of characteristic subvarieties.

\subsection{Normally Nonsingular Integration}\label{Normally Nonsingular Integration}
Given a Gysin coherent characteristic class $c \ell = (c \ell^{\ast}, c \ell_{\ast})$ with respect to $\mathcal{X}$, we construct in \Cref{proposition invariance of integral coefficient for recursive formula} below a map
\begin{equation}\label{equation product integral}
\langle c \ell^{\ast} \rangle \colon \mathcal{P}(m', k') \times \mathcal{P}(m'', k'') \rightarrow \mathbb{Q},
\end{equation}
and show that it does not depend on the choices involved in its construction.
The map (\ref{equation product integral}) and its properties (see \Cref{theorem normally nonsingular integration}) will be important for the proof of our main result in \Cref{proof of normally nonsingular expansion}.
In view of \Cref{remark integral expression}, we may call the map (\ref{equation product integral}) normally nonsingular integration.

First, let us fix some notation that will be used throughout this section.
Let $m', k' \geq 0$ and $m'', k'' \geq 0$ be integers.
Define the integers $m, k \geq 0$ by $m = m'+m''$ and $k=k'+k''$, and the integers $n, n', n'' \geq 0$ by $n = m+k$, $n' = m'+k'$, and $n'' = m''+k''$.
Using the notation introduced in \Cref{Intersection of Schubert varieties}, let $c' = [m' \times k'] \in \mathcal{P}(m', k')$, $c'' = [m'' \times k''] \in \mathcal{P}(m'', k'')$, and $c = c' \sqcup c'' \in \mathcal{P}(m, k)$ (see \Cref{figure proof partitions c}(a)).

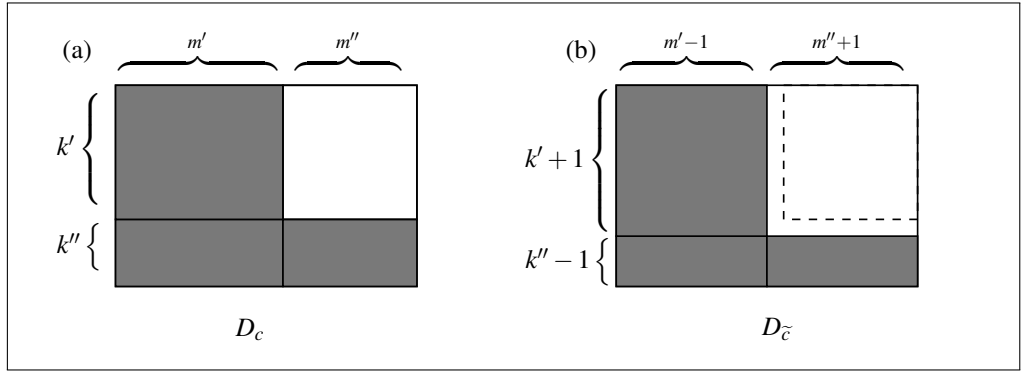
\begin{figure}[htbp]
\centering
\fbox{\begin{tikzpicture}
\draw (0, 0) node {\includegraphics[width=1.0\textwidth]{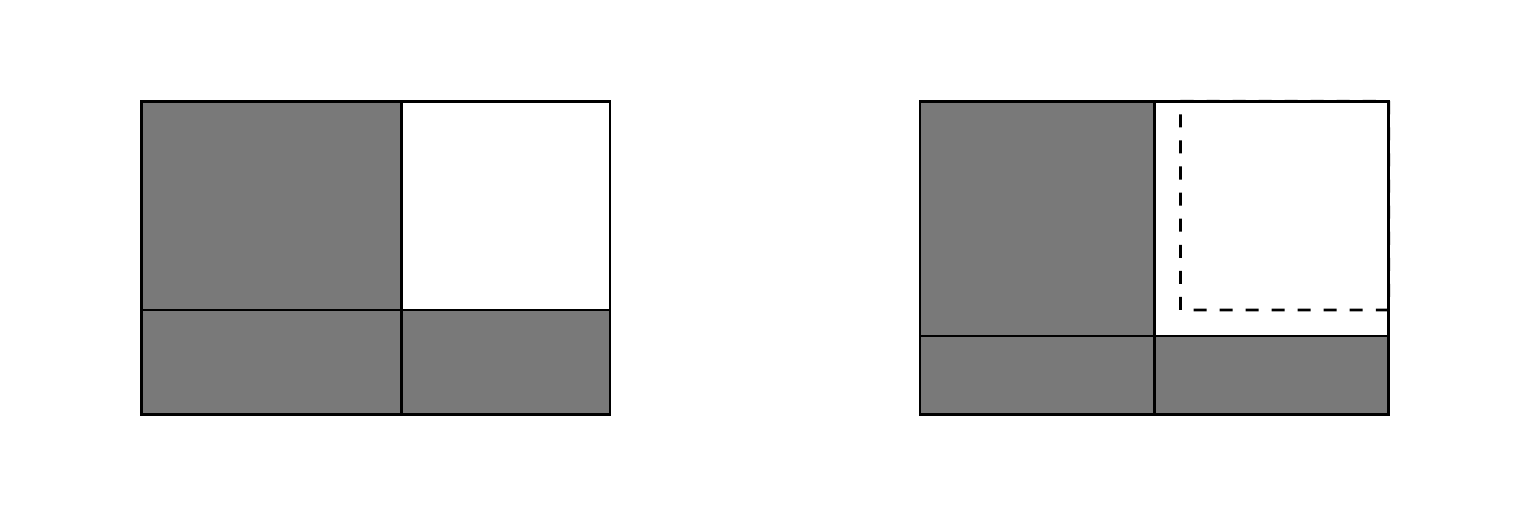}};
\draw (-5.8, 1.8) node {(a)};
\draw (-3.5, -1.9) node {$D_{c}$};
\draw (-5.4, 0.5) node {$k' \begin{cases}
\quad \\
\quad \\
\quad
\end{cases}$};
\draw (-5.4, -0.8) node {$k''
\begin{cases}
\quad
\end{cases}$};
\draw (-4.2, 1.8) node {$\overbrace{\quad \quad \quad \quad \quad \quad}^{m'}$};
\draw (-2.2, 1.8) node {$\overbrace{\quad \quad \quad \quad}^{m''}$};

\draw (0.9, 1.8) node {(b)};
\draw (3.5, -1.9) node {$D_{\widetilde{c}}$};
\draw (1.1, 0.35) node {$k'+1 \begin{cases}
\quad \\
\quad \\
\quad \\
\quad
\end{cases}$};
\draw (1.1, -1.0) node {$k''-1
\begin{cases}
\quad
\end{cases}$};
\draw (2.3, 1.8) node {$\overbrace{\quad \quad \quad \quad \quad}^{m'-1}$};
\draw (4.3, 1.8) node {$\overbrace{\quad \quad \quad \quad \quad}^{m''+1}$};
\end{tikzpicture}}
\caption{(a) Definition of the Young diagram $D_{c}$.
(b) Definition of the Young diagram $D_{\widetilde{c}}$.}
\label{figure proof partitions c}
\end{figure}

Let $G = G_{k}(\mathbb{C}^{n})$, $G' = G_{k'}(\mathbb{C}^{n'})$, and $G'' = G_{k''}(\mathbb{C}^{n''})$.
Given any monomorphisms $\iota' \colon \mathbb{C}^{n'} \rightarrow \mathbb{C}^{n}$ and $\iota'' \colon \mathbb{C}^{n''} \rightarrow \mathbb{C}^{n}$ that satisfy $\mathbb{C}^{n} = \operatorname{im}(\iota') \oplus \operatorname{im}(\iota'')$, we define the Segre product
\begin{equation}\label{equation segre product in terms of monomorphisms}
\Sigma = \Sigma_{\iota', \iota''}^{k', k''} \colon G' \times G'' \rightarrow G, \qquad \Sigma(P', P'') = \iota'(P') \oplus \iota''(P''),
\end{equation}
of $G'$ and $G''$ in the ambient Grassmannian $G$.
Note that the above Segre product $\Sigma$ is related to the Segre product $S$ of \Cref{proposition segre product of schubert varieties} by the commutative diagram
\begin{equation}\label{equation compatibility of Segre products}
\xymatrix{
G' \times G'' \ar[rr]_{\cong}^{\Phi' \times \Phi''} \ar@{^{(}->}[drr]_{\Sigma^{k', k''}_{\iota', \iota''}}^{(\ref{equation segre product in terms of monomorphisms})} & & G_{k'}(V') \times G_{k''}(V'') \ar@{^{(}->}[d]^{S^{k', k''}_{V', V''}}_{(\ref{equation segre product})} \\
 & & G,
}
\end{equation}
where $V' = \operatorname{im}(\iota')$ and $V'' = \operatorname{im}(\iota'')$, and $\Phi' \colon G' = G_{k'}(\mathbb{C}^{n'}) \cong G_{k'}(V')$ and $\Phi'' \colon G'' = G_{k''}(\mathbb{C}^{n''}) \cong G_{k''}(V'')$ are the isomorphisms given by $P' \mapsto \iota'(P')$ and $P'' \mapsto \iota''(P'')$, respectively.

In the following, by a homotopy of monomorphisms $U \rightarrow V$ of complex vector spaces we mean a continuous map $\alpha \colon [0, 1] \times U \rightarrow V$, $(t, v) \mapsto \alpha(t, u) =: \alpha_{t}(u)$, such that $\alpha_{t} \colon U \rightarrow V$ is a monomorphism for all $t \in [0, 1]$.

\begin{lemma}\label{lemma homotopy of monomorphisms and induced segre products}
If $\alpha'$ is a homotopy of monomorphisms $\mathbb{C}^{n'} \rightarrow \mathbb{C}^{n}$ and $\alpha''$ is a homotopy of monomorphisms $\mathbb{C}^{n''} \rightarrow \mathbb{C}^{n}$ such that $\mathbb{C}^{n} = \operatorname{im}(\alpha_{t}') \oplus \operatorname{im}(\alpha_{t}'')$ for all $t \in [0, 1]$, then the associated Segre products $\Sigma_{\alpha_{t}', \alpha_{t}''}^{k', k''}$ defined in (\ref{equation segre product in terms of monomorphisms}) form a homotopy
$$
\Sigma_{\alpha', \alpha''}^{k', k''} \colon [0, 1] \times G' \times G'' \rightarrow G, \qquad (t, P', P'') \mapsto \Sigma_{\alpha_{t}', \alpha_{t}''}^{k', k''}(P', P'').
$$
\end{lemma}

\begin{proof}
Let $\alpha$ be the homotopy of monomorphisms $\mathbb{C}^{n'} \oplus \mathbb{C}^{n''} \rightarrow \mathbb{C}^{n}$ given by $(t, v', v'') \mapsto \alpha'_{t}(v') + \alpha''_{t}(v'')$.
Note that $\mathbb{C}^{n} = \operatorname{im}(\alpha_{t}') \oplus \operatorname{im}(\alpha_{t}'')$ implies that $\alpha_{t}$ is an isomorphism for all $t \in [0, 1]$.
Hence, $t \mapsto g_{t} = \alpha_{t} \circ \alpha_{0}^{-1}$ is a continuous path in $GL_{n}(\mathbb{C})$ such that $g_{0} = \operatorname{id}_{\mathbb{C}^{n}}$.
It follows that the map $\Sigma_{\alpha', \alpha''}^{k', k''}$ is given by
$$
(t, P', P'') \mapsto \Sigma_{\alpha_{t}', \alpha_{t}''}^{k', k''}(P', P'') = \alpha_{t}(P', P'') = g_{t} \cdot \alpha_{0}(P', P'') = g_{t} \cdot \Sigma_{\alpha_{0}', \alpha_{0}''}^{k', k''}(P', P'').
$$
Then, the claim follows by noting that $GL_{n}(\mathbb{C})$ acts topologically on $G = G_{k}(\mathbb{C}^{n})$.
\end{proof}

\begin{prop}\label{proposition segre avoids singular locus}
Let $F_{\ast}$ be a flag on $\mathbb{C}^{n}$.
We define the smooth irreducible quasiprojective complex algebraic variety $W = P \setminus Z$, where $Z$ denotes the singular locus of the Schubert subvariety $X_{c}(F_{\ast}) \subset G$.
Then, the Segre product $\Sigma = \Sigma_{\iota', \iota''}^{k', k''} \colon G' \times G'' \rightarrow G$ defined in (\ref{equation segre product in terms of monomorphisms}) restricts under the condition $\mathbb{C}^{n} = F_{n'} \oplus \operatorname{im}(\iota'')$ to a closed embedding $\Sigma^{W} \colon G' \times G'' \rightarrow W$.
\end{prop}

\begin{proof}
We set $V' = \operatorname{im}(\iota')$ and $V'' = \operatorname{im}(\iota'')$.
The assumption $\mathbb{C}^{n} = F_{n'} \oplus V''$ implies that we can define a flag $F_{\ast}''$ on $V''$ by setting $F_{w}'' = V'' \cap F_{m'+k'+w}$ for $0 \leq w \leq m''+k''$, and a flag $F_{\ast}'$ on $V'$ by setting $F_{w}' = \pi_{2}(F_{w})$ for $0 \leq w \leq m'+k'$, where $\pi_{2} \colon \mathbb{C}^{n} = V'' \oplus V' \rightarrow V'$ denotes the projection to the second summand.
Let $E_{\ast}''$ be a flag on $V''$ that is transverse to $F_{\ast}''$, and let $E_{\ast}'$ be a flag on $V'$ that is transverse to $F_{\ast}'$.
Using the notation introduced in \Cref{Intersection of Schubert varieties}, we define the flag $E_{\ast} := E_{\ast}'' \oplus E_{\ast}'$ on $\mathbb{C}^{n} = V'' \oplus V'$.
By construction, the flags $E_{\ast}$ and $F_{\ast}$ are transverse.
By \Cref{proposition segre product of schubert varieties}, the Segre product $S = S^{k', k''}_{V', V''}$ satisfies
\begin{equation*}
S((X_{c'}(E_{\ast}') \cap X_{c'}(E_{\ast}')) \times (X_{c''}(E_{\ast}'') \cap X_{c''}(E_{\ast}''))) = X_{c' \sqcup c''}(E_{\ast}' \oplus E_{\ast}'') \cap X_{c'' \sqcup c'}(E_{\ast}'' \oplus E_{\ast}'),
\end{equation*}
that is, $S(G_{k'}(V') \times G_{k''}(V'')) \subset X_{c'' \sqcup c'}(E_{\ast})$.
Therefore, the commutative diagram (\ref{equation compatibility of Segre products}) yields $\Sigma(G' \times G'') \subset X_{c'' \sqcup c'}(E_{\ast})$.
Hence, to prove the claim, it suffices to show that
\begin{equation}\label{equation product embedded x avoids singular locus}
X_{c'' \sqcup c'}(E_{\ast}) \cap Z = \varnothing.
\end{equation}
By a result of Lakshmibai-Weyman (see \Cref{thm.lakwey}), the singular locus $Z$ of the Schubert subvariety $X_{c}(F_{\ast}) \subset G$ is given by the Schubert subvariety $Z = X_{\widetilde{c}}(F_{\ast}) \subset G$, where
$$
\widetilde{c} = [(m'-1) \times (k'+1)] \sqcup [(m''+1) \times (k''-1)]
$$
(see \Cref{figure proof partitions c}(b)).
Since $Z = \varnothing$ for $m' = 0$ or $k'' = 0$, we may assume in the following that $m', k'' \geq 1$.
Since the flags $E_{\ast}$ and $F_{\ast}$ are transverse by construction, we may apply \Cref{proposition empty intersection of schubert varieties} to obtain the claim (\ref{equation product embedded x avoids singular locus}) provided there exists $1 \leq t_{0} \leq k$ such that
$$
(c'' \sqcup c')_{t_{0}} + \widetilde{c}_{k+1-t_{0}} < m.
$$
Indeed, for $t_{0} = k'+1$, we have $(c'' \sqcup c')_{t_{0}} = (c'' \sqcup c')_{k'+1} = c''_{1} = m''$ and $\widetilde{c}_{k+1-t_{0}} = \widetilde{c}_{k''} = m'-1$, so that $(c'' \sqcup c')_{t_{0}} + \widetilde{c}_{k+1-t_{0}} = m'' + (m'-1) < m$, and the claim follows.
\end{proof}

\begin{lemma}\label{lemma homotopy of linear subspaces}
Let $0 \leq p \leq q$ be integers.
Suppose that $V \subset \mathbb{C}^{q}$ is a linear subspace and $\iota_{0}, \iota_{1} \colon \mathbb{C}^{p} \rightarrow \mathbb{C}^{q}$ are monomorphisms satisfying $\mathbb{C}^{q} = V \oplus \operatorname{im}(\iota_{0}) = V \oplus \operatorname{im}(\iota_{1})$.
Then, there exists a homotopy $\alpha$ of monomorphisms $\mathbb{C}^{p} \rightarrow \mathbb{C}^{q}$ between $\alpha_{0} = \iota_{0}$ and $\alpha_{1} = \iota_{1}$ such that $\mathbb{C}^{q} = V \oplus \operatorname{im}(\alpha_{t})$ for all $t \in [0, 1]$.
\end{lemma}

\begin{proof}
Let $e_{1}^{(r)}, \dots, e_{r}^{(r)}$ be the standard basis of $\mathbb{C}^{r}$.
By applying a linear automorphism of $\mathbb{C}^{q}$, we may assume without loss of generality that $\iota_{0}(e_{i}^{(p)}) = e_{i}^{(q)}$ for $i=1, \dots, p$, and $V$ is spanned by $e_{p+1}^{(q)}, \dots, e_{q}^{(q)}$.
Then, according to \cite[p. 193]{gh}, the linear subspaces $U \subset \mathbb{C}^{q}$ satisfying $U \oplus V = \mathbb{C}^{q}$ are in bijection with complex $p \times (q-p)$ matrices $A$, where the matrix $A$ corresponds to the linear subspace $U \subset \mathbb{C}^{q}$ spanned by the $p$ row vectors of the $p \times q$ matrix $[1_{p} \;\; A]$.
Let $A_{0}$ and $A_{1}$ be the matrices corresponding to $U_{0} := \operatorname{im}(\iota_{0})$ and $U_{1} := \operatorname{im}(\iota_{1})$, respectively.
It follows from $U_{0} = \operatorname{im}(\iota_{0}) = \langle e_{1}^{(q)}, \dots, e_{p}^{(q)}\rangle$ that $A_{0} = [0]$.
Thus,
$$
\beta \colon [0, 1] \times \mathbb{C}^{p} \rightarrow \mathbb{C}^{q}, \qquad (t, v) \mapsto \beta_{t}(v) := \beta(t, v) = v \cdot [1_{p} \;\; t \cdot A_{1}]
$$
is a homotopy of monomorphisms $\mathbb{C}^{p} \rightarrow \mathbb{C}^{q}$ satisfying $\beta_{0} = \iota_{0}$ and $\operatorname{im}(\beta_{t}) \oplus V = \mathbb{C}^{q}$ for all $t \in [0, 1]$.
Since we have $\operatorname{im}(\beta_{1}) = U_{1} = \operatorname{im}(\iota_{1})$ by construction, we can consider $g := (\beta_{1})^{-1} \circ \iota_{1} \in GL_{q}(\mathbb{C})$.
As $GL_{q}(\mathbb{C})$ is path connected, we may choose a continuous family $g_{t} \in GL_{q}(\mathbb{C})$, $t \in [0, 1]$, such that $g_{0}$ is the identity map on $\mathbb{C}^{q}$, and $g_{1} = g$.
Thus,
$$
\gamma \colon [0, 1] \times \mathbb{C}^{p} \rightarrow \mathbb{C}^{q}, \qquad (t, v) \mapsto \gamma_{t}(v) := \gamma(t, v) = \beta_{1}(g_{t}v)
$$
is a homotopy of monomorphisms $\mathbb{C}^{p} \rightarrow \mathbb{C}^{q}$ satisfying $\gamma_{0} = \beta_{1}$, $\gamma_{1} = \iota_{1}$, and $\operatorname{im}(\gamma_{t}) = \operatorname{im}(\beta_{1}) = U_{1}$ for all $t \in [0, 1]$.
In particular, we have $\operatorname{im}(\gamma_{t}) \oplus V = \mathbb{C}^{q}$ for all $t \in [0, 1]$.
Finally, the concatenation $\alpha$ of the homotopies $\beta$ and $\gamma$ has the desired properties.
\end{proof}

\begin{prop}\label{proposition homotopic monomorphisms yield homotopic segre restrictions}
In \Cref{proposition segre avoids singular locus}, the homotopy class of the map $\Sigma^{W}$ does not depend on the choice of the monomorphisms $\iota'$ and $\iota''$.
\end{prop}

\begin{proof}
We consider two pairs of monomorphisms $\iota'_{0}, \iota'_{1} \colon \mathbb{C}^{n'} \rightarrow \mathbb{C}^{n}$ and $\iota''_{0}, \iota''_{1} \colon \mathbb{C}^{n''} \rightarrow \mathbb{C}^{n}$ satisfying $\mathbb{C}^{n} = \operatorname{im}(\iota_{i}') \oplus \operatorname{im}(\iota_{i}'')$ and $\mathbb{C}^{n} = F_{n'} \oplus \operatorname{im}(\iota_{i}'')$ for $i = 0, 1$.
In view of \Cref{lemma homotopy of monomorphisms and induced segre products}, it suffices to construct a homotopy $\alpha'$ of monomorphisms $\mathbb{C}^{n'} \rightarrow \mathbb{C}^{n}$ between $\alpha'_{0} = \iota'_{0}$ and $\alpha'_{1} = \iota'_{1}$, and a homotopy $\alpha''$ of monomorphisms $\mathbb{C}^{n''} \rightarrow \mathbb{C}^{n}$ between $\alpha''_{0} = \iota''_{0}$ and $\alpha''_{1} = \iota''_{1}$ such that $\mathbb{C}^{n} = \operatorname{im}(\alpha_{t}') \oplus \operatorname{im}(\alpha_{t}'')$ and $\mathbb{C}^{n} = F_{n'} \oplus \operatorname{im}(\alpha_{t}'')$ for all $t \in [0, 1]$.

Let $\theta' \colon \mathbb{C}^{n'} \rightarrow \mathbb{C}^{n}$ be a monomorphism such that $\operatorname{im}(\theta') = F_{n'}$.

First, we use \Cref{lemma homotopy of linear subspaces} to find a homotopy $\beta'$ of monomorphisms $\mathbb{C}^{n'} \rightarrow \mathbb{C}^{n}$ between $\beta'_{0} = \iota'_{0}$ and $\beta'_{1} = \theta'$ such that $\mathbb{C}^{n} = \operatorname{im}(\beta_{t}') \oplus \operatorname{im}(\iota_{0}'')$ for all $t \in [0, 1]$.
Let $\beta''$ be the constant homotopy of monomorphisms $\iota_{0}'' \colon \mathbb{C}^{n''} \rightarrow \mathbb{C}^{n}$.
Thus, we have $\mathbb{C}^{n} = \operatorname{im}(\beta_{t}') \oplus \operatorname{im}(\beta_{t}'')$ and $\mathbb{C}^{n} = F_{n'} \oplus \operatorname{im}(\beta_{t}'')$ for all $t \in [0, 1]$.

Second, we use \Cref{lemma homotopy of linear subspaces} to find a homotopy $\gamma''$ of monomorphisms $\mathbb{C}^{n''} \rightarrow \mathbb{C}^{n}$ between $\gamma''_{0} = \iota''_{0}$ and $\gamma''_{1} = \iota''_{1}$ such that $\mathbb{C}^{n} = F_{n'} \oplus \operatorname{im}(\gamma_{t}'')$ for all $t \in [0, 1]$.
Let $\gamma'$ be the constant homotopy of monomorphisms $\theta' \colon \mathbb{C}^{n'} \rightarrow \mathbb{C}^{n}$.
Thus, we have $\mathbb{C}^{n} = \operatorname{im}(\gamma_{t}') \oplus \operatorname{im}(\gamma_{t}'')$ and $\mathbb{C}^{n} = F_{n'} \oplus \operatorname{im}(\gamma_{t}'')$ for all $t \in [0, 1]$.

Third, we use \Cref{lemma homotopy of linear subspaces} to find a homotopy $\delta'$ of monomorphisms $\mathbb{C}^{n'} \rightarrow \mathbb{C}^{n}$ between $\delta'_{0} = \theta'$ and $\delta'_{1} = \iota'_{1}$ such that $\mathbb{C}^{n} = \operatorname{im}(\delta_{t}') \oplus \operatorname{im}(\iota_{1}'')$ for all $t \in [0, 1]$.
Let $\delta''$ be the constant homotopy of monomorphisms $\iota_{1}'' \colon \mathbb{C}^{n''} \rightarrow \mathbb{C}^{n}$.
Thus, we have $\mathbb{C}^{n} = \operatorname{im}(\delta_{t}') \oplus \operatorname{im}(\delta_{t}'')$ and $\mathbb{C}^{n} = F_{n'} \oplus \operatorname{im}(\delta_{t}'')$ for all $t \in [0, 1]$.

Finally, the desired homotopies $\alpha'$ and $\alpha''$ are the concatenations of the homotopies $\beta'$, $\gamma'$, $\delta'$, and $\beta''$, $\gamma''$, $\delta''$, respectively.
\end{proof}

\begin{prop}\label{proposition invariance of integral coefficient for recursive formula}
Let $F_{\ast}$ be a flag on $\mathbb{C}^{n}$.
Let $Z$ denote the singular set of the Schubert subvariety $X_{c}(F_{\ast}) \subset G$.
Then, we define the smooth irreducible quasiprojective complex algebraic variety $W = G \setminus Z$, and the smooth irreducible closed subvariety $M = X_{c}(F_{\ast}) \setminus Z \subset W$ given by the set of nonsingular points of $X_{c}(F_{\ast})$.
Let $f^{!} \colon H_{\ast}(W) \rightarrow H_{\ast-2k'm''}(M)$ denote the Gysin map associated to the inclusion $f \colon M \hookrightarrow W$, where we note that $M \subset W$ is a smooth submanifold of (real) codimension $2k'm''$ that is closed as a subset.
Moreover, let $\Sigma_{\ast}^{W} \colon H_{\ast}(G' \times G'') \rightarrow H_{\ast}(W)$ be the map induced on homology by the map $\Sigma^{W}$ introduced in \Cref{proposition segre avoids singular locus}, where we note that according to \Cref{proposition homotopic monomorphisms yield homotopic segre restrictions}, $\Sigma_{\ast}^{W}$ does not depend on the choice of the monomorphisms $\iota'$ and $\iota''$ employed in the Segre product $\Sigma = \Sigma_{\iota', \iota''}^{k', k''} \colon G' \times G'' \rightarrow G$ defined in (\ref{equation segre product in terms of monomorphisms}).
Let $c \ell = (c \ell^{\ast}, c \ell_{\ast})$ be a Gysin coherent characteristic class with respect to $\mathcal{X}$.
Then, for $b' \in \mathcal{P}(m', k')$ and $b'' \in \mathcal{P}(m'', k'')$, the expression
\begin{equation}\label{equation integral constant in recursive formula}
\langle c \ell^{\ast} \rangle(b', b'') := \langle c\ell^{\ast}(f)^{-1}, f^{!} \Sigma_{\ast}^{W}([X_{b'}]_{G'} \times [X_{b''}]_{G''})\rangle \in \mathbb{Q}
\end{equation}
is independent of the choice of the flag $F_{\ast}$.
\end{prop}

\begin{proof}
Let $\widetilde{F}_{\ast}$ be another flag on $\mathbb{C}^{n}$.
Let $\widetilde{Z}$ denote the singular set of the Schubert subvariety $X_{c}(\widetilde{F}_{\ast}) \subset G$.
Then, we define the smooth irreducible quasiprojective complex algebraic variety $\widetilde{W} = G \setminus \widetilde{Z}$, and the smooth irreducible closed subvariety $\widetilde{M} = X_{c}(\widetilde{F}_{\ast}) \setminus \widetilde{Z} \subset \widetilde{W}$ given by the set of nonsingular points of $X_{c}(\widetilde{F}_{\ast})$.
Let $\widetilde{f}^{!} \colon H_{\ast}(\widetilde{W}) \rightarrow H_{\ast-2k'm''}(\widetilde{M})$ denote the Gysin map associated to the inclusion $\widetilde{f} \colon \widetilde{M} \hookrightarrow \widetilde{W}$.
Fix monomorphisms $\iota' \colon \mathbb{C}^{n'} \rightarrow \mathbb{C}^{n}$ and $\iota'' \colon \mathbb{C}^{n''} \rightarrow \mathbb{C}^{n}$ such that $\mathbb{C}^{n} = \operatorname{im}(\iota') \oplus \operatorname{im}(\iota'')$, $\mathbb{C}^{n} = F_{n'} \oplus \operatorname{im}(\iota'')$, and $\mathbb{C}^{n} = \widetilde{F}_{n'} \oplus \operatorname{im}(\iota'')$.
Let $\Sigma = \Sigma_{\iota', \iota''}^{k', k''} \colon G' \times G'' \rightarrow G$ be the associated Segre product defined in (\ref{equation segre product in terms of monomorphisms}).
Since $\mathbb{C}^{n} = F_{n'} \oplus \operatorname{im}(\iota'')$ and $\mathbb{C}^{n} = \widetilde{F}_{n'} \oplus \operatorname{im}(\iota'')$, the map $\Sigma$ restricts by \Cref{proposition segre avoids singular locus} to maps $\Sigma^{W} \colon G' \times G'' \rightarrow W$ and $\Sigma^{\widetilde{W}} \colon G' \times G'' \rightarrow \widetilde{W}$.
We fix $g \in GL_{n}(\mathbb{C})$ such that $\widetilde{F}_{\ast} = g \cdot F_{\ast}$.
Then, the isomorphism $\Gamma \colon G \cong G$, $P \mapsto g \cdot P$, induced by $g$ restricts to isomorphisms $\Gamma_{M} := \Gamma| \colon M \cong \widetilde{M}$ and $\Gamma_{W} := \Gamma| \colon W \cong \widetilde{W}$ such that the diagram
\[ \xymatrix{
M \ar[r]_{f} \ar[d]_{\Gamma_{M}}^{\cong} & W \ar[d]_{\Gamma_{W}}^{\cong} \\
\widetilde{M} \ar[r]_{\widetilde{f}} & \widetilde{W}
} \]
commutes.
By axiom (\ref{axiom isomorphism}) for the ambient characteristic class $c \ell$, we have $\Gamma_{M}^{\ast} c\ell^{\ast}(\widetilde{f})^{-1} = c\ell^{\ast}(f)^{-1}$.
Moreover, we have $\Gamma_{M \ast}f^{!} = \widetilde{f}^{!} \Gamma_{W \ast}$ by \Cref{proposition gysin compatible with inclusions}.
We define the monomorphisms $\widetilde{\iota}' := g \cdot \iota' \colon \mathbb{C}^{n'} \rightarrow \mathbb{C}^{n}$ and $\widetilde{\iota}'' := g \cdot \iota'' \colon \mathbb{C}^{n''} \rightarrow \mathbb{C}^{n}$.
Let $\widetilde{\Sigma} := \Sigma_{\widetilde{\iota}', \widetilde{\iota}''}^{k', k''} \colon G' \times G'' \rightarrow G$ be the Segre product defined in (\ref{equation segre product in terms of monomorphisms}) associated to the monomorphisms $\widetilde{\iota}'$ and $\widetilde{\iota}''$, where we note that $\mathbb{C}^{n} = \operatorname{im}(\widetilde{\iota}') \oplus \operatorname{im}(\widetilde{\iota}'')$ follows from $\mathbb{C}^{n} = \operatorname{im}(\iota') \oplus \operatorname{im}(\iota'')$.
By construction, we have $\widetilde{\Sigma} = \Gamma \circ \Sigma$.
Since $\mathbb{C}^{n} = F_{n'} \oplus \operatorname{im}(\iota'')$ implies $\mathbb{C}^{n} = \widetilde{F}_{n'} \oplus \operatorname{im}(\widetilde{\iota}'')$, the map $\widetilde{\Sigma}$ restricts by \Cref{proposition segre avoids singular locus} to a map $\widetilde{\Sigma}^{\widetilde{W}} \colon G' \times G'' \rightarrow \widetilde{W}$.
Hence, it follows that $\widetilde{\Sigma}^{\widetilde{W}} = \Gamma_{W} \circ \Sigma^{W}$.

Altogether, the expression (\ref{equation integral constant in recursive formula}) does not depend on the choice of the flag $F_{\ast}$ because
\begin{align*}
\langle c\ell^{\ast}(f)^{-1}, f^{!} \Sigma_{\ast}^{W}([X_{b'}]_{G'} \times [X_{b''}]_{G''})\rangle &= \langle \Gamma_{M}^{\ast} c\ell^{\ast}(\widetilde{f})^{-1}, f^{!} \Sigma_{\ast}^{W}([X_{b'}]_{G'} \times [X_{b''}]_{G''})\rangle \\
&= \langle c\ell^{\ast}(\widetilde{f})^{-1}, \Gamma_{M \ast}f^{!} \Sigma^{W}_{\ast}([X_{b'}]_{G'} \times [X_{b''}]_{G''})\rangle \\
&= \langle c\ell^{\ast}(\widetilde{f})^{-1}, \widetilde{f}^{!} \Gamma_{W \ast} \Sigma^{W}_{\ast}([X_{b'}]_{G'} \times [X_{b''}]_{G''})\rangle \\
&= \langle c\ell^{\ast}(\widetilde{f})^{-1}, \widetilde{f}^{!} \widetilde{\Sigma}^{\widetilde{W}}_{\ast}([X_{b'}]_{G'} \times [X_{b''}]_{G''})\rangle \\
&= \langle c\ell^{\ast}(\widetilde{f})^{-1}, \widetilde{f}^{!} \Sigma^{\widetilde{W}}_{\ast}([X_{b'}]_{G'} \times [X_{b''}]_{G''})\rangle
\end{align*}
where the last equality holds because $\widetilde{\Sigma}^{\widetilde{W}}_{\ast} = \Sigma^{\widetilde{W}}_{\ast}$ by \Cref{proposition homotopic monomorphisms yield homotopic segre restrictions}.

This completes the proof of \Cref{proposition invariance of integral coefficient for recursive formula}.
\end{proof}

The following proposition enables us to pick out distinguished partitions in intersection products by considering triple intersections.

\begin{prop}\label{theorem triple intersection}
We assume that $k' > 0$.
Suppose that $a' \in \mathcal{P}(m', k')$ is given such that $a_{1}' = m''$ and $\operatorname{max} \{t \in \mathbb{Z} | 1 \leq t \leq k, \; a_{t}' = a_{1}'\} = k' - k''$.
Let $a'' \in \mathcal{P}(m'', k'')$ be defined by $a_{t}'' = m'' - a_{k'+1-t}'$ for $1 \leq t \leq k''$ (see \Cref{figure complementary partition}).
We also set $a = a'' \sqcup c' \in \mathcal{P}(m, k)$ (see \Cref{figure proof partitions ab}(a)).
Then, given $b' \in \mathcal{P}(m', k')$ with $|b'| = |a'|$, the element $b = b' \sqcup c'' \in \mathcal{P}(m, k)$ (see \Cref{figure proof partitions ab}(b)) satisfies
\begin{align*}
[X_{a}]_{G} \cdot [X_{b}]_{G} \cdot [X_{c}]_{G} = \delta_{a'b'} \cdot [\operatorname{pt}]_{G},
\end{align*}
where $\delta$ is the Kronecker delta given by $\delta_{a'b'} = 1$ for $a' = b'$, and $\delta_{a'b'} = 0$ else.
\end{prop}

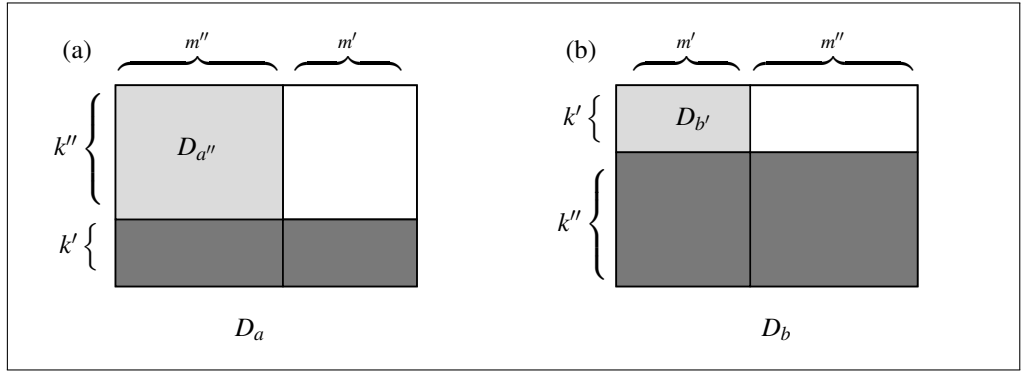
\begin{figure}[htbp]
\centering
\fbox{\begin{tikzpicture}
\draw (0, 0) node {\includegraphics[width=1.0\textwidth]{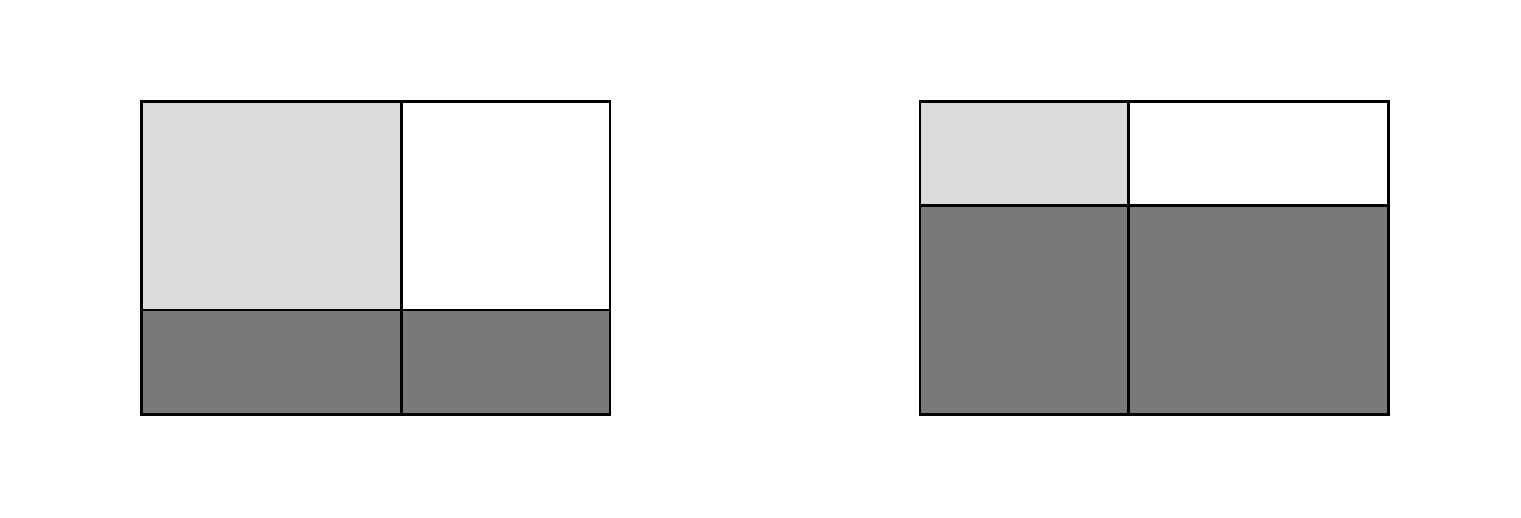}};
\draw (-5.8, 1.8) node {(a)};
\draw (-4.2, 0.5) node {$D_{a''}$};
\draw (-3.5, -1.9) node {$D_{a}$};
\draw (-5.4, 0.5) node {$k'' \begin{cases}
\quad \\
\quad \\
\quad
\end{cases}$};
\draw (-5.4, -0.8) node {$k'
\begin{cases}
\quad
\end{cases}$};
\draw (-4.2, 1.8) node {$\overbrace{\quad \quad \quad \quad \quad \quad}^{m''}$};
\draw (-2.2, 1.8) node {$\overbrace{\quad \quad \quad \quad}^{m'}$};

\draw (0.9, 1.8) node {(b)};
\draw (2.4, 0.9) node {$D_{b'}$};
\draw (3.5, -1.9) node {$D_{b}$};
\draw (1.3, -0.5) node {$k'' \begin{cases}
\quad \\
\quad \\
\quad
\end{cases}$};
\draw (1.3, 0.9) node {$k'
\begin{cases}
\quad
\end{cases}$};
\draw (2.3, 1.8) node {$\overbrace{\quad \quad \quad \quad}^{m'}$};
\draw (4.25, 1.8) node {$\overbrace{\quad \quad \quad \quad \quad \quad}^{m''}$};
\end{tikzpicture}}
\caption{(a) Definition of the Young diagram $D_{a}$ in terms of the Young diagram $D_{a''}$.
(b) Definition of the Young diagram $D_{b}$ in terms of the Young diagram $D_{b'}$.}
\label{figure proof partitions ab}
\end{figure}

\begin{proof}
Let $a^{\ast} \in \mathcal{P}(m'', k'')$ be given by $a_{t}^{\ast} = a_{k-k'+t}'$ for $1 \leq t \leq k''$.
Then, by construction,
\begin{equation}\label{equation complementary partitions}
a_{t}'' = m'' - a_{k'+1-t}^{\ast}, \quad 1 \leq t \leq k''.
\end{equation}

Writing $[-]_{G} = [-]$, we have
\begin{align*}
[X_{a}] \cdot [X_{b}] \cdot [X_{c}] &= [X_{a'' \sqcup [m' \times k']}] \cdot [X_{b' \sqcup [m'' \times k'']}] \cdot [X_{[m' \times k'] \sqcup [m'' \times k'']}] \\
&\stackrel{(\ref{step 1})}{=} [X_{[m'' \times k''] \sqcup [m' \times k']}] \cdot [X_{b' \sqcup [m'' \times k'']}] \cdot [X_{[m' \times k'] \sqcup a''}] \\
&\stackrel{(\ref{step 2})}{=} \delta_{a'b'} \cdot [X_{a' \sqcup [m'' \times k'']}] \cdot [X_{[m'' \times k''] \sqcup [m' \times k']}] \cdot [X_{[m' \times k'] \sqcup a''}] \\
&\stackrel{(\ref{step 3})}{=} \delta_{a'b'} \cdot [X_{a^{\ast} \sqcup [m' \times k'']_{m', k'}}] \cdot [X_{[m' \times k'] \sqcup a''}] \cdot [X_{[m'' \times k''] \sqcup [m' \times k']}] \\
&\stackrel{(\ref{step 4})}{=} \delta_{a'b'} \cdot [X_{[m' \times k'']_{m, k}}] \cdot [X_{[m'' \times k''] \sqcup [m' \times k']}] \\
&\stackrel{(\ref{step 5})}{=} \delta_{a'b'} \cdot [\operatorname{pt}],
\end{align*}
where the individual steps are explained in the following.

Since $[X_{a''}]_{G''} \cdot [X_{[m'' \times k'']}]_{G''} = [X_{[m'' \times k'']}]_{G''} \cdot [X_{a''}]_{G''}$, the Intersection Box Extension Principle (\Cref{corollary twisting rule}) implies that
\begin{align}\label{step 1}
[X_{a'' \sqcup [m' \times k']}] \cdot [X_{[m' \times k'] \sqcup [m'' \times k'']}] = [X_{[m'' \times k''] \sqcup [m' \times k']}] \cdot [X_{[m' \times k'] \sqcup a''}].
\end{align}

Next, let us show that
\begin{align}\label{step 2}
[X_{b' \sqcup [m'' \times k'']}] \cdot [X_{[m' \times k'] \sqcup a''}] = \delta_{a'b'} \cdot [X_{a' \sqcup [m'' \times k'']}] \cdot [X_{[m' \times k'] \sqcup a''}].
\end{align}
By \Cref{corollary empty intersection of schubert varieties}, it suffices to show the implication
\begin{align}\label{step 2 implication}
(b' \sqcup [m'' \times k''])_{k+1-t} + ([m' \times k'] \sqcup a'')_{t} \geq m \quad \forall \; 1 \leq t \leq k' \quad \Rightarrow \quad a' = b'.
\end{align}
For this purpose, note that
$$
(b' \sqcup [m'' \times k''])_{k+1-t} + ([m' \times k'] \sqcup a'')_{t} = b_{k'+1-t}' + m' + a''_{t}, \qquad 1 \leq t \leq k',
$$
where we wrote $a_{t}'' := 0$ for $t > k''$.
Hence, if the assumption in (\ref{step 2 implication}) holds, then summation over $1 \leq t \leq k'$ yields
$$
k'm \leq |b'| + k'm' + |a''| = k'm,
$$
where we have used that $|b'| + |a''| = |a'| + |a''| = k'm''$ by construction of $a''$.
Consequently, we obtain $b_{k'+1-t}' + m' + a''_{t} = m$ for all $1 \leq t \leq k'$.
Hence, writing $a_{k''+1-t}^{\ast} := m'' \; (= a_{k'+1-t}')$ for $t > k''$, we get
$$
b_{k'+1-t}' = m'' - a_{t}'' = a_{k''+1-t}^{\ast} = a_{k'+1-t}', \qquad 1 \leq t \leq k',
$$
and (\ref{step 2}) follows.

By construction of $a^{\ast}$, we have
\begin{align}\label{step 3}
a' \sqcup [m'' \times k''] = a^{\ast} \sqcup [m' \times k'']_{m', k'}.
\end{align}

Next, let us show that
\begin{align}\label{step 4}
[X_{a^{\ast} \sqcup [m' \times k'']_{m', k'}}] \cdot [X_{[m' \times k'] \sqcup a''}] = [X_{[m' \times k'']_{m, k}}].
\end{align}
In fact, since $[X_{a^{\ast}}]_{G''} \cdot [X_{a''}]_{G''} = [X_{[m'' \times k'']}]_{G''} \cdot [X_{[0 \times 0]_{m'', k''}}]_{G''} \; (= 0)$ by (\ref{equation complementary partitions}) and \Cref{corollary empty intersection of schubert varieties}(a), the Intersection Box Extension Principle (\Cref{corollary twisting rule}) implies that
\begin{align*}
[X_{a^{\ast} \sqcup [m' \times k'']_{m', k'}}] \cdot [X_{[m' \times k'] \sqcup a''}] &\stackrel{\ref{corollary twisting rule}}{=} [X_{[m'' \times k''] \sqcup [m' \times k'']_{m', k'}}] \cdot [X_{[m' \times k'] \sqcup [0 \times 0]_{m'', k''}}] \\
&= [X_{[m'' \times 0] \sqcup [m' \times k'']_{m', k}}] \cdot [X_{[m' \times k] \sqcup [m'' \times 0]}] \\
&\stackrel{\ref{corollary twisting rule}}{=} [X_{[m'' \times 0] \sqcup [m' \times k]}] \cdot [X_{[m' \times k'']_{m', k} \sqcup [m'' \times 0]}] \\
&= [X_{[m \times k]}] \cdot [X_{[m' \times k'']_{m, k}}] \\
&= [X_{[m' \times k'']_{m, k}}].
\end{align*}

Finally, by \Cref{corollary empty intersection of schubert varieties}, we have
\begin{align}\label{step 5}
[X_{[m' \times k'']_{m, k}}] \cdot [X_{[m'' \times k''] \sqcup [m' \times k']}] = [\operatorname{pt}].
\end{align}

This completes the proof of \Cref{theorem triple intersection}.
\end{proof}

\begin{thm}\label{theorem normally nonsingular integration}
We assume that $k' > 0$.
Suppose that $a' \in \mathcal{P}(m', k')$ is given such that $a_{1}' = m''$ and $\operatorname{max} \{t \in \mathbb{Z} | 1 \leq t \leq k, \; a_{t}' = a_{1}'\} = k' - k''$.
Let $a'' \in \mathcal{P}(m'', k'')$ be defined by $a_{t}'' = m'' - a_{k'+1-t}'$ for $1 \leq t \leq k''$ (see \Cref{figure complementary partition}).
Let $c \ell = (c \ell^{\ast}, c \ell_{\ast})$ be a Gysin coherent characteristic class with respect to $\mathcal{X}$.
Then, for $b' \in \mathcal{P}(m', k')$ with $|b'| = |a'|$, we have $\langle c \ell^{\ast} \rangle(b', a'') = \delta_{a'b'}$.
\end{thm}

\begin{proof}
Fix monomorphisms $\iota' \colon \mathbb{C}^{n'} \rightarrow \mathbb{C}^{n}$ and $\iota'' \colon \mathbb{C}^{n''} \rightarrow \mathbb{C}^{n}$ such that $\mathbb{C}^{n} = \operatorname{im}(\iota') \oplus \operatorname{im}(\iota'')$.
We set $V' = \operatorname{im}(\iota')$ and $V'' = \operatorname{im}(\iota'')$.
Fix pairs of transverse flags $(D_{\ast}', E_{\ast}')$ on $V'$ and $(D_{\ast}'', E_{\ast}'')$ on $V''$.
Using the notation introduced in \Cref{Intersection of Schubert varieties}, it follows that $D_{\ast} = D_{\ast}'' \oplus D_{\ast}'$ and $E_{\ast} = E_{\ast}' \oplus E_{\ast}''$ are transverse flags on $\mathbb{C}^{n}$.
(Note that we write $\mathbb{C}^{n} = V'' \oplus V'$ for the definition of $D_{\ast}$, and $\mathbb{C}^{n} = V' \oplus V''$ for the definition of $E_{\ast}$.)
Fix $b' \in \mathcal{P}(m', k')$, and define $a, b \in \mathcal{P}(m, k)$ by $a = a'' \sqcup c'$ and $b = b' \sqcup c''$.
By \Cref{proposition transverse flags imply whitney transverse}, the Schubert subvarieties $X_{a}(D_{\ast}), X_{b}(E_{\ast}) \subset G$ are simultaneously Whitney transverse and generically transverse.
In particular, by \Cref{prop codimension of transverse intersection}, the intersection $R_{ab} := X_{a}(D_{\ast}) \cap X_{b}(E_{\ast}) \subset G$ is a pure-dimensional closed subvariety.
By \Cref{proposition existence of flag with transversality properties}, we may choose a flag $F_{\ast}$ on $\mathbb{C}^{n}$ that is transverse to $D_{\ast}$, and such that the Schubert subvariety $X_{c}(F_{\ast}) \subset G$ is simultaneously Whitney transverse and generically transverse to $R_{ab} \subset G$.
In particular, by \Cref{prop codimension of transverse intersection}, the intersection $R_{abc} := R_{ab} \cap X_{c}(F_{\ast}) = X_{a}(D_{\ast}) \cap X_{b}(E_{\ast}) \cap X_{c}(F_{\ast}) \subset G$ is a pure-dimensional closed subvariety.
Let $Z$ denote the singular set of the Schubert subvariety $X_{c}(F_{\ast}) \subset G$.
Then, we define the smooth irreducible quasiprojective complex algebraic variety $W = G \setminus Z$, and the smooth irreducible closed subvariety $M = X_{c}(F_{\ast}) \setminus Z \subset W$ given by the set of nonsingular points of $X_{c}(F_{\ast})$.
Let $f^{!} \colon H_{\ast}(W) \rightarrow H_{\ast-2k'm''}(M)$ denote the Gysin map associated to the inclusion $f \colon M \hookrightarrow W$, where we note that $M \subset W$ is a smooth submanifold of (real) codimension $2k'm''$ that is closed as a subset.
Since the flags $D_{\ast}$ and $F_{\ast}$ are transverse, we have $\mathbb{C}^{n} = F_{n'} \oplus D_{n''} = F_{n'} \oplus \operatorname{im}(\iota'')$.
Therefore, the Segre product $\Sigma = \Sigma_{\iota', \iota''}^{k', k''} \colon G' \times G'' \rightarrow G$ defined in (\ref{equation segre product in terms of monomorphisms}) associated to the monomorphisms $\iota' \colon \mathbb{C}^{n'} \rightarrow \mathbb{C}^{n}$ and $\iota'' \colon \mathbb{C}^{n''} \rightarrow \mathbb{C}^{n}$ restricts by \Cref{proposition segre avoids singular locus} to a map $\Sigma^{W} \colon G' \times G'' \rightarrow W$.
According to the commutative diagram (\ref{equation compatibility of Segre products}), the map $\Sigma \colon G' \times G'' \rightarrow G$ is related to the Segre product $S = S_{V', V''}^{k', k''} \colon G_{k'}(V') \times G_{k''}(V'') \hookrightarrow G$ by $\Sigma = S \circ (\Phi' \times \Phi'')$, where $\Phi' \colon G' \cong G_{k'}(V')$ and $\Phi'' \colon G'' \cong G_{k''}(V'')$ are the isomorphisms induced by the monomorphisms $\iota'$ and $\iota''$, respectively.
Hence, we have $\Sigma^{W} = S^{W} \circ (\Phi' \times \Phi'')$, where $S^{W} \colon G_{k'}(V') \times G_{k''}(V'') \hookrightarrow W$ denotes the restriction of the closed embedding $S$.
By applying the induced map on homology $\Sigma_{\ast}^{W} \colon H_{\ast}(G' \times G'') \rightarrow H_{\ast}(W)$ to
$$
[X_{b'}]_{G'} \times [X_{a''}]_{G''} = [X_{b'}(E_{\ast}')]_{G'} \times [X_{a''}(D_{\ast}'')]_{G''} = [X_{b'}(E_{\ast}') \times X_{a''}(D_{\ast}'')]_{G' \times G''},
$$
it follows that
\begin{equation}\label{equation application of segre product}
\Sigma^{W}_{\ast}([X_{b'}]_{G'} \times [X_{a''}]_{G''}) = S^{W}_{\ast} [X_{b'}(E_{\ast}') \times X_{a''}(D_{\ast}'')]_{G_{k'}(V') \times G_{k''}(V'')} = [S(X_{b'}(E_{\ast}') \times X_{a''}(D_{\ast}''))]_{W}.
\end{equation}
Furthermore, \Cref{proposition segre product of schubert varieties} implies that
\begin{equation*}
S((X_{b'}(E_{\ast}') \cap X_{c'}(D_{\ast}')) \times (X_{c''}(E_{\ast}'') \cap X_{a''}(D_{\ast}''))) = X_{b' \sqcup c''}(E_{\ast}' \oplus E_{\ast}'') \cap X_{a'' \sqcup c'}(D_{\ast}'' \oplus D_{\ast}'),
\end{equation*}
that is,
\begin{equation}\label{equation segre product of specific schubert}
S(X_{b'}(E_{\ast}') \times X_{a''}(D_{\ast}'')) = X_{b}(E_{\ast}) \cap X_{a}(D_{\ast}) = R_{ab}.
\end{equation}
In particular, we have $R_{ab} \subset S(G_{k'}(V') \times G_{k''}(V'')) = \Sigma(G' \times G'') \subset W$.
Consequently, we have
\begin{equation}\label{equation triple intersection alternative description}
R_{abc} = R_{ab} \cap X_{c}(F_{\ast}) = R_{ab} \cap M.
\end{equation}
Since $X_{c}(F_{\ast})$ and $R_{ab}$ are Whitney transverse in $G$, we may fix transverse Whitney stratifications on $X_{c}(F_{\ast})$ and $R_{ab}$.
By virtue of \Cref{lemma whitney restratification manifold}, it follows that $M$ and $R_{ab}$ are Whitney transverse in $W$, where $M$ is equipped with the trivial stratification with single stratum $M$.
Then, we apply \Cref{lem.gysinfund} to the Whitney stratified subspaces $X := W$ (equipped with the trivial stratification with single stratum $W$) and $K := R_{ab}$ of $W$ (both of which are pure-dimensional closed subvarieties of $W$, and hence oriented pseudomanifolds), and to the oriented smooth submanifold $M \subset W$ that is closed as a subset and clearly transverse to $X = W$, and also transverse to $K = R_{ab}$.
Consequently,
\begin{equation}\label{equation gysin fundamental class}
f^{!} [R_{ab}]_{W} = [R_{ab} \cap M]_{M}.
\end{equation}
Altogether, we obtain
\begin{align}\label{equation intermediate gysin}
f^{!} \Sigma_{\ast}^{W}([X_{b'}]_{G'} \times [X_{a''}]_{G''}) &\stackrel{(\ref{equation application of segre product})}{=} f^{!} [\Sigma^{W}(X_{b'}(E_{\ast}') \times X_{a''}(D_{\ast}''))]_{W} \stackrel{(\ref{equation segre product of specific schubert})}{=} f^{!} [R_{ab}]_{W} \stackrel{(\ref{equation gysin fundamental class})}{=} [R_{ab} \cap M]_{M} \stackrel{(\ref{equation triple intersection alternative description})}{=} [R_{abc}]_{M}.
\end{align}
Since $R_{ab}$ and $X_{c}(F_{\ast})$, as well as $X_{a}(D_{\ast})$ and $X_{b}(E_{\ast})$ are generically transverse in $G$ by construction and $R_{ab}$ is irreducible by (\ref{equation segre product of specific schubert}), we may apply \Cref{lemma intersection is cap product for whitney transverse subvarieties} twice to obtain
\begin{equation}\label{equation product rule for triple}
[R_{abc}]_{G} = [R_{ab} \cap X_{c}(F_{\ast})]_{G} = [R_{ab}]_{G} \cdot [X_{c}]_{G} = [X_{a}]_{G} \cdot [X_{b}]_{G} \cdot [X_{c}]_{G}.
\end{equation}
Consequently, using the inclusion map $h \colon M \rightarrow G$, we obtain
\begin{align*}
h_{\ast} f^{!} \Sigma^{W}_{\ast}([X_{b'}]_{G'} \times [X_{a''}]_{G''}) &\stackrel{(\ref{equation intermediate gysin})}{=} h_{\ast}[R_{ab} \cap M]_{M} = [R_{ab} \cap M]_{G} \stackrel{(\ref{equation product rule for triple})}{=} [X_{a}]_{G} \cdot [X_{b}]_{G} \cdot [X_{c}]_{G}.
\end{align*}
Finally, under the given assumptions on $b' \in \mathcal{P}(m', k')$ and $a'' \in \mathcal{P}(m'', k'')$, we see that $f^{!} \Sigma^{W}_{\ast}([X_{b'}]_{G'} \times [X_{a''}]_{G''}) \in H_{0}(M)$.
Hence, using the normalization $(c \ell^{\ast}(f)^{-1})^{0} = 1$, \Cref{theorem triple intersection} yields
$$
\langle c \ell^{\ast} \rangle(b', a'') = \varepsilon_{\ast}f^{!}\Sigma^{W}_{\ast}([X_{b'}]_{G'} \times [X_{a''}]_{G''}) = \varepsilon_{\ast}h_{\ast}f^{!}\Sigma^{W}_{\ast}([X_{b'}]_{G'} \times [X_{a''}]_{G''}) = \varepsilon_{\ast}([X_{a}]_{G} \cdot [X_{b}]_{G} \cdot [X_{c}]_{G}) = \delta_{a'b'}.
$$

This completes the proof of \Cref{theorem normally nonsingular integration}.
\end{proof}

\begin{remark}\label{remark integral expression}
The proof of \Cref{theorem normally nonsingular integration} shows that for suitable flags $D_{\ast}$, $E_{\ast}$, and $F_{\ast}$ on $\mathbb{C}^{n}$, the expression $\langle c \ell^{\ast} \rangle(b', a'')$ (\ref{equation integral constant in recursive formula}) can be computed for all $b' \in \mathcal{P}(m', k')$ and $a'' \in \mathcal{P}(m'', k'')$ by
$$
\langle c \ell^{\ast} \rangle(b', a'') = \langle c \ell^{\ast}(f)^{-1}, f^{!}\Sigma_{\ast}^{W}([X_{b'}]_{G'} \times [X_{a''}]_{G''}) \rangle \stackrel{(\ref{equation intermediate gysin})}{=} \langle c \ell^{\ast}(f)^{-1}, R_{abc} \rangle =: \int_{R_{abc}} c \ell^{\ast}(f)^{-1},
$$
that is, by integration of the class $c \ell^{\ast}(f)^{-1} \in H^{\ast}(M)$ over the subvariety $R_{abc} = X_{a}(D_{\ast}) \cap X_{b}(E_{\ast}) \cap X_{c}(F_{\ast}) \subset M$.
(Note that $R_{abc}$ might not be irreducible according to \cite[Remark 2.2]{bc}.)
\end{remark}

\subsection{Genera of Characteristic Subvarieties}\label{Characteristic Subvarieties}
We turn to the construction and analysis of the genera $|c \ell_{\ast}|(i', i'') \in \mathbb{Q}$ associated to a Gysin coherent characteristic class $c \ell = (c \ell^{\ast}, c \ell_{\ast})$ with respect to $\mathcal{X}$, where $i' \colon X' \hookrightarrow G'$ and $i'' \colon X'' \hookrightarrow G''$ are inclusions of possibly singular irreducible compact subvarieties in ambient Grassmannians.
Here, the notion of $\mathcal{X}$-transversality with respect to the fixed family $\mathcal{X}$ of admissible embeddings enters for the first time.

We continue to use the notation of \Cref{Normally Nonsingular Integration}.
That is, $m', k' \geq 0$ and $m'', k'' \geq 0$ are integers, and we define the integers $m, k \geq 0$ by $m = m'+m''$ and $k=k'+k''$, and the integers $n, n', n'' \geq 0$ by $n = m+k$, $n' = m'+k'$, and $n'' = m''+k''$.
Using the notation introduced in \Cref{Intersection of Schubert varieties}, let $c' = [m' \times k'] \in \mathcal{P}(m', k')$, $c'' = [m'' \times k''] \in \mathcal{P}(m'', k'')$, and $c = c' \sqcup c'' \in \mathcal{P}(m, k)$ (see \Cref{figure proof partitions c}(a)).
Finally, let $G = G_{k}(\mathbb{C}^{n})$, $G' = G_{k'}(\mathbb{C}^{n'})$, and $G'' = G_{k''}(\mathbb{C}^{n''})$.

The main result of this section is the following

\begin{thm}\label{main theorem signature independent of choices}
Let $i' \colon X' \hookrightarrow G'$ be the inclusion of an irreducible closed algebraic subvariety $X' \subset G'$ of dimension $d'$, and let $i'' \colon X'' \hookrightarrow G''$ be the inclusion of an irreducible closed algebraic subvariety $X'' \subset G''$ of dimension $d''$.
Fix monomorphisms $\iota' \colon \mathbb{C}^{n'} \rightarrow \mathbb{C}^{n}$ and $\iota'' \colon \mathbb{C}^{n''} \rightarrow \mathbb{C}^{n}$ such that $\mathbb{C}^{n} = \operatorname{im}(\iota') \oplus \operatorname{im}(\iota'')$.
Let $\Sigma = \Sigma_{\iota', \iota''}^{k', k''} \colon G' \times G'' \rightarrow G$ be the associated Segre product defined in (\ref{equation segre product in terms of monomorphisms}).
Let $X \subset G$ be the irreducible closed algebraic subvariety of dimension $d = d' + d''$ given by the Segre product $X = \Sigma(X' \times X'') \subset G$ of $X'$ and $X''$.
Let $c \ell = (c \ell^{\ast}, c \ell_{\ast})$ be a Gysin coherent characteristic class with respect to $\mathcal{X}$ such that $i', i'' \in \mathcal{X}$.
Suppose that $F_{\ast}$ is a flag on $\mathbb{C}^{n}$ such that $\mathbb{C}^{n} = F_{n'} \oplus \operatorname{im}(\iota'')$, and such that the Schubert subvariety $X_{c}(F_{\ast}) \subset G$ is $\mathcal{X}$-transverse to $X \subset G$.
(Such a flag $F_{\ast}$ exists by an argument similar to the proof of \Cref{proposition existence of flag with transversality properties}, where we note that the embeddings $X \hookrightarrow G$ and $X_{c} \hookrightarrow G$ are both in $\mathcal{X}$ so that the Kleiman transversality axiom is applicable.)
Then, the associated family of coefficients $\{\lambda^{a}_{i}\}_{i, a}$ introduced in (\ref{equation l class linear combination of schubert cycles}) satisfies
\begin{equation}\label{equation signature well defined}
|c\ell_{\ast}|(X \cap X_{c}(F_{\ast}) \subset G) = \sum_{r = 0}^{l} \sum_{\substack{0 \leq r' \leq l, \\ 0 \leq r'' \leq l, \\ r' + r'' = r}} \sum_{\substack{b' \in \mathcal{P}(m', k'), \\ |b'| = d'-r'}} \sum_{\substack{b'' \in \mathcal{P}(m'', k''), \\ |b''| = d''-r''}} \lambda_{i'}^{b'} \lambda_{i''}^{b''} \cdot \langle c \ell^{\ast} \rangle(b', b''),
\end{equation}
where $l := d-k'm''$, and $\langle c \ell^{\ast} \rangle(b', b'') \in \mathbb{Q}$ is defined in (\ref{equation integral constant in recursive formula}).
Consequently, by \Cref{proposition invariance of integral coefficient for recursive formula}, the characteristic number
$$
|c\ell_{\ast}|(i', i'') := |c\ell_{\ast}|(X \cap X_{c}(F_{\ast}) \subset G)
$$
is independent of all choices involved.
(That is, $|c\ell_{\ast}|(i', i'')$ does not depend on the choice of the monomorphisms $\iota'$ and $\iota''$, and the choice of the flag $F_{\ast}$.)
\end{thm}

\begin{proof}
Let $Z$ denote the singular set of the Schubert subvariety $X_{c}(F_{\ast}) \subset G$.
Then, we define the smooth irreducible quasiprojective complex algebraic variety $W = G \setminus Z$.
Since $\mathbb{C}^{n} = F_{n'} \oplus \operatorname{im}(\iota'')$, we have $P := \Sigma(G' \times G'') \subset W$ by \Cref{proposition segre avoids singular locus}.
Let $i^{P} \colon X \rightarrow P$ and $f_{P} \colon P \rightarrow W$ be the inclusion maps, and let $i = f_{P} \circ i^{P} \colon X \rightarrow W$ be their composition.
Then, axiom (\ref{axiom locality}) of the Gysin coherent characteristic class $c \ell$ yields $c \ell_{\ast}(i) = f_{P \ast}c \ell_{\ast}(i^{P})$.
Moreover, by applying axiom (\ref{axiom isomorphism}) of the Gysin coherent characteristic class $c \ell$ to the commutative diagram
\begin{align*}
\xymatrix{
X' \times X'' \ar[rr]^{\Sigma|}_{\cong} \ar[d]_{i' \times i''} & & X \ar[d]^{i^{P}} \\
G' \times G'' \ar[rr]^{\Phi := \Sigma|}_{\cong} & & P,
}
\end{align*}
we obtain $c \ell_{\ast}(i^{P}) = \Phi_{\ast}c \ell_{\ast}(i' \times i'')$.
Furthermore, axiom (\ref{axiom multiplicativity}) states $c \ell_{\ast}(i' \times i'') = c \ell_{\ast}(i') \times c \ell_{\ast}(i'')$.
Altogether, we obtain
\begin{equation}\label{first step of proof of independence theorem}
c \ell_{\ast}(i) = f_{P \ast}c \ell_{\ast}(i^{P}) = f_{P \ast}\Phi_{\ast}c \ell_{\ast}(i' \times i'') = \Sigma^{W}_{\ast}(c \ell_{\ast}(i') \times c \ell_{\ast}(i'')),
\end{equation}
where $\Sigma^{W} := f_{P} \circ \Phi \colon G' \times G'' \rightarrow W$ denotes the composition.

Note that the class $c \ell_{\ast}(i') \in H_{\ast}(G'; \mathbb{Q})$ is concentrated in even degrees of the form $2d'-2r'$ for nonnegative integers $r'$ because $c \ell_{q}(i') = 0$ for $q > 2d' = 2\operatorname{dim}_{\mathbb{C}}X'$, and the homology of the Grassmannian $G'$ is concentrated in even degrees.
Similarly, the class $c \ell_{\ast}(i'') \in H_{\ast}(G''; \mathbb{Q})$ is concentrated in even degrees of the form $2d''-2r''$ for nonnegative integers $r''$.
Therefore, equation (\ref{first step of proof of independence theorem}) implies that the class $c \ell_{\ast}(i) \in H_{\ast}(W; \mathbb{Q})$ is concentrated in even degrees of the form $2d-2r$ for nonnegative integers $r$, and we have
\begin{align}\label{equation characteristic class of product of spaces}
c \ell_{2d-2r}(i) = \sum_{\substack{r', r'' \geq 0, \\ r' + r'' = r}} \Sigma_{\ast}^{W}(c \ell_{2d'-2r'}(i') \times c \ell_{2d''-2r''}(i'')), \qquad r \geq 0.
\end{align}

Next, we define the smooth irreducible closed subvariety $M = X_{c}(F_{\ast}) \setminus Z \subset W$ given by the set of nonsingular points of $X_{c}(F_{\ast})$.
Let $f \colon M \rightarrow W$ denote the inclusion map.
Since $X, X_{c}(F_{\ast}) \subset G$ are $\mathcal{X}$-transverse by assumption, it follows that $X \cap W = X$ and $X_{c}(F_{\ast}) \cap W = M$ are $\mathcal{X}$-transverse in $W$.
We apply axiom (\ref{axiom gysin}) of the Gysin coherent characteristic class $c \ell$ to the
pure-dimensional compact subvariety $Y := M \cap X = X_{c}(F_{\ast}) \cap X \subset M$ with inclusion map $j \colon Y \rightarrow M$ to obtain
\begin{equation}\label{equation gysin restriction in uniqueness proof}
c \ell^{\ast}(f)^{-1} \cap f^{!}c \ell_{\ast}(i) = c \ell_{\ast}(j) \qquad \in H_{\ast}(M; \mathbb{Q}).
\end{equation}

As shown above, the class $c \ell_{\ast}(i) \in H_{\ast}(W; \mathbb{Q})$ is concentrated in even degrees of the form $2d-2r$ for nonnegative integers $r$.
The Gysin map $f^{!}$ drops the degree in homology by $2k'm''$ because $M \subset W$ has complex codimension $k'm''$.
Setting $l := d-k'm''$, we conclude that $f^{!}c \ell_{\ast}(i) \in H_{\ast}(M; \mathbb{Q})$ is concentrated in even degrees of the form $2d-2r-2k'm'' = 2l-2r$ for nonnegative integers $r$.
Therefore, the part in $H_{0}(M; \mathbb{Q})$ of equation (\ref{equation gysin restriction in uniqueness proof}) is
\begin{equation}\label{equation zero degree part even degrees}
\sum_{r = 0}^{l} (c \ell^{\ast}(f)^{-1})^{2l-2r} \cap f^{!} c \ell_{2d-2r}(i) = c \ell_{0}(j) \quad \in H_{0}(M; \mathbb{Q}).
\end{equation}

By inserting equation (\ref{equation characteristic class of product of spaces}) into equation (\ref{equation zero degree part even degrees}), we obtain
\begin{equation*}
\sum_{r = 0}^{l} \sum_{\substack{r', r'' \geq 0, \\ r' + r'' = r}} (c \ell^{\ast}(f)^{-1})^{2l-2r} \cap f^{!} \Sigma_{\ast}^{W}(c \ell_{2d'-2r'}(i') \times c \ell_{2d''-2r''}(i'')) = c \ell_{0}(j).
\end{equation*}
Therefore, by applying the augmentation $\varepsilon_{\ast} \colon H_{\ast}(M; \mathbb{Q}) \rightarrow \mathbb{Q}$, we have
\begin{equation*}
\sum_{r = 0}^{l} \sum_{\substack{r', r'' \geq 0, \\ r' + r'' = r}} \langle c \ell^{\ast}(f)^{-1}, f^{!} \Sigma_{\ast}^{W}(c \ell_{2d'-2r'}(i') \times c \ell_{2d''-2r''}(i'')) \rangle = |c\ell_{\ast}|(j).
\end{equation*}
Using (\ref{equation l class linear combination of schubert cycles}) to write
\begin{align*}
c \ell_{2d'-2r'}(i') &= \sum_{\substack{b' \in \mathcal{P}(m', k'), \\ |b'| = d'-r'}} \lambda_{i'}^{b'} \cdot [X_{b'}]_{G'}, \\
c \ell_{2d''-2r''}(i'') &= \sum_{\substack{b'' \in \mathcal{P}(m'', k''), \\ |b''| = d''-r''}} \lambda_{i''}^{b''} \cdot [X_{b''}]_{G''},
\end{align*}
we obtain
$$
|c\ell_{\ast}|(j) = \sum_{r = 0}^{l} \sum_{\substack{r', r'' \geq 0, \\ r' + r'' = r}} \sum_{\substack{b' \in \mathcal{P}(m', k'), \\ |b'| = d'-r'}} \sum_{\substack{b'' \in \mathcal{P}(m'', k''), \\ |b''| = d''-r''}} \lambda_{i'}^{b'} \lambda_{i''}^{b''} \cdot \langle c \ell^{\ast}(f)^{-1}, f^{!} \Sigma_{\ast}^{W}([X_{b'}]_{G'} \times [X_{b''}]_{G''}) \rangle.
$$
Hence, equation (\ref{equation signature well defined}) follows in view of (\ref{equation integral constant in recursive formula}).
Consequently, the expression $|c\ell_{\ast}|(i', i'') := |c\ell_{\ast}|(X \cap X_{c}(F_{\ast}) \subset G)$ is independent of all choices involved.

This completes the proof of \Cref{main theorem signature independent of choices}.
\end{proof}

\begin{remark}\label{remark characteristic varieties}
The characteristic number $|c\ell_{\ast}|(i', i'')$ introduced in \Cref{main theorem signature independent of choices} is defined as the genus of the inclusion $X \cap X_{c}(F_{\ast}) \subset G$, which may be called a characteristic subvariety in view of its role in \Cref{main theorem normally nonsingular expansion}.
Note that if the spaces $X'$ and $X''$ are Schubert varieties, then it follows from \Cref{proposition segre product of schubert varieties} that the characteristic subvariety is the intersection of three Schubert varieties in general position.
It is worthwhile mentioning that for a finite number of given Gysin coherent characteristic classes with respect to $\mathcal{X}$, we may assume that the same notion of $\mathcal{X}$-transversality applies in axiom (\ref{axiom gysin}), so that we can simultaneously use the same collection of characteristic subvarieties in their computation.
\end{remark}

\subsection{Proof of \Cref{main theorem normally nonsingular expansion}}\label{proof of normally nonsingular expansion}
We apply \Cref{main theorem signature independent of choices} to the inclusions $i' \colon X' \hookrightarrow G'$ and $i'' \colon X'' := X_{a''}(D_{\ast}'') \hookrightarrow G''$, which are both contained in $\mathcal{X}$, to obtain
\begin{equation}\label{equation signature well defined used in proof of main theorem}
|c\ell_{\ast}|(i', i'') = \sum_{r = 0}^{l} \sum_{\substack{0 \leq r' \leq l, \\ 0 \leq r'' \leq l, \\ r' + r'' = r}} \sum_{\substack{b' \in \mathcal{P}(m', k'), \\ |b'| = d'-r'}} \sum_{\substack{b'' \in \mathcal{P}(m'', k''), \\ |b''| = d''-r''}} \lambda_{i'}^{b'} \lambda_{i''}^{b''} \cdot \langle c \ell^{\ast} \rangle(b', b''),
\end{equation}
where we note that $l = d'+d''-k'm''$ with $d'' := \operatorname{dim}_{\mathbb{C}}(X'') = |a''|$ because we have $|a'| + |a''| = k'm''$ by definition of $a''$ (see \Cref{figure complementary partition}).
Hence, to show (\ref{equation main result recuvrsive coefficient formula}), it remains to show the equality
$$
\lambda^{a'}_{i'} = \sum_{\substack{b' \in \mathcal{P}(m', k'), \\ |b'| = |a'|}} \sum_{\substack{b'' \in \mathcal{P}(m'', k''), \\ |b''| = |a''|}} \lambda_{i'}^{b'} \lambda_{i''}^{b''} \cdot \langle c \ell^{\ast} \rangle(b', b''),
$$
whose right hand side equals the summand indexed by $(r, r', r'') = (l, l, 0)$ on the right hand side of (\ref{equation signature well defined used in proof of main theorem}).

By definition of the Gysin coherent characteristic class $c\ell$, we have $c \ell_{2d''}(X'') = [X'']_{G''}$ in $H_{2d''}(G''; \mathbb{Q})$.
In other words, we have $\lambda_{i''}^{b''} = \delta_{a''b''}$ for all $b'' \in \mathcal{P}(m'', k'')$ with $|b''| = |a''|$.
Furthermore, by \Cref{theorem normally nonsingular integration}, we have $\langle c \ell^{\ast} \rangle(b', a'') = \delta_{a'b'}$, and the claim follows.

This completes the proof of \Cref{main theorem normally nonsingular expansion}.

\section{Proof of \Cref{main result on Gysin coherent characteristic classes with respect to x}}\label{proof of main theorem}
Let $c \ell$ and $\widetilde{c \ell}$ be Gysin coherent characteristic classes with respect to $\mathcal{X}$ such that $c\ell^{\ast} = \widetilde{c\ell}^{\ast}$ and $|c\ell_{\ast}| = |\widetilde{c\ell}_{\ast}|$ for the associated genera.
We prove the claim by induction on the complex dimension $g = km$ (where $m:=n-k$) of the target Grassmannian $G = G_{k}(\mathbb{C}^{n})$ ($0 \leq k \leq n$) of inclusions $i \colon X \hookrightarrow G$ of irreducible closed algebraic subvarieties $X$.
As for the induction basis $g = 0$, we note that $G = \operatorname{pt}$ is a one point space, and we have $X = \operatorname{pt}$ since irreducibility of $X$ implies that $X \neq \varnothing$.
Hence, we obtain $c \ell_{\ast}(i) = [X]_{G} = [\operatorname{pt}] = \widetilde{c \ell}_{\ast}(i)$ in $H_{\ast}(G; \mathbb{Q}) = \mathbb{Q} \cdot [\operatorname{pt}]$, so that the inclusion $i \colon X \hookrightarrow G$ satisfies $c \ell_{\ast}(i) = \widetilde{c \ell}_{\ast}(i)$.
As for the induction hypothesis, we fix an integer $g > 0$, and assume that
\begin{itemize}
\item[$(\ast)_{g}$] for every inclusion $i \colon X \hookrightarrow G$ in $\mathcal{X}$ of an irreducible closed algebraic subvariety $X$ in a Grassmannian $G = G_{k}(\mathbb{C}^{n})$ of complex dimension $< g$, we have $c \ell_{\ast}(i) = \widetilde{c \ell}_{\ast}(i)$.
\end{itemize}
As for the induction step, we consider the inclusion $i' \colon X' \hookrightarrow G'$ in $\mathcal{X}$ of an irreducible closed $d'$-dimensional algebraic subvariety $X'$ in a $g$-dimensional Grassmannian $G' = G_{k'}(\mathbb{C}^{n'})$ (where we note that $0 < k' < n'$ since $g = k'm' > 0$ with $m' := n'-k'$), and have to show that $c \ell_{\ast}(i') = \widetilde{c \ell}_{\ast}(i')$ in $H_{\ast}(G'; \mathbb{Q})$.
That is, if $\{\lambda^{b}_{i}\}_{i, b}$ and $\{\widetilde{\lambda}^{b}_{i}\}_{i, b}$ denote the families of coefficients associated to $c \ell$ and $\widetilde{c \ell}$, respectively, as introduced in (\ref{equation l class linear combination of schubert cycles}), then we have to show that $\lambda^{b'}_{i'} = \widetilde{\lambda}^{b'}_{i'}$ for all $b' \in \mathcal{P}(m', k')$ with $|b'| \in \{0, \dots, d'\}$, which we shall prove by induction on $l := d' - |b'|$ (see \Cref{figure induction scheme}).
As for the induction basis $l = 0$, we note that $c \ell_{2d'}(i') = [X']_{G'} = \widetilde{c \ell}_{2d'}(i')$ in $H_{2d'}(G'; \mathbb{Q})$.
As for the induction hypothesis, we suppose that $l \in \{1, \dots, d'\}$ is given such that
\begin{itemize}
\item[$(\ast\ast)_{l}$] for every $b' \in \mathcal{P}(m', k')$ with $d'-|b'| \in \{0, \dots, l-1\}$, we have $\lambda^{b'}_{i'} = \widetilde{\lambda}^{b'}_{i'}$.
\end{itemize}

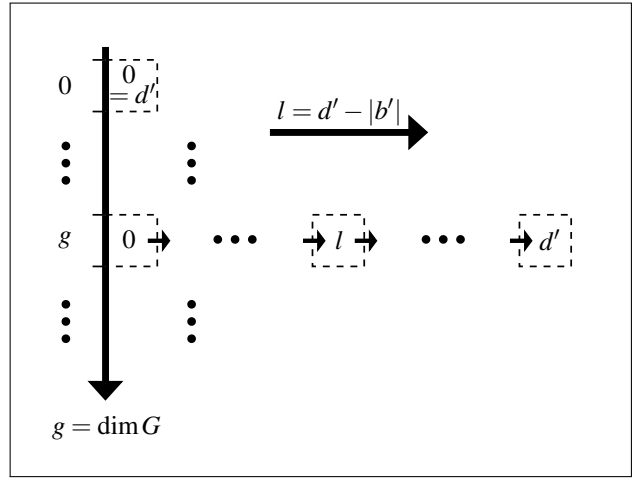
\begin{figure}[htbp]
\centering
\fbox{\begin{tikzpicture}
\draw (0, 0) node {\includegraphics[width=0.6\textwidth]{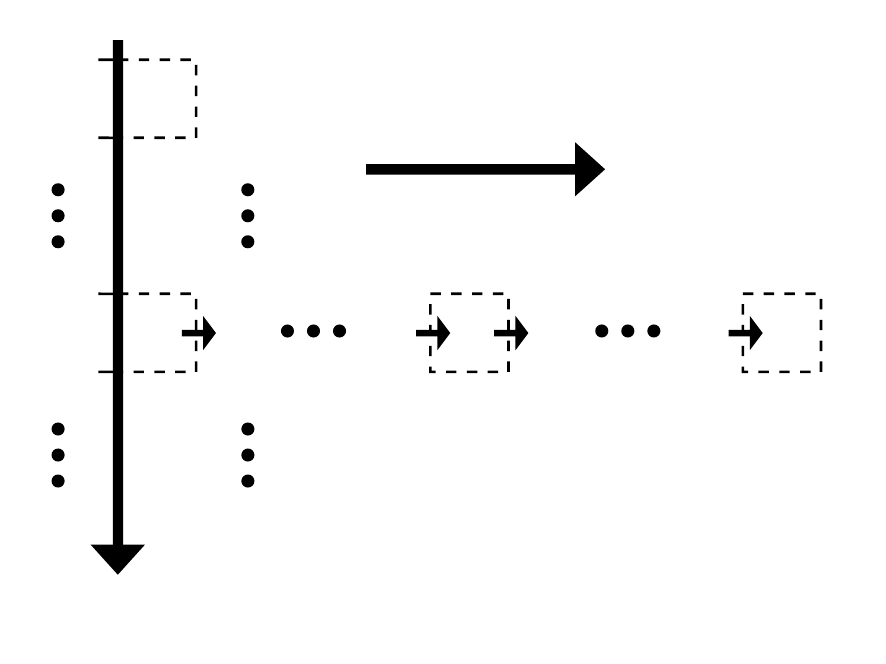}};
\draw (-3.4, 2.05) node {$0$};
\draw (-3.4, 0) node {$g$};
\draw (-2.85, -2.5) node {$g = \operatorname{dim} G$};
\draw (0.25, 0) node {$l$};
\draw (-2.55, 0) node {$0$};
\draw (-2.55, 2.2) node {$0$};
\draw (-2.5, 1.95) node {$= d'$};
\draw (3.05, 0.0) node {$d'$};
\draw (0.25, 1.7) node {$l = d' - |b'|$};
\end{tikzpicture}}
\caption{Induction scheme for the proof of \Cref{main result on Gysin coherent characteristic classes with respect to x}.
The outer induction runs over the dimension $g$ of the ambient Grassmannian $G$, starting with $g = 0$, while the inner induction runs for every fixed embedding $i' \colon X' \rightarrow G' = G_{k'}(\mathbb{C}^{m'+k'})$ in $\mathcal{X}$ with $\operatorname{dim}G' = g$ and $d' := \operatorname{dim}X'$ over the formal codimension $l := d' - |b'|$ of partitions $b' \in \mathcal{P}(m', k')$ in $X'$, starting with $l = 0$.}
\label{figure induction scheme}
\end{figure}

As for the induction step, we fix $a' \in \mathcal{P}(m', k')$ with $d'-|a'| = l$, and have to show that $\lambda^{a'}_{i'} = \widetilde{\lambda}^{a'}_{i'}$.
By applying \Cref{main theorem normally nonsingular expansion} to the inclusion $i' \colon X' \hookrightarrow G'$ in $\mathcal{X}$, we obtain an inclusion $i'' \colon X_{a''}(D_{\ast}'') \hookrightarrow G''$ of a Schubert subvariety $X_{a''}(D_{\ast}'') \subset G''$ into the Grassmannian $G'' = G_{k''}(\mathbb{C}^{m''+k''})$ whose dimension $k''m''$ is strictly smaller than the dimension $g = k'm'$ of $G'$ because we have $0 \leq k'' < k'$ and $0 \leq m'' \leq m'$ by definition of $k''$ and $m''$.
Consequently, we have $\lambda^{b''}_{i''} = \widetilde{\lambda}^{b''}_{i''}$ for all $b'' \in \mathcal{P}(m'', k'')$ by induction hypothesis $(\ast)_{g}$.
Furthermore, we have $\lambda^{b'}_{i'} = \widetilde{\lambda}^{b'}_{i'}$ for all $b' \in \mathcal{P}(m', k')$ with $d'-|b'| \in \{0, \dots, l-1\}$ by induction hypothesis $(\ast \ast)_{l}$.
All in all, we conclude that $\lambda^{a'}_{i'} = \widetilde{\lambda}^{a'}_{i'}$ by comparing the equations (\ref{equation main result recuvrsive coefficient formula}) induced by $c \ell$ and $\widetilde{c \ell}$, respectively.
Here, note that for the genera, we have
$$
|c \ell_{\ast}|(i', i'') = |c \ell_{\ast}|(X \cap X_{c}(F_{\ast}) \hookrightarrow G) = |\widetilde{c \ell}_{\ast}|(X \cap X_{c}(F_{\ast}) \hookrightarrow G) = |\widetilde{c \ell}_{\ast}|(i', i'')
$$
because we have $|c \ell_{\ast}| = |\widetilde{c \ell}_{\ast}|$, and we may use the same collection of characteristic subvarieties by \Cref{remark characteristic varieties}.
For the integrals, we have $\langle c \ell^{\ast} \rangle(b', b'') = \langle \widetilde{c \ell}^{\ast} \rangle(b', b'')$ since $(c \ell^{\ast})^{-1} = (\widetilde{c \ell}^{\ast})^{-1}$, and both of these cohomology classes are integrated over the same triple intersection of Schubert varieties by \Cref{remark integral expression}.

This completes the proof of \Cref{main result on Gysin coherent characteristic classes with respect to x}.

\section{The Goresky-MacPherson $L$-Class}\label{The Goresky-MacPherson L-Class}
Following Siegel \cite{siegel}, an oriented PL pseudomanifold $X$ is called a Witt space if for some (and hence, as shown in \cite[Section 2.4]{gmih2}, for any) PL stratification of $X$, the middle degree, lower middle perversity rational intersection homology of all even-dimensional links vanishes.
For example, any pure-dimensional complex algebraic variety (endowed with the complex topology) is a Witt space.
(In fact, such a space has a natural orientation, and can be PL stratified without strata of odd codimension.)
Closed Witt spaces $X$ have homological $L$-classes $L_{i}(X) \in H_{i}(X; \mathbb{Q})$ that generalize the Poincar\'{e} duals of the cohomological $L$-classes of Hirzebruch \cite{hirzebruch} defined for smooth manifolds.
These classes were first introduced by Goresky and MacPherson \cite{gmih1} for PL pseudomanifolds without strata of odd codimension.
The construction of the Goresky-MacPherson-Siegel $L$-classes is based on a Thom-Pontrjagin type approach that exploits transversality techniques and bordism invariance of the signature $\sigma(X)$ of the Goresky-MacPherson-Siegel intersection form on middle-perversity intersection homology of the Witt space $X$.
Cheeger \cite{cheeger} gave a local formula for $L$-classes in terms of the eta-invariant of simplicial links.
In \cite{banagl-mem} and \cite{banagl-lcl}, the first author extended $L$-classes to oriented stratified pseudomanifolds that allow for Lagrangian sheaves along strata of odd codimension.
For the central role of $L$-classes in the topological classification of singular spaces similar to that of Hirzebruch's $L$-classes in the classification theory of high-dimensional manifolds, we refer to works of Weinberger, Cappell and Shaneson (\cite{csw}, \cite{cw2}, \cite{weinberger}).

The collection of $L$-classes $L_{\ast}(X) = L_{0}(X) + L_{1}(X) + \dots \in H_{\ast}(X; \mathbb{Q})$ with $X$ ranging over all closed Witt spaces is characterized by the following axioms (see e.g. \cite{cs}, and Proposition 8.2.11 in \cite{banagltiss}, and Theorem 9.4.18 in \cite{friedman}):
\begin{enumerate}
\item \emph{(Signature normalization)}
For all $X$, we have $\epsilon_{\ast}L_{\ast}(X) = \sigma(X)$.
\item\label{gysin axiom trivial normal bundle} \emph{(Gysin restriction for normally nonsingular inclusions with trivial normal bundle)}
If $g \colon Y \hookrightarrow X$ is a normally nonsingular inclusion of closed Witt spaces with trivial normal bundle, then 
$$
g^{!}L_{\ast}(X) = L_{\ast}(Y)
$$
under the Gysin homomorphism $g^{!} \colon H_{\ast}(X; \mathbb{Q}) \rightarrow H_{\ast}(Y; \mathbb{Q})$.
\end{enumerate}

For concrete computations in the complex projective algebraic setting, axiom (\ref{gysin axiom trivial normal bundle}) is often not applicable due to the failure of triviality of normal bundles.
For example, normal bundles of normally nonsingular embeddings that arise from transverse intersections of singular projective varieties with smooth varieties are frequently nontrivial.
By using the machinery of Banagl-Laures-McClure \cite{blm}, the first author established the following Gysin restriction formula for the Goresky-MacPherson-Siegel $L$-class for arbitrary normally nonsingular inclusions of even-dimensional Witt spaces.

\begin{thm}[see Theorem 3.18 in \cite{banagllgysin}]\label{thm banagl topological gysin restriction}
Let $g \colon Y \hookrightarrow X$ be a normally nonsingular inclusion of closed oriented even-dimensional PL Witt pseudomanifolds.
Let $\nu$ be the topological normal bundle of $g$.
Then
$$
g^{!}L_{\ast}(X) = L^{\ast}(\nu) \cap L_{\ast}(Y).
$$
\end{thm}

A different axiomatization involving codimension $0$ restrictions was used by Matsui in \cite[p. 61]{mats} to investigate ambient intersection formulae for Goresky-MacPherson $L$-classes.
In \cite{banaglcovertransfer}, the first author analyzed the behavior of $L$-classes under Gysin transfers associated to finite degree covers.

The Goresky-MacPherson-Siegel $L$-class of complex projective algebraic varieties fits into the framework of Gysin coherent characteristic classes (\Cref{definition gysin coherent characteristic classes}), as we shall detail next.

\begin{thm}\label{proposition l class is l type characteristic class}
The pair $\mathcal{L} = (\mathcal{L}^{\ast}, \mathcal{L}_{\ast})$ defined by $\mathcal{L}^{\ast}(f) = L^{\ast}(\nu_{f})$ for every inclusion $f \colon M \rightarrow W$ of a smooth closed subvariety $M \subset W$ in a smooth variety $W$ with normal bundle $\nu_{f}$, and by $\mathcal{L}_{\ast}(i) = i_{\ast} L_{\ast}(X)$ for every inclusion $i \colon X \rightarrow W$ of a compact possibly singular subvariety $X \subset W$ in a smooth variety $W$ is a Gysin coherent characteristic class.
\end{thm}

\begin{proof}
(Here, we take $\mathcal{X}$ to be the family of all inclusions $i \colon X \rightarrow W$ of compact subvarieties $X$ in smooth varieties $W$ such that $X$ is irreducible.)
By the properties of the cohomological Hirzebruch $L$-class, the class $\mathcal{L}^{\ast}(f) = L^{\ast}(\nu_{f}) \in H^{\ast}(M; \mathbb{Q})$ is normalized for all $f$.
Moreover, the highest nontrivial homogeneous component of the Goresky-MacPherson $L$-class $L_{\ast}(X) = L_{0}(X) + L_{1}(X) + \dots \in H_{\ast}(X; \mathbb{Q})$ is known to be $L_{2d}(X) = [X]_{X}$, where $d$ denotes the complex dimension of $X$.
Consequently, the highest nontrivial homogeneous component of $\mathcal{L}_{\ast}(i) = i_{\ast} L_{\ast}(X) \in H_{\ast}(W; \mathbb{Q})$ is the ambient fundamental class $i_{\ast}L_{2d}(X) = i_{\ast}[X]_{X} = [X]_{W}$.
We proceed to check the axioms of Gysin coherent characteristic classes for the pair $\mathcal{L}$.
As for axiom (\ref{axiom multiplicativity}), we have $L_{\ast}(X \times X') = L_{\ast}(X) \times L_{\ast}(X')$ in $H_{\ast}(X \times X'; \mathbb{Q})$ for all Witt spaces $X$ and $X'$ by a result of Woolf \cite[Proposition 5.16]{woolf}.
(Alternatively, multiplicativity of the Goreksy-MacPherson $L$-class under products follows from results in \cite{blm}, especially Section 11 there.)
Hence, for every $i \colon X \rightarrow W$ and $i' \colon X' \rightarrow W'$, the claim follows by applying $(i \times i')_{\ast}$ and using naturality of the cross product:
\begin{align*}
\mathcal{L}_{\ast}(i \times i') &= (i \times i')_{\ast}L_{\ast}(X \times X') = (i \times i')_{\ast}(L_{\ast}(X) \times L_{\ast}(X')) \\
&= i_{\ast}L_{\ast}(X) \times i_{\ast}'L_{\ast}(X') = \mathcal{L}_{\ast}(i) \times \mathcal{L}_{\ast}(i').
\end{align*}
Next, let us show that the pair $\mathcal{L}$ is compatible with ambient isomorphisms as stated in axiom (\ref{axiom isomorphism}).
As for $\mathcal{L}^{\ast}$, we consider $f \colon M \rightarrow W$ and $f' \colon M' \rightarrow W'$, and an isomorphism $W \stackrel{\cong}{\longrightarrow} W'$ that restricts to an isomorphism $\phi \colon M \stackrel{\cong}{\longrightarrow} M'$.
Then, we have $\phi^{\ast}\nu_{f'} = \nu_{f}$, and thus
$$
\phi^{\ast}\mathcal{L}^{\ast}(f') = \phi^{\ast}L^{\ast}(\nu_{f'}) = L^{\ast}(\phi^{\ast}\nu_{f'}) = L^{\ast}(\nu_{f}) = \mathcal{L}^{\ast}(f).
$$
As for $\mathcal{L}_{\ast}$, we consider $i \colon X \rightarrow W$ and $i' \colon X' \rightarrow W'$, and an isomorphism $\Phi \colon W \stackrel{\cong}{\longrightarrow} W'$ that restricts to an isomorphism $X \stackrel{\cong}{\longrightarrow} X'$.
Since $W'$ is smooth and quasiprojective, it follows that the compact subvariety $X' \subset W'$ can be Whitney stratified with only even-codimensional strata.
We equip $X \subset W$ with the Whitney stratification induced from that of $X' \subset W'$ by the isomorphism $\Phi$.
Let $\Phi_{0} \colon X \stackrel{\cong}{\longrightarrow} X'$ denote the restriction of $\Phi$.
As the isomorphism $\Phi_{0}$ is the restriction of an ambient diffeomorphism underlying $\Phi$, it follows directly from the construction of the Goresky-MacPherson $L$-class of Whitney stratified pseudomanifolds with only even-codimensional strata (see \cite[Section 5.3]{gmih1}) that $\Phi_{0 \ast}L_{\ast}(X) = L_{\ast}(X')$.
Hence, we obtain
$$
\Phi_{\ast}\mathcal{L}_{\ast}(i) = \Phi_{\ast}i_{\ast}L_{\ast}(X) = i_{\ast}'\Phi_{0\ast}L_{\ast}(X) = i_{\ast}'L_{\ast}(X') = \mathcal{L}_{\ast}(i').
$$
To verify axiom (\ref{axiom locality}), we consider $i \colon X \rightarrow W$ and $f \colon M \rightarrow W$ such that $X \subset M$.
Then, the inclusion $i^{M} := i| \colon X \rightarrow M$ satisfies $f \circ i^{M} = i$, and we obtain
$$
f_{\ast}\mathcal{L}_{\ast}(i^{M}) = f_{\ast}i^{M}_{\ast}L_{\ast}(X) = i_{\ast}L_{\ast}(X) = \mathcal{L}_{\ast}(i).
$$
Finally, to show axiom (\ref{axiom gysin}), let us call closed irreducible subvarieties $Z, Z' \subset W$ of a smooth variety $W$ $\mathcal{X}$-transverse if $Z$ and $Z'$ are simultaneously Whitney transverse and generically transverse in $W$.
This notion of $\mathcal{X}$-transversality has indeed all required properties.
(In fact, properness of $\mathcal{X}$-transverse intersections holds by \Cref{prop codimension of transverse intersection}.
Moreover, Kleiman's transversality theorem for the action of $GL_{n}(\mathbb{C})$ on the Grassmannians $G = G_{k}(\mathbb{C}^{n})$ holds for Whitney transversality by \Cref{thm kleiman for whitney transversality}, and for generic transversality by \Cref{thm.kleiman}.
Here, we also use that Zariski dense open subsets are also dense in the complex topology by \cite[Theorem 1, p. 58]{mumfordred}.
Finally, locality clearly holds for our notion of $\mathcal{X}$-transversality.)
Now, consider an inclusion $i \colon X \rightarrow W$ in $\mathcal{X}$ and an inclusion $f \colon M \rightarrow W$ such that $M$ is irreducible, and $M$ and $X$ are $\mathcal{X}$-transverse in $W$.
Let $j \colon Y \rightarrow M$ and $g \colon Y \rightarrow X$ denote the inclusions of the pure-dimensional compact subvariety $Y := M \cap X$.
Then, the inclusion $g \colon Y \hookrightarrow X$ is normally nonsingular, with topological normal bundle $\nu = j^{\ast} \nu_{f}$ given by the restriction of the normal bundle $\nu_{f}$ of $M$ in $W$ (see \Cref{theorem normally nonsingular inclusions induced by transverse intersections}).
According to the first author's Gysin restriction formula for the Goresky-MacPherson $L$-class (see \Cref{thm banagl topological gysin restriction}), we have
\begin{equation*}
g^{!}L_{\ast}(X) = L^{\ast}(\nu) \cap L_{\ast}(Y).
\end{equation*}
Using that $f^{!}i_{\ast} = j_{\ast}g^{!}$ by \Cref{proposition gysin compatible with inclusions}, as well as $L^{\ast}(\nu) = L^{\ast}(j^{\ast} \nu_{f}) = j^{\ast} L^{\ast}(\nu_{f})$, we conclude that
\begin{align*}
f^{!} \mathcal{L}_{\ast}(i) &= f^{!} i_{\ast}L_{\ast}(X) = j_{\ast}g^{!} L_{\ast}(X) = j_{\ast}(L^{\ast}(\nu) \cap L_{\ast}(Y)) \\
&= j_{\ast}(j^{\ast} L^{\ast}(\nu_{f}) \cap L_{\ast}(Y)) = L^{\ast}(\nu_{f}) \cap j_{\ast}L_{\ast}(Y) = \mathcal{L}^{\ast}(f) \cap \mathcal{L}_{\ast}(j).
\end{align*}

This completes the proof of \Cref{proposition l class is l type characteristic class}.
\end{proof}

\section{An Example: The $L$-Class of $X_{3, 2, 1}$}
\label{l class of x 3 2 1}

We set $X = X_{3, 2, 1}$.
Note that $X$ is a singular Schubert variety of real dimension $12$ that does not satisfy global Poincar\'{e} duality over the rationals since, for example,
$$
H_{8}(X; \mathbb{Q}) = \mathbb{Q}[X_{3, 1}]_{X} \oplus \mathbb{Q} [X_{2, 2}]_{X} \oplus \mathbb{Q} [X_{2, 1, 1}]_{X},
$$
is a $3$-dimensional rational vector space, whereas
$$
H_{4}(X; \mathbb{Q}) = \mathbb{Q}[X_{2}]_{X} \oplus \mathbb{Q} [X_{1, 1}]_{X},
$$
has dimension $2$.
Let us compute the total Goresky-MacPherson $L$-class
$$
L_{\ast}(X_{3, 2, 1}) = L_{12}(X) + L_{8}(X) + L_{4}(X) + L_{0}(X) \in H_{\ast}(X_{3, 2, 1}; \mathbb{Q})
$$
with $L_{j}(X) \in H_{j}(X; \mathbb{Q})$.
The highest class $L_{12}(X) = [X]_{X}$ is the fundamental class, and $L_{0}(X) = \sigma(X) \cdot [\operatorname{pt}]_{X}$ is determined by the signature.
In the following, we are concerned with the computation of the $L$-classes $L_{8}(X) \in H_{8}(X; \mathbb{Q})$ and $L_{4}(X) \in H_{4}(X; \mathbb{Q})$, which can be uniquely written as rational linear combinations
\begin{equation}\label{equation codimension 4 example}
L_{8}(X) = \lambda_{3, 1} [X_{3, 1}]_{X} + \lambda_{2, 2} [X_{2, 2}]_{X} + \lambda_{2, 1, 1} [X_{2, 1, 1}]_{X}, \qquad  \lambda_{3, 1},  \lambda_{2, 2},  \lambda_{2, 1, 1} \in \mathbb{Q},
\end{equation}
and
\begin{equation}\label{equation codimension 8 example}
L_{4}(X) = \lambda_{2} [X_{2}]_{X} + \lambda_{1, 1} [X_{1, 1}]_{X}, \qquad \lambda_{2}, \lambda_{1, 1} \in \mathbb{Q},
\end{equation}
respectively.
Note that our computation of $L_{4}(X)$ goes beyond the scope of the computations in \cite[Section 4]{banagllgysin}, which are limited to $L$-classes in real codimension $4$.
For our purpose, we will consider $X = X_{3, 2, 1}(F_{\ast})$ as a Schubert subvariety of the Grassmannian $G = G_{3}(\mathbb{C}^{6})$ with respect to some complete flag $F_{\ast}$ on $\mathbb{C}^{6}$.

Similarly to the normally nonsingular expansion of $L_{6}(X_{3, 2})$ in \cite[Section 4]{banagllgysin}, our method is to produce equations for the unknown coefficients in (\ref{equation codimension 4 example}) and (\ref{equation codimension 8 example}) as follows.
By intersecting $X \subset G$ transversely with a nonsingular subvariety $M \subset G$ with topological normal bundle $\nu_{M}$, we obtain an oriented normally nonsingular inclusion
\begin{equation}\label{equation example inclusion map}
g \colon Y = M \cap X \hookrightarrow X
\end{equation}
with normal bundle $\nu_{Y} = \nu_{M}|_{Y}$.
We consider the associated Gysin homomorphism
\begin{equation}\label{equation example gysin map}
g^{!} \colon H_{\ast}(X; \mathbb{Q}) \rightarrow H_{\ast-2c}(Y; \mathbb{Q}),
\end{equation}
where $c$ denotes the complex codimension of $M$ in $G$.
Then, we compute the values of $g^{!}$ on the Schubert generators of $H_{\ast}(X; \mathbb{Q})$ by using intersection theory of Schubert cycles.
Provided that $L_{\ast}(Y)$ is known, the Gysin restriction formula
\begin{equation}\label{equation example gysin formula}
g^{!}L_{\ast}(X) = L^{\ast}(\nu_{Y}) \cap L_{\ast}(Y)
\end{equation}
of \Cref{thm banagl topological gysin restriction} then yields equations in the unknown coefficients in (\ref{equation codimension 4 example}) and (\ref{equation codimension 8 example}).
Finally, for appropriate choices of nonsingular subvarieties $M \subset G$, we are able to derive normally nonsingular expansions for these coefficients.

Compared to the calculations in \cite[Section 4]{banagllgysin}, our method is more general in the following two ways.
First, we will choose $M$ to be the regular part of a possibly singular Schubert subvariety of $G$, thus allowing noncompact $M$ for which the intersection $Y = M \cap X$ is still compact, so that $g^{!}$ and $L_{\ast}(Y)$ are still defined.
Second, the Segre product of Grassmannians in an ambient Grassmannian is a new ingredient that allows us to identify $Y$ in some cases with a product of Schubert varieties with known $L$-classes.
\\

Let us compute $\lambda_{3, 1}$.
Our computation is similar to that of the coefficient $\mu$ of $L_{6}(X_{3, 2})$ in \cite[Section 4]{banagllgysin}.
We choose a direct sum decomposition $\mathbb{C}^{6} = V' \oplus V''$ with $\operatorname{dim}_{\mathbb{C}}(V') = 5$ and $\operatorname{dim}_{\mathbb{C}}(V'') = 1$.
Let $(E_{\ast}', F_{\ast}')$ be a pair of transverse flags on $V'$, and let $(E_{\ast}'', F_{\ast}'')$ be a pair of transverse flags on $V''$.
Then, $(E_{\ast}, F_{\ast}) = (E_{\ast}' \oplus E_{\ast}'', F_{\ast}'' \oplus F_{\ast}')$ is a pair of transverse flags on $\mathbb{C}^{6}$.
Hence, by \Cref{proposition transverse flags imply whitney transverse}, $M = X_{3, 3}(E_{\ast})$ and $X = X_{3, 2, 1}(F_{\ast})$ are Whitney transverse in $G$ with respect to suitable Whitney stratifications.
By a result of Lakshmibai-Weyman (see \Cref{thm.lakwey}), the Schubert subvariety $M = X_{3, 3}(E_{\ast}) \subset G$ is nonsingular.
In order to compute the transverse intersection $Y = M \cap X$, we consider the Segre product
$$
S \colon G_{2}(V') \times G_{1}(V'') \hookrightarrow G_{3}(\mathbb{C}^{6}).
$$
Writing $(3, 3, 0) \in \mathcal{P}(3, 3)$ as $(3, 3, 0) = (0) \sqcup (3, 3)$ with $(0) \in \mathcal{P}(0, 1)$ and $(3, 3) \in \mathcal{P}(3, 2)$, and $(3, 2, 1) \in \mathcal{P}(3, 3)$ as $(3, 2, 1) = (2, 1) \sqcup (0)$ with $(2, 1) \in \mathcal{P}(3, 2)$ and $(0) \in \mathcal{P}(0, 1)$, we conclude from \Cref{proposition segre product of schubert varieties} that
\begin{align*}
Y &= M \cap X \\
&= X_{3, 3}(E_{\ast}) \cap X_{3, 2, 1}(F_{\ast}) \\
&= X_{(0) \sqcup (3, 3)}(E_{\ast}' \oplus E_{\ast}'') \cap X_{(2, 1) \sqcup (0)}(F_{\ast}'' \oplus F_{\ast}') \\
&= S((X_{0}(E_{\ast}') \cap X_{0}(F_{\ast}')) \times (X_{3, 3}(E_{\ast}'') \cap X_{2, 1}(F_{\ast}''))) \\
&= S((X_{0}(E_{\ast}') \cap X_{0}(F_{\ast}')) \times (X_{2, 1}(F_{\ast}'') \cap X_{3, 3}(E_{\ast}''))) \\
&= X_{(0) \sqcup (2, 1)}(E_{\ast}' \oplus F_{\ast}'') \cap X_{(3, 3) \sqcup (0)}(E_{\ast}'' \oplus F_{\ast}') \\
&= X_{2, 1}(E_{\ast}' \oplus F_{\ast}'') \cap X_{3, 3, 3}(E_{\ast}'' \oplus F_{\ast}') \\
&= X_{2, 1}(E_{\ast}' \oplus F_{\ast}'').
\end{align*}
We set $X_{2, 1}' := X_{2, 1}(E_{\ast}' \oplus F_{\ast}'')$ and $X_{1}' := X_{1}(E_{\ast}' \oplus F_{\ast}'')$.
The Gysin homomorphism $g^{!} \colon H_{8}(X; \mathbb{Q}) \rightarrow H_{2}(Y; \mathbb{Q})$ associated to the normally nonsingular inclusion $g \colon Y \hookrightarrow X$ is of the form
$$
g^{!} \colon \mathbb{Q}[X_{3, 1}]_{X} \oplus \mathbb{Q} [X_{2, 2}]_{X} \oplus \mathbb{Q} [X_{2, 1, 1}]_{X} \rightarrow \mathbb{Q}[X_{1}']_{Y}.
$$
Since $(E_{\ast}, F_{\ast})$ is a pair of transverse flags on $\mathbb{C}^{6}$, it follows from \Cref{proposition transverse flags imply whitney transverse} that $M \subset G$ is Whitney transverse to $X_{3, 1}$, $X_{2, 2}$, and $X_{2, 1, 1}$ (all defined with respect to the flag $F_{\ast}$) with respect to suitable Whitney stratifications.
Hence, using \Cref{lem.gysinfund}, we have $g^{!}[X_{3, 1}]_{X} = [M \cap X_{3, 1}]_{Y}$, $g^{!}[X_{2, 2}]_{X} = [M \cap X_{2, 2}]_{Y}$, and $g^{!}[X_{2, 1, 1}]_{X} = [M \cap X_{2, 1, 1}]_{Y}$.
By a computation similar to the above, \Cref{proposition segre product of schubert varieties} implies that $M \cap X_{3, 1} = X_{1}'$.
Moreover, we have $M \cap X_{2, 2} = \varnothing$ and $M \cap X_{2, 1, 1} = \varnothing$ by \Cref{proposition empty intersection of schubert varieties}.
All in all, we obtain
$$
g^{!}L_{8}(X) = \lambda_{3, 1} \cdot [X_{1}']_{Y}.
$$
On the other hand, (\ref{equation example gysin formula}) yields
\begin{align*}
g^{!}L_{8}(X) &= ((1 + L^{1}(\nu_{Y}) + L^{2}(\nu_{Y}) + \dots) \cap (L_{6}(Y) + L_{2}(Y)))_{2} \\
&= L_{2}(Y) + L^{1}(\nu_{Y}) \cap [Y]_{Y}.
\end{align*}
The $L$-class $L_{\ast}(Y) = L_{\ast}(X_{2, 1}') = L_{\ast}(X_{2, 1})$ was computed in \cite[Section 4]{banagllgysin} to be
\begin{equation}\label{l class of x 2 1}
L_{\ast}(X_{2, 1}) = L_{6}(X_{2, 1}) + L_{2}(X_{2, 1}) = [X_{2, 1}]_{X_{2, 1}} + \frac{2}{3} [X_{1}]_{X_{2, 1}} \in H_{\ast}(X_{2, 1}; \mathbb{Q}).
\end{equation}
Altogether,
$$
\lambda_{3, 1} \cdot [X_{1}']_{Y} = \frac{2}{3} [X_{1}']_{Y}  + L^{1}(\nu_{Y}) \cap [Y]_{Y} \in H_{2}(Y; \mathbb{Q}) = \mathbb{Q}[X_{1}']_{Y}, \qquad Y = X_{2, 1}'.
$$
The coefficient $\lambda_{2, 1, 1}$ can be computed in a similar way, by taking $M = X_{3, 3}(E_{\ast})$ for a suitable flag $E_{\ast}$ on $\mathbb{C}^{6}$.
It turns out that $\lambda_{3, 1} = \lambda_{2, 1, 1}$.
\\

Let us compute $\lambda_{2, 2}$.
We choose a direct sum decomposition $\mathbb{C}^{6} = V' \oplus V''$ with $\operatorname{dim}_{\mathbb{C}}(V') = 3$ and $\operatorname{dim}_{\mathbb{C}}(V'') = 3$.
Let $(E_{\ast}', F_{\ast}')$ be a pair of transverse flags on $V'$, and let $(E_{\ast}'', F_{\ast}'')$ be a pair of transverse flags on $V''$.
Then, $(E_{\ast}, F_{\ast}) = (E_{\ast}' \oplus E_{\ast}'', F_{\ast}'' \oplus F_{\ast}')$ is a pair of transverse flags on $\mathbb{C}^{6}$.
Hence, by \Cref{proposition transverse flags imply whitney transverse}, $X_{3, 1, 1}(E_{\ast})$ and $X = X_{3, 2, 1}(F_{\ast})$ are Whitney transverse in $G$ with respect to suitable Whitney stratifications.
By a result of Lakshmibai-Weyman (see \Cref{thm.lakwey}), the singular locus of the Schubert subvariety $X_{3, 1, 1}(E_{\ast}) \subset G$ is $X_{0}(E_{\ast}) \subset G$, a single point.
Let $M$ denote the regular part of $X_{3, 1, 1}(E_{\ast})$, which is an open subvariety of $X_{3, 1, 1}(E_{\ast})$.
Note that $M \cap X = X_{3, 1, 1}(E_{\ast}) \cap X$ because $X_{0}(E_{\ast}) \cap X = \varnothing$ by \Cref{proposition empty intersection of schubert varieties}.
In order to compute the transverse intersection $Y = M \cap X$, we consider the Segre product
$$
S \colon G_{1}(V') \times G_{2}(V'') \hookrightarrow G_{3}(\mathbb{C}^{6}).
$$
Writing $(3, 1, 1) \in \mathcal{P}(3, 3)$ as $(3, 1, 1) = (1, 1) \sqcup (2)$ with $(1, 1) \in \mathcal{P}(1, 2)$ and $(2) \in \mathcal{P}(2, 1)$, and $(3, 2, 1) \in \mathcal{P}(3, 3)$ as $(3, 2, 1) = (1) \sqcup (1, 0)$ with $(1) \in \mathcal{P}(2, 1)$ and $(1, 0) \in \mathcal{P}(1, 2)$, we conclude from \Cref{proposition segre product of schubert varieties} that
\begin{align*}
Y &= M \cap X \\
&= X_{3, 1, 1}(E_{\ast}) \cap X_{3, 2, 1}(F_{\ast}) \\
&= X_{(1, 1) \sqcup (2)}(E_{\ast}' \oplus E_{\ast}'') \cap X_{(1) \sqcup (1, 0)}(F_{\ast}'' \oplus F_{\ast}') \\
&= S((X_{1,1}(E_{\ast}') \cap X_{1}(F_{\ast}')) \times (X_{2}(E_{\ast}'') \times X_{1}(F_{\ast}''))) \\
&= S(X_{1}(F_{\ast}') \times X_{1}(F_{\ast}'')) \\
&\cong \mathbb{P}^{1} \times \mathbb{P}^{1}.
\end{align*}
The Gysin homomorphism $g^{!} \colon H_{8}(X; \mathbb{Q}) \rightarrow H_{0}(Y; \mathbb{Q})$ associated to the normally nonsingular inclusion $g \colon Y \hookrightarrow X$ is of the form
$$
g^{!} \colon \mathbb{Q}[X_{3, 1}]_{X} \oplus \mathbb{Q} [X_{2, 2}]_{X} \oplus \mathbb{Q} [X_{2, 1, 1}]_{X} \rightarrow \mathbb{Q}[\operatorname{pt}]_{Y}.
$$
Since $(E_{\ast}, F_{\ast})$ is a pair of transverse flags on $\mathbb{C}^{6}$, it follows from \Cref{proposition transverse flags imply whitney transverse} and \Cref{lemma whitney restratification manifold} that $M$ is transverse to all Whitney strata of $X_{3, 1}$, $X_{2, 2}$, and $X_{2, 1, 1}$ (all defined with respect to the flag $F_{\ast}$) in $G$ with respect to suitable Whitney stratifications.
Hence, using \Cref{lem.gysinfund}, we have $g^{!}[X_{3, 1}]_{X} = [M \cap X_{3, 1}]_{Y}$, $g^{!}[X_{2, 2}]_{X} = [M \cap X_{2, 2}]_{Y}$, and $g^{!}[X_{2, 1, 1}]_{X} = [M \cap X_{2, 1, 1}]_{Y}$.
By \Cref{proposition empty intersection of schubert varieties}, $M \cap X_{2, 2} = \operatorname{pt}$ is a single point, whereas $M \cap X_{3, 1} = \varnothing$ and $M \cap X_{2, 1, 1} = \varnothing$.
All in all, we obtain
$$
g^{!}L_{8}(X) = \lambda_{2, 2} \cdot [\operatorname{pt}]_{Y}.
$$
On the other hand, (\ref{equation example gysin formula}) yields
\begin{align*}
g^{!}L_{8}(X) &= ((1 + L^{1}(\nu_{Y}) + L^{2}(\nu_{Y}) + \dots) \cap (L_{4}(Y) + L_{0}(Y)))_{0} \\
&= L_{0}(Y) + L^{1}(\nu_{Y}) \cap [Y]_{Y}.
\end{align*}
Note that $L_{0}(Y) = \sigma(Y) \cdot [\operatorname{pt}]_{Y} = 0$ since the signature of $Y \cong \mathbb{P}^{1} \times \mathbb{P}^{1}$ vanishes.
Altogether, we obtain
$$
\lambda_{2, 2} = \langle L^{1}(\nu_{Y}), [Y]_{Y}\rangle = \int_{Y} L^{1}(\nu_{M}), \qquad Y = S(X_{1}(F_{\ast}') \times X_{1}(F_{\ast}'')) \subset M = X_{3, 1, 1}(E_{\ast}).
$$
\\

Let us compute $\lambda_{2}$.
We choose a direct sum decomposition $\mathbb{C}^{6} = V' \oplus V''$ with $\operatorname{dim}_{\mathbb{C}}(V') = 2$ and $\operatorname{dim}_{\mathbb{C}}(V'') = 4$.
Let $(E_{\ast}', F_{\ast}')$ be a pair of transverse flags on $V'$, and let $(E_{\ast}'', F_{\ast}'')$ be a pair of transverse flags on $V''$.
Then, $(E_{\ast}, F_{\ast}) = (E_{\ast}' \oplus E_{\ast}'', F_{\ast}'' \oplus F_{\ast}')$ is a pair of transverse flags on $\mathbb{C}^{6}$.
Hence, by \Cref{proposition transverse flags imply whitney transverse}, $X_{3, 3, 1}(E_{\ast})$ and $X = X_{3, 2, 1}(F_{\ast})$ are Whitney transverse in $G$ with respect to suitable Whitney stratifications.
By a result of Lakshmibai-Weyman (see \Cref{thm.lakwey}), the singular locus of the Schubert subvariety $X_{3, 3, 1}(E_{\ast}) \subset G$ is $X_{3}(E_{\ast}) \subset G$.
Let $M$ denote the regular part of $X_{3, 3, 1}(E_{\ast})$, which is an open subvariety of $X_{3, 3, 1}(E_{\ast})$.
Note that $M \cap X = X_{3, 3, 1}(E_{\ast}) \cap X$ because $X_{3}(E_{\ast}) \cap X = \varnothing$ by \Cref{proposition empty intersection of schubert varieties}.
In order to compute the transverse intersection $Y = M \cap X$, we consider the Segre product
$$
S \colon G_{1}(V') \times G_{2}(V'') \hookrightarrow G_{3}(\mathbb{C}^{6}).
$$
Writing $(3, 3, 1) \in \mathcal{P}(3, 3)$ as $(3, 3, 1) = (1) \sqcup (2, 2)$ with $(1) \in \mathcal{P}(1, 1)$ and $(2, 2) \in \mathcal{P}(2, 2)$, and $(3, 2, 1) \in \mathcal{P}(3, 3)$ as $(3, 2, 1) = (2, 1) \sqcup (1)$ with $(2, 1) \in \mathcal{P}(2, 2)$ and $(1) \in \mathcal{P}(1, 1)$, we conclude from \Cref{proposition segre product of schubert varieties} that
\begin{align*}
Y &= M \cap X \\
&= X_{3, 3, 1}(E_{\ast}) \cap X_{3, 2, 1}(F_{\ast}) \\
&= X_{(1) \sqcup (2, 2)}(E_{\ast}' \oplus E_{\ast}'') \cap X_{(2, 1) \sqcup (1)}(F_{\ast}'' \oplus F_{\ast}') \\
&= S((X_{1}(E_{\ast}') \cap X_{1}(F_{\ast}')) \times (X_{2, 2}(E_{\ast}'') \times X_{2, 1}(F_{\ast}''))) \\
&= S(X_{1}(F_{\ast}') \times X_{2, 1}(F_{\ast}'')).
\end{align*}
The Gysin homomorphism $g^{!} \colon H_{4}(X; \mathbb{Q}) \rightarrow H_{0}(Y; \mathbb{Q})$ associated to the normally nonsingular inclusion $g \colon Y \hookrightarrow X$ is of the form
$$
g^{!} \colon \mathbb{Q}[X_{2}]_{X} \oplus \mathbb{Q} [X_{1, 1}]_{X} \rightarrow \mathbb{Q}[\operatorname{pt}]_{Y}.
$$
Since $(E_{\ast}, F_{\ast})$ is a pair of transverse flags on $\mathbb{C}^{6}$, it follows from \Cref{proposition transverse flags imply whitney transverse} and \Cref{lemma whitney restratification manifold} that $M$ is transverse to all Whitney strata of $X_{2}$ and $X_{1, 1}$ (all defined with respect to the flag $F_{\ast}$) in $G$ with respect to suitable Whitney stratifications.
Hence, using \Cref{lem.gysinfund}, we have $g^{!}[X_{2}]_{X} = [M \cap X_{2}]_{Y}$ and $g^{!}[X_{1, 1}]_{X} = [M \cap X_{1, 1}]_{Y}$.
By \Cref{proposition empty intersection of schubert varieties}, $M \cap X_{2} = \operatorname{pt}$ is a single point, whereas $M \cap X_{1, 1} = \varnothing$.
All in all, we obtain
$$
g^{!}L_{4}(X) = \lambda_{2} \cdot [\operatorname{pt}]_{Y}.
$$
On the other hand, (\ref{equation example gysin formula}) yields
\begin{align*}
g^{!}L_{4}(X) &= ((1 + L^{1}(\nu_{Y}) + L^{2}(\nu_{Y}) + \dots) \cap (L_{8}(Y) + L_{4}(Y) + L_{0}(Y)))_{0} \\
&= L_{0}(Y) + L^{1}(\nu_{Y}) \cap L_{4}(Y) + L^{2}(\nu_{Y}) \cap [Y]_{Y}.
\end{align*}
Note that $L_{0}(Y) = \sigma(Y) \cdot [\operatorname{pt}]_{Y} = 0$ since the signature of $Y \cong \mathbb{P}^{1} \times X_{2, 1}$ vanishes.
Next, we compute
$$
L_{4}(Y) \in H_{4}(Y; \mathbb{Q}) = \mathbb{Q}[Y_{1 \times 1}]_{Y} \oplus \mathbb{Q}[Y_{0 \times 2}]_{Y} \oplus \mathbb{Q}[Y_{0 \times (1, 1)}]_{Y},
$$
where $Y_{1 \times 1} := S(X_{1}(F_{\ast}') \times X_{1}(F_{\ast}''))$, $Y_{0 \times 2} := S(X_{0}(F_{\ast}') \times X_{2}(F_{\ast}''))$, and $Y_{0 \times (1, 1)} := S(X_{0}(F_{\ast}') \times X_{1, 1}(F_{\ast}''))$.
The Segre product $S$ restricts to an isomorphism $G_{1}(V') \times G_{2}(V'') \cong S(G_{1}(V') \times G_{2}(V''))$, which in turn restricts to an isomorphism $\varphi \colon X_{1}(F_{\ast}') \times X_{2, 1}(F_{\ast}'') \cong Y$.
Hence, writing $X_{1}' = X_{1}(F_{\ast}')$, $X_{2, 1}'' = X_{2, 1}(F_{\ast}'')$, and $X_{1}'' = X_{1}(F_{\ast}'')$, we have
\begin{align*}
L_{4}(Y) &= \varphi_{\ast}L_{4}(X_{1}' \times X_{2, 1}'') \\
&= \varphi_{\ast}(L_{\ast}(X_{1}') \times L_{\ast}(X_{2, 1}''))_{4} \\
&= \varphi_{\ast}(L_{2}(X_{1}') \times L_{2}(X_{2, 1}'')) \\
&\stackrel{(\ref{l class of x 2 1})}{=} \frac{2}{3} \cdot \varphi_{\ast}([X_{1}']_{X_{1}'} \times [X_{1}'']_{X_{2, 1}''}) \\
&= \frac{2}{3} \cdot \varphi_{\ast}([X_{1}' \times X_{1}'']_{X_{1}' \times X_{2, 1}''}) \\
&= \frac{2}{3} \cdot [S(X_{1}' \times X_{1}'')]_{Y} \\
&=  \frac{2}{3} \cdot [Y_{1 \times 1}]_{Y}.
\end{align*}
Altogether, we obtain
\begin{align*}
\lambda_{2} = \frac{2}{3} \cdot \langle L^{1}(\nu_{Y}), [Y_{1 \times 1}]_{Y}\rangle + \langle L^{2}(\nu_{Y}), [Y]_{Y}\rangle = \frac{2}{3} \cdot \int_{Y_{1 \times 1}} L^{1}(\nu_{M}) + \int_{Y} L^{2}(\nu_{M}), \\
Y_{1 \times 1} = S(X_{1}(F_{\ast}') \times X_{1}(F_{\ast}'')) \subset Y = S(X_{1}(F_{\ast}') \times X_{2, 1}(F_{\ast}'')) \subset M = X_{3, 3, 1}(E_{\ast}).
\end{align*}
The coefficient $\lambda_{1, 1}$ can be computed in a similar way, by taking $M = X_{3, 2, 2}(E_{\ast})$ for a suitable flag $E_{\ast}$ on $\mathbb{C}^{6}$.

\end{document}